\documentclass[a4paper,16pt]{article}
\usepackage[utf8]{inputenc}
\usepackage{amssymb,amsmath,mathrsfs}
\usepackage{amsthm}
\usepackage[colorlinks,
            linkcolor=blue,
            anchorcolor=red,
            citecolor=green
            ]{hyperref}
\numberwithin{equation}{section}
\usepackage{graphicx}
\usepackage{amsmath}
\newtheorem{theorem}{Theorem}[section]

\theoremstyle{plain}
\newtheorem{thm}[theorem]{Theorem } 
\newtheorem{defi}[theorem]{Definition }
\newtheorem{prop}[theorem]{Proposition }
\newtheorem{cor}[theorem]{Corollary }

\newtheorem{lem}[theorem]{Lemma }

\newtheorem{rem}[theorem] {Remark}

\newtheorem{thmprincipal}{Theorem } 
\usepackage{tikz}
\usetikzlibrary{matrix}

\begin{document}

\title{Complete integrability of the Benjamin--Ono equation on the multi-soliton manifolds}
\author{Ruoci Sun\footnote{Laboratoire de Math\'ematiques d’Orsay, Univ. Paris-Sud \uppercase\expandafter{\romannumeral11}, CNRS, Universit\'e Paris-Saclay, F-91405 Orsay, France (ruoci.sun.16@normalesup.org).} \footnote{The author is partially supported by the grant "ANA\'E"  ANR-13-BS01-0010-03 of the 'Agence Nationale de la Recherche'. This research is carried out during the author's PhD studies, financed by the PhD fellowship of \'Ecole Doctorale de Math\'ematique Hadamard.}}

\maketitle

\noindent $\mathbf{Abstract}$ \quad This paper is dedicated to proving the complete integrability of the Benjamin--Ono (BO) equation on the line when restricted to every $N$-soliton manifold, denoted by $\mathcal{U}_N$. We construct  generalized action--angle coordinates which establish a real analytic symplectomorphism from $\mathcal{U}_N$ onto some open convex subset of $\mathbb{R}^{2N}$ and allow to solve the equation by quadrature for any such initial datum.  As a consequence,  $\mathcal{U}_N$ is the universal covering of the manifold of $N$-gap potentials for the BO equation on the torus as described by   G\'erard--Kappeler $[\ref{Gerard kappeler Benjamin ono birkhoff coordinates}]$.  The global well-posedness of the BO equation in $\mathcal{U}_N$ is given by a polynomial characterization and a spectral characterization of the manifold $\mathcal{U}_N$. Besides the spectral analysis of the Lax operator of the BO equation and the shift semigroup acting on some Hardy spaces, the construction of such coordinates also relies on the use of a generating functional, which encodes the entire BO hierarchy. \\

\noindent $\mathbf{Keywords}$ \quad Benjamin--Ono equation, generalized action--angle coordinates, Lax pair, inverse spectral transform, multi-solitons, universal covering manifold  \\

\noindent \textbf{Throughout this paper, the main results of each section are stated at the beginning. Their proofs are left inside the corresponding subsections.} 
\bigskip
\bigskip
\bigskip
\bigskip

\begin{center}
$Acknowledgments$
\end{center}
The author would like to express his sincere gratitude towards his PhD advisor Prof. Patrick G\'erard for introducing this problem, for his deep insight, generous advice and continuous encouragement. He also would like to thank warmly Dr. Yang Cao for  introducing fiber product method. \\

\clearpage
\tableofcontents

\section{Introduction}
The Benjamin--Ono (BO) equation on the line reads as
\begin{equation}\label{Benjamin Ono equation on the line}
\partial_t u =  \mathrm{H} \partial_x^2 u - \partial_x (u^2), \qquad (t,x) \in \mathbb{R}\times \mathbb{R},
\end{equation}where $u$ is real-valued and $ \mathrm{H} =-i\mathrm{sign}(\mathrm{D}) : L^2(\mathbb{R}) \to  L^2(\mathbb{R})$ denotes the Hilbert transform, $\mathrm{D}=-i \partial_x$,   
\begin{equation}\label{Hilbert transform Fourier multiplier def}
\widehat{\mathrm{H}f}(\xi) =  -i \mathrm{sign}(\xi)\hat{f}(\xi) ,  \qquad \forall f \in  L^2(\mathbb{R}).
\end{equation}$\mathrm{sign}(\pm \xi)=\pm 1$, for all $\xi > 0$ and $\mathrm{sign}(0)=0$, $\hat{f} \in L^2(\mathbb{R})$ denotes the Fourier--Plancherel transform of $f \in L^2(\mathbb{R})$. We adopt the convention $L^p(\mathbb{R})=L^p(\mathbb{R}, \mathbb{C})$.  Its $\mathbb{R}$-subspace consisting of  all real-valued $L^p$-functions is specially emphasized as $L^p(\mathbb{R}, \mathbb{R})$ throughout this paper. Equipped with the inner product $(f, g) \in L^2(\mathbb{R}) \times L^2(\mathbb{R}) \mapsto \langle f, g \rangle_{L^2} = \int_{\mathbb{R}}f(x) \overline{g(x)}\mathrm{d}x \in \mathbb{C}$, $L^2(\mathbb{R})$ is a $\mathbb{C}$-Hilbert space. \\

\noindent Derived by Benjamin $[\ref{Benjamin Internal waves of permanent form in fluids of great depth}]$ and Ono $[\ref{Ono Algebraic solitary waves in stratified fluids}]$, this equation describes the evolution of weakly nonlinear internal long waves in a two-layer fluid. The BO equation is globally well-posed in every Sobolev spaces $H^s(\mathbb{R}, \mathbb{R})$, $s\geq 0$. (see Tao $[\ref{Tao GWP of BO eq R}]$ for $s\geq 1$, Burq--Planchon $[\ref{Burq Planchon  BO GWP s bigger than 0.25}]$ for $s>\frac{1}{4}$, Ionescu--Kenig $[\ref{Ionescu Kenig GWP of BO low regularity}]$, Molinet--Pilod $[\ref{Molinet Pilod L2 GWP of BOeq}]$ and Ifrim--Tataru $[\ref{Ifrim Tataru Well-posedness and dispersive decay of BO eq}]$ for $s\geq 0$, etc.) Recall the scaling and translation invariances of equation $(\ref{Benjamin Ono equation on the line})$: if $u=u(t,x)$ is a solution, so is $u_{c, y} : (t,x) \mapsto c u(c^2 t, c(x-y))$. A smooth solution $u=u(t,x)$ is called a solitary wave of $(\ref{Benjamin Ono equation on the line})$ if there exists $\mathcal{R} \in C^{\infty}(\mathbb{R})$ solving the following non local elliptic equation
\begin{equation}\label{Elliptic equation of soliton R}
\mathrm{H}\mathcal{R}' + \mathcal{R} - \mathcal{R}^2=0, \qquad \mathcal{R}(x) >0   
\end{equation}and $u(t,x)=\mathcal{R}_c(x-y-ct)$, where $\mathcal{R}_c(x)=c \mathcal{R}(cx)$, for some $c>0$ and $y \in \mathbb{R}$. The unique (up to translation) solution of equation $(\ref{Elliptic equation of soliton R})$ is given by the following formula
\begin{equation}\label{Explicit formula for soliton R}
\mathcal{R}(x)= \frac{2}{1+x^2}, \qquad \forall x \in \mathbb{R},
\end{equation}in Benjamin $[\ref{Benjamin Internal waves of permanent form in fluids of great depth}]$ and Amick--Toland $[\ref{Amick Toland B o equation traveling wave classification}]$ for the uniqueness statement. Inspired from the complete classification of solitary waves of the BO equation, we introduce the main object of this paper.

\begin{defi}\label{definition of The N soliton in introduction}
A function of the form $u(x)=\sum_{j=1}^N \mathcal{R}_{c_j}(x-x_j)$ is called an $N$-soliton, for some positive integer $N \in \mathbb{N}_+ :=\mathbb{Z}\bigcap (0,+\infty)$, where $c_j>0$ and $x_j \in \mathbb{R}$, for every $j=1,2, \cdots, N$. Let $\mathcal{U}_N \subset L^{2}(\mathbb{R},\mathbb{R})$ denote the subset consisting of all the $N$-solitons.
\end{defi}

\noindent In the point of view of topology and differential manifolds,  the subset $\mathcal{U}_N$ is a simply connected, real analytic, embedded submanifold of the $\mathbb{R}$-Hilbert space $L^2(\mathbb{R}, \mathbb{R})$. It has real dimension $2N$. The tangent space to $\mathcal{U}_N$ at an arbitrary $N$-soliton is included in an auxiliary space 
\begin{equation}\label{Auxiliary space containing all tangent spaces}
\mathcal{T}:=\{h \in L^2(\mathbb{R}, (1+x^2)\mathrm{d}x) : h(\mathbb{R})\subset \mathbb{R}, \quad \int_{\mathbb{R}}h=0\},
\end{equation}in which a $2$-covector $\boldsymbol{\omega} \in \boldsymbol{\Lambda}^2(\mathcal{T}^*)$ is well defined by $\boldsymbol{\omega} (h_1, h_2) =\frac{i}{2\pi} \int_{\mathbb{R}} \frac{\hat{h}_1(\xi) \overline{\hat{h}_2(\xi)}}{\xi}\mathrm{d}\xi$, for every $h_1, h_2 \in \mathcal{T}$, by  Hardy's inequality. We define a translation-invariant $2$-form $\omega : u \in \mathcal{U}_N \mapsto \boldsymbol{\omega}\in \boldsymbol{\Lambda}^2(\mathcal{T}^*)$, endowed with which $\mathcal{U}_N$ is a symplectic manifold. The tangent space to $\mathcal{U}_N$ at $u\in \mathcal{U}_N$ is denoted by $\mathcal{T}_u(\mathcal{U}_N)$. For every smooth function  $f :\mathcal{U}_N \to \mathbb{R}$, its Hamiltonian vector field  $X_f \in \mathfrak{X}(\mathcal{U}_N)$ is given by  
\begin{equation*}
X_f : u \in \mathcal{U}_N \mapsto \partial_x \nabla_u f(u)  \in\mathcal{T}_u(\mathcal{U}_N),
\end{equation*}where $\nabla_u f(u)$ denotes the Fr\'echet derivative of $f$, i.e. $\mathrm{d}f(u)(h)=\langle h, \nabla_u f(u)\rangle_{L^2}$, for every $h \in \mathcal{T}_u(\mathcal{U}_N)$. The Poisson bracket of $f$ and another smooth function $g :\mathcal{U}_N \to \mathbb{R}$ is defined by
\begin{equation*}
 \{f,g\} : u \in \mathcal{U}_N \mapsto \omega_u(X_f(u), X_g(u))=\langle \partial_x \nabla_u f(u), \nabla_u g(u) \rangle_{L^2}\in \mathbb{R}.
\end{equation*}Then the BO equation $(\ref{Benjamin Ono equation on the line})$ in the $N$-soliton manifold $(\mathcal{U}_N, \omega)$ can be written in Hamiltonian form
\begin{equation}\label{Hamiltonian energy of BO eq and Hamiltonian form of BO}
\partial_t u = X_E (u), \qquad \mathrm{where} \quad          E(u)= \frac{1}{2} \langle |\mathrm{D}|u, u \rangle_{H^{-\frac{1}{2}}, H^{\frac{1}{2}}} - \frac{1}{3}\int_{\mathbb{R}} u^3.
\end{equation}The Cauchy problem of $(\ref{Hamiltonian energy of BO eq and Hamiltonian form of BO})$ is globally well-posed in the manifold $\mathcal{U}_N$ (see  proposition $\ref{Invariance of UN under BO flow}$). Inspired from the construction of Birkhoff coordinates of the space-periodic BO equation discovered by  G\'erard--Kappeler $[\ref{Gerard kappeler Benjamin ono birkhoff coordinates}]$, we want to show the complete integrability of $(\ref{Hamiltonian energy of BO eq and Hamiltonian form of BO})$ in the Liouville sense. \\

\noindent Let $\Omega_N:=\{(r^1, r^2, \cdots, r^N) \in \mathbb{R}^N : r^j <r^{j+1}<0, \quad \forall j =1,2, \cdots, N-1\}$ denote the subset of actions and $\nu =\sum_{j=1}^N \mathrm{d} r^j \wedge \mathrm{d} \alpha^j$ denotes the canonical symplectic form on $\Omega_N \times  \mathbb{R}^N$. The main result of this paper is stated as follows.  
\begin{thmprincipal}\label{principal theorem of this paper}
There exists a real analytic symplectomorphism $\Phi_N : (\mathcal{U}_N, \omega) \to (\Omega_N \times \mathbb{R}^N, \nu)$ such that 
\begin{equation}\label{Explicit formula in action variables}
E \circ \Phi_N^{-1}(r^1, r^2, \cdots, r^N;\alpha^1, \alpha^2, \cdots, \alpha^N) = - \frac{1}{2\pi}\sum_{j=1}^N |r^j|^2.
\end{equation}
\end{thmprincipal}

\begin{rem}\label{Universal covering map UN soliton to UN gap potential}
A consequence of theorem $\ref{principal theorem of this paper}$ is that $\mathcal{U}_N$ is simply connected. In fact the manifold $\mathcal{U}_N$ can be interpreted as the universal covering of the manifold of $N$-gap potentials for the Benjamin--Ono equation on the torus as described by G\'erard--Kappeler in $[\ref{Gerard kappeler Benjamin ono birkhoff coordinates}]$. We refer to section $\mathbf{\ref{section appendix}}$ for a direct proof of these topological facts, independently of theorem $\ref{principal theorem of this paper}$.
\end{rem}

\begin{rem}\label{remark behind principal thm 1 to define action angles}
Then  $\Phi_N : u \in \mathcal{U}_N \mapsto (I_1(u), I_2(u),  \cdots, I_N(u); \gamma_1(u), \gamma_2(u), \cdots, \gamma_N(u)) \in \Omega_N \times \mathbb{R}^N$ introduces the  generalized action--angle coordinates of  the BO equation in the $N$-soliton manifold, i.e.  
\begin{equation}\label{Action angle meaning Poisson brackets}
\{I_k, E\}(u) =0, \qquad \{\gamma_k, E\}(u)=    \frac{I_k (u)}{\pi}, \qquad \forall u \in \mathcal{U}_N.
\end{equation}Theorem $\ref{principal theorem of this paper}$ gives a complete description of the orbit structure  of the flow of equation  $(\ref{Hamiltonian energy of BO eq and Hamiltonian form of BO})$ up to real bi-analytic conjugacy. Let $u : t \in \mathbb{R} \mapsto u(t) \in \mathcal{U}_N$ denote the solution of equation $(\ref{Hamiltonian energy of BO eq and Hamiltonian form of BO})$, $r^k(t)=I_k \circ u(t)$ denotes action coordinates and $\alpha^k(t)=\gamma_k \circ u(t)$ denotes the generalized angle coordinates, then we have 
\begin{equation}\label{intro conjugated flow}
r^k(t)= r^k(0), \qquad \alpha^k(t)= \alpha^k(0)  -  \frac{r^{k}(0)t}{\pi}, \qquad \forall k =1,2,\cdots, N.
\end{equation}We refer to definition $\ref{j th action and angle}$ and theorem $\ref{action angle coordinate thm}$ for a precise description of $\Phi_N$.
\end{rem}

\noindent In order to establish the link between  the action--angle coordinates and the translation--scaling parameters of an $N$-soliton, we introduce the inverse spectral matrix associated to  $\Phi_N$, denoted by  
\begin{equation}\label{formula introduction coefficients of matrix of G def}
M : u \in \mathcal{U}_N \mapsto (M_{kj}(u) )_{1\leq j, k\leq N}\in \mathbb{C}^{N\times N}, \quad M_{k j}(u) = \begin{cases}
\frac{2 \pi i}{I_k(u)-  I_j(u)} \sqrt{\frac{ I_k(u) }{I_j(u)}}, \quad \mathrm{if}\quad j\ne k,\\
  \gamma_j(u) + \frac{\pi i}{  I_j(u)}, \qquad \quad \mathrm{if}\quad j = k,
\end{cases}
\end{equation}where $I_k, \gamma_k : \mathcal{U} \to \mathbb{R}$ is given by remark $\ref{remark behind principal thm 1 to define action angles}$. Then  $\mathcal{U}_N$ has the following polynomial characterization.
\begin{prop}\label{prop introduction polynomial characterization}
A real-valued function $u \in \mathcal{U}_N$ if and only if  there exists a  monic polynomial $Q_u \in \mathbb{C}[X]$ of degree $N$, whose roots are contained in the lower half-plane $\mathbb{C}_-$ and $  u = -2 \mathrm{Im} \frac{Q_u'}{Q_u}$. Precisely, $Q_u$ is unique and is the characteristic polynomial of the  matrix $M(u) \in \mathbb{C}^{N\times N}$ defined by $(\ref{formula introduction coefficients of matrix of G def})$.
\end{prop}

\noindent An $N$-soliton is expressed by $u(x)=\sum_{j=1}^N \mathcal{R}_{c_j}(x- x_j)$ if and only if its translation--scaling parameters $\{x_j-c_j^{-1}i\}_{1\leq j\leq N} \subset \mathbb{C}^N_-$ are the roots of the characteristic polynomial $Q_u(X)=\det(X-M(u))$, whose coefficients are expressed in terms of the action--angle coordinates $(I_j(u), \gamma_j(u))_{1\leq j\leq N} \in \Omega_N \times \mathbb{R}^N$. Proposition $\ref{prop introduction polynomial characterization}$ is restated with more details in proposition $\ref{Polynomial characterization of N soliton}$, formula $(\ref{inversion formula in terms of matrix M X Y})$ and theorem  $\ref{characterization of multisoliton manifold}$ which gives a spectral characterization of $\mathcal{U}_N$.   If $u : t \in \mathbb{R} \mapsto u(t) \in \mathcal{U}_N $ solves the BO equation $(\ref{Benjamin Ono equation on the line})$, then we have the following explicit formula
\begin{equation}\label{explicit formula of soliton solutions}
u(t,x)= 2\mathrm{Im} \langle \left(M(u_0) - (x+ \tfrac{t}{\pi}\mathfrak{V}(u_0))\right)^{-1} X(u_0), Y(u_0)\rangle_{\mathbb{C}^N}, \qquad (t,x)\in \mathbb{R} \times \mathbb{R} ,
\end{equation}where the inner product of $\mathbb{C}^N$ is $\langle X, Y\rangle_{\mathbb{C}^N} = X^T   \overline{Y}$, for every $u \in \mathcal{U}_N$, the matrix $\mathfrak{V}(u) \in \mathbb{C}^{N\times N}$ and the vectors $X(u), Y(u) \in \mathbb{C}^N$  are defined by 
\begin{equation*}
\begin{split}
\sqrt{2\pi} X(u)^T & = (\sqrt{|I_1(u)|}, \sqrt{|I_2(u)|}, \cdots, \sqrt{|I_N(u)|}), \\
\qquad \sqrt{2\pi}^{-1} Y(u)^T & = (\sqrt{|I_1(u)|^{-1}}, \sqrt{|I_2(u)|^{-1}}, \cdots, \sqrt{|I_N(u)|^{-1}}),
\end{split}\qquad \mathfrak{V}(u ) =  \left(\begin{smallmatrix}
I_1(u ) \\
&  I_2(u )  \\
& &  \ddots \\
& & &  I_N(u ) 
\end{smallmatrix}\right).
\end{equation*}

\subsection{Notation}\label{subsection Notations intro}
Before outlining the construction of action--angle coordinates,  we introduce some notations used in this paper. The indicator function of a subset $A \subset X$ is denoted by $\mathbf{1}_A$, i.e. $\mathbf{1}_A(x) =1$ if $x \in A$ and $\mathbf{1}_A(x) =0$ if $x \in X\backslash A$.  Recall that $\mathrm{H} : L^2(\mathbb{R}) \to L^2(\mathbb{R})$ denotes the Hilbert transform given by $(\ref{Hilbert transform Fourier multiplier def})$. Set $\mathrm{Id}_{L^2( \mathbb{R})} (f)=f$, for every $f \in L^2(\mathbb{R})$. Let $\Pi : L^2(\mathbb{R}) \to L^2(\mathbb{R})$ denote  the Szeg\H{o} projector, defined by 
\begin{equation}\label{Szego projector}
\Pi := \frac{\mathrm{Id}_{L^2( \mathbb{R})}+i\mathrm{H}}{2}  \Longleftrightarrow  \widehat{\Pi f}(\xi)= \mathbf{1}_{[0,+\infty)}(\xi)  \hat{f}(\xi),  \quad \forall \xi \in \mathbb{R}, \quad \forall f \in  L^2(\mathbb{R}).
\end{equation}If $\mathfrak{O}$ is an open subset of $\mathbb{C}$, we denote by $\mathrm{Hol}(\mathfrak{O})$ all holomorphic functions on $\mathfrak{O}$. Let the upper half-plane and the lower half-plane be denoted by $\mathbb{C}_+ = \{z \in \mathbb{C} : \mathrm{Im}z >0\}$ and $\mathbb{C}_- = \{z \in \mathbb{C} : \mathrm{Im}z <0\}$ respectively. For every $p \in (0, +\infty]$, we denote by $L^p_+$ to be the Hardy space of holomorphic functions on $\mathbb{C}_+$ such that $L^p_+= \{g \in \mathrm{Hol}(\mathbb{C}_+) : \|g\|_{L^p_+}<+\infty\}$, where
\begin{equation}\label{Hardy space def}
 \|g\|_{L^p_+}= \sup_{y >0} \left(\int_{\mathbb{R}}|g(x+iy)|^p\mathrm{d}x\right)^{\frac{1}{p}}, \qquad\mathrm{if} \quad p \in (0, +\infty),
\end{equation}and $\|g\|_{L^{\infty}_+}=\sup_{z \in \mathbb{C}_+}|g(z)|$. A function $g \in L^{\infty}_+$ is called an $inner$ $function$ if $|g |=1$ on $\mathbb{R}$. When $p=2$, the Paley--Wiener theorem  yields the identification between $L^2_+$ and $\Pi [L^2(\mathbb{R})]$:
\begin{equation*}
L^2_+ = \mathcal{F}^{-1}[ L^2(0, +\infty)] = \{f \in L^2(\mathbb{R}) : \mathrm{supp} \hat{f} \subset [0, +\infty)\} = \Pi (L^2(\mathbb{R})),
\end{equation*}where $\mathcal{F} : f\in L^2(\mathbb{R}) \mapsto \hat{f}\in L^2(\mathbb{R})$ denotes the Fourier--Plancherel transform. Similarly, we set $L^2_-= (\mathrm{Id}_{L^2( \mathbb{R})} - \Pi)( L^2(\mathbb{R}))$. Let the filtered Sobolev spaces be denoted as $H^s_+:=L^2_+ \bigcap H^s(\mathbb{R})$ and $H^s_-:=L^2_- \bigcap H^s(\mathbb{R})$, for every $s\geq 0$.\\

\noindent The domain of definition of an unbounded operator $\mathcal{A}$ on some Hilbert space $\mathcal{E}$ is denoted by $\mathbf{D}(\mathcal{A}) \subset \mathcal{E}$. Given another operator $\mathcal{B}$ on $\mathbf{D}(\mathcal{B}) \subset \mathcal{E}$ such that $\mathcal{A}(\mathbf{D}(\mathcal{A}))\subset \mathbf{D}(\mathcal{B})$ and $\mathcal{B}(\mathbf{D}(\mathcal{B})) \subset \mathbf{D}(\mathcal{A})$, their Lie bracket is an operator defined on $\mathbf{D}(\mathcal{A}) \bigcap  \mathbf{D}(\mathcal{B}) \subset \mathcal{E}$, which is given by 
\begin{equation}\label{def of Lie bracket of two operators}
 [\mathcal{A}, \mathcal{B}]:=\mathcal{A}  \mathcal{B} -  \mathcal{B}\mathcal{A}.
\end{equation}If the operator $\mathcal{A}$ is self-adjoint, let $\sigma(\mathcal{A})$ denote its spectrum, $\sigma_{\mathrm{pp}}(\mathcal{A})$ denotes the set of  its eigenvalues and $\sigma_{\mathrm{cont}}(\mathcal{A})$ denotes its continuous spectrum. Then $\sigma_{\mathrm{cont}}(\mathcal{A}) \bigcup \overline{\sigma_{\mathrm{pp}}(\mathcal{A})}= \sigma(A) \subset \mathbb{R}$. Given two $\mathbb{C}$-Hilbert spaces $\mathcal{E}_1$ and $\mathcal{E}_2$, let $\mathfrak{B}(\mathcal{E}_1, \mathcal{E}_2)$ denote the $\mathbb{C}$-Banach space of all bounded $\mathbb{C}$-linear transformations $\mathcal{E}_1 \to\mathcal{E}_2$, equipped with the uniform norm. \\

\noindent Given a smooth manifold $\mathbf{M}$ of real dimension $N$, let $C^{\infty}(\mathbf{M})$ denote all smooth functions $f : \mathbf{M} \to \mathbb{R}$ and the set of all smooth vector fields is denoted by $\mathfrak{X}(\mathbf{M})$. The tangent (resp. cotangent) space to $\mathbf{M}$ at $p \in \mathbf{M}$ is denoted by $\mathcal{T}_p(\mathbf{M})$ (resp. $\mathcal{T}_p^*(\mathbf{M})$). Given $k\in \mathbb{N}$, the $\mathbb{R}$-vector space of smooth $k$-forms on $\mathbf{M}$ is denoted by $\boldsymbol{\Omega}^k(\mathbf{M})$. Given a $\mathbb{R}$-vector space $\mathbb{V}$, we denote by $\boldsymbol{\Lambda}^k(\mathbb{V}^*)$ the vector space of all $k$-covectors on $\mathbb{V}$. Given a smooth covariant tensor field $\mathbf{A}$ on $\mathbf{M}$ and $X\in \mathfrak{X}(\mathbf{M})$, the Lie derivative of $\mathbf{A}$ with respect to $X$ is denoted by $\mathscr{L}_X (\mathbf{A})$, which is also a smooth tensor field on $\mathbf{M}$. If $\mathbf{N}$ is another smooth manifold, $\mathbf{F} : \mathbf{N} \to \mathbf{M}$ is a smooth map and $\mathbf{A}$ is a smooth covariant $k$-tensor field on $\mathbf{M}$, the pullback of $\mathbf{A}$ by $\mathbf{F}$ is denoted by $\mathbf{F}^*\mathbf{A}$, which is a smooth $k$-tensor field on $\mathbf{N}$ defined by $\forall p \in \mathbf{N}$, $\forall j =1,2, \cdots, k$,
\begin{equation}\label{pullback notation}
(\mathbf{F}^* \mathbf{A})_p (v_1,v_2, \cdots, v_k)= \mathbf{A}_{\mathbf{F}(p)} \left(\mathrm{d}\mathbf{F}(p)(v_1), \mathrm{d}\mathbf{F}(p)(v_2), \cdots, \mathrm{d}\mathbf{F}(p)(v_k) \right),  \quad \forall v_j \in \mathcal{T}_p(\mathbf{N}).
\end{equation}Given a positive integer $N$, let $\mathbb{C}_{\leq N-1}[X]$ denote the $\mathbb{C}$-vector space of all polynomials with complex coefficients whose degree is no greater than $N-1$ and $\mathbb{C}_N[X]= \mathbb{C}_{\leq N}[X] \backslash \mathbb{C}_{\leq N-1}[X]$ consists of all polynomials of degree exactly $N$. $\mathbb{R}_+=[0,+\infty)$ and $ \mathbb{R}_+^*=(0,+\infty)$. $D(z,r) \subset \mathbb{C}$ denotes the open disc of radius $r>0$, whose center is $z \in \mathbb{C}$.

\subsection{Organization of this paper}
The construction of action--angle coordinates for the BO equation $(\ref{Hamiltonian energy of BO eq and Hamiltonian form of BO})$ mainly relies on the Lax pair formulation $\partial_t L_u = [B_u, L_u]$, discovered by Nakamura $[\ref{Nakamura Backlund transform and conservation laws of the Benjamin Ono equation}]$ and Bock--Kruskal $[\ref{Bock  Kruskal A two parameter Miura transformation BO eq}]$. Section $\mathbf{\ref{Section Lax operator}}$ is dedicated to the spectral analysis of the Lax operator $L_u : h \in H^1_+ \mapsto -i \partial_x h - \Pi (uh) \in L^2_+$ given by definition $\ref{definition of L u and B u}$ for  general symbol $u \in L^2(\mathbb{R}, \mathbb{R})$, where $\Pi$ denotes the Szeg\H{o} projector given in $(\ref{Szego projector})$ and the Hardy space $L^2_+$ is defined in $(\ref{Hardy space def})$. $L_u$ is an unbounded self-adjoint operator on $L^2_+$ that is  bounded from below, it has essential spectrum $\sigma_{\mathrm{ess}}(L_u)=[0, +\infty)$. If $x \mapsto xu(x) \in L^2(\mathbb{R})$ in addition, every eigenvalue is negative and simple, thanks to an identity firstly found by Wu $[\ref{Wu Simplicity and finiteness of discrete spectrum of the Benjamin--Ono}]$. Then we introduce a generating function which encodes the entire BO hierarchy,
\begin{equation}\label{Generating function in introduction}
\mathcal{H}_{\lambda}(u)= \langle (L_u+ \lambda)^{-1} \Pi u , \Pi u\rangle_{L^2}, \qquad \mathrm{if} \quad \lambda \in \mathbb{C}\backslash \sigma(-L_u),
\end{equation}in definition $\ref{def of generating function}$. It provides a sequence of conservation laws controlling every Sobolev norms.\\

\noindent In section $\mathbf{\ref{section infinitesimal generators}}$, we study the shift semigroup $(S(\eta)^*)_{\eta \geq 0}$ acting on the Hardy space $L^2_+$, where $S(\eta)f = e_{\eta} f$ and $e_{\eta}(x)=e^{i\eta x}$. Then a  weak version of Beurling--Lax theorem can be obtained by solving a linear differential equation with constant coefficients. Every $N$-dimensional   subspace of $L^2_+$ that is invariant under its infinitesimal generator $G = i \frac{\mathrm{d}}{\mathrm{d}\eta}\big|_{\eta=0^+} S(\eta)^*$ is of the form $\frac{\mathbb{C}_{\leq N-1}[X]}{Q}$, for some monic polynomial $Q$ whose roots are contained in the lower half-plane $\mathbb{C}_-$.\\

\noindent In section $\mathbf{\ref{section of Manifold}}$, the real analytic structure and symplectic structure of the $N$-soliton subset $\mathcal{U}_N$ are established at first.  Then we continue the spectral analysis of the Lax operator $L_u$, $\forall u \in \mathcal{U}_N$. $L_u$ has $N$ simple eigenvalues $\lambda_1^u<\lambda_2^u< \cdots  <\lambda_N^u <0$ and the Hardy space $L^2_+$ splits as
\begin{equation}\label{spectral split of Lu introduction}
L^2_+ = \mathscr{H}_{\mathrm{cont}}(L_u) \bigoplus \mathscr{H}_{\mathrm{pp}}(L_u), \quad 
\mathscr{H}_{\mathrm{cont}}(L_u) = \mathscr{H}_{\mathrm{ac}}(L_u) = \Theta_u L^2_+,   \quad \mathscr{H}_{\mathrm{pp}}(L_u)   = \tfrac{\mathbb{C}_{\leq  N-1}[X]}{Q_u}.
\end{equation}where $Q_u$ denotes the characteristic polynomial of $u$ given by proposition $\ref{prop introduction polynomial characterization}$ and $\Theta_u = \frac{\overline{Q}_u}{Q_u}$ is an inner function on the upper half-plane $\mathbb{C}_+$. Proposition $\ref{prop introduction polynomial characterization}$ is proved by identifying $M(u)$ in $(\ref{formula introduction coefficients of matrix of G def})$  as the matrix of the restriction $G|_{\mathscr{H}_{\mathrm{pp}}(L_u)}$ associated to the spectral basis $\{\varphi_1^u, \varphi_2^u, \cdots, \varphi_N^u\}$, where $\varphi_j^u \in \mathrm{Ker}(\lambda_j^u-L_u)$ such that $\|\varphi_j^{u}\|_{L^2} =1$ and $\int_{\mathbb{R}}u \varphi_j^u >0$. The generating function $\mathcal{H}_{\lambda}$ in $(\ref{Generating function in introduction})$ can be identified as the Borel--Cauchy transform of the spectral measure of $L_u$ associated to the vector $\Pi u$, which yields the invariance of $\mathcal{U}_N$ under the BO flow in $H^{\infty}(\mathbb{R}, \mathbb{R})$. Hence $(\ref{Hamiltonian energy of BO eq and Hamiltonian form of BO})$ is a globally well-posed Hamiltonian system on $\mathcal{U}_N$. \\

\noindent Section $\mathbf{\ref{section of action angle coordinates}}$ is dedicated to completing the proof of theorem $\ref{principal theorem of this paper}$. The generalized angle-variables are the real parts of the diagonal elements of the matrix $M(u)$, i.e. $\gamma_j : u \in \mathcal{U}_N \mapsto \mathrm{Re} \langle G \varphi_j^u, \varphi_j^u \rangle_{L^2} \in \mathbb{R}$ and the action-variables are $I_j : u \in \mathcal{U}_N \mapsto 2\pi \lambda_j^u \in \mathbb{R}$. Thanks to the Lax pair formulation $\mathrm{d}L(u) (X_{\mathcal{H}_{\lambda}}(u))=[B_u^{\lambda}, L_u]$, where $L : u \in \mathcal{U}_N \mapsto L_u \in \mathfrak{B}(H^1_+, L^2_+)$ is $\mathbb{R}$-affine and $B^{\lambda}_u$ is some skew-adjoint operator on $L^2_+$, we have the following formulas of Poisson brackets,
\begin{equation}\label{poisson bracket of lambda gamma introduction}
2\pi \{\lambda_j, \gamma_k\}= \mathbf{1}_{j=k},  \qquad \{\gamma_j, \gamma_k\}= 0  \quad \mathrm{on} \quad \mathcal{U}_N, \quad 1\leq j , k \leq N.
\end{equation}which implies that $\Phi_N : u \in \mathcal{U}_N \mapsto (I_1(u), I_2(u), \cdots, I_N(u);\gamma_1(u), \gamma_2(u), \cdots, \gamma_N(u)) \in \Omega_N \times \mathbb{R}^N$ is a real analytic immersion. The diffeomorphism property of $\Phi_N$ is given by  Hadamard's global inverse theorem. The inverse spectral formula  $\Pi u = \frac{Q_u'}{Q_u}$ with $Q_u(X)=\det(X-G|_{\mathscr{H}_{\mathrm{pp}}(L_u)})$, which is restated as formula $(\ref{inversion formula in terms of matrix M X Y})$, implies the explicit formula $(\ref{explicit formula of soliton solutions})$ of all multi-soliton solutions of the BO equation $(\ref{Benjamin Ono equation on the line})$ and $(\ref{inversion formula in terms of matrix M X Y})$ provides an alternative proof of the injectivity of $\Phi_N$. Finally, we show that $\Phi_N : (\mathcal{U}_N, \omega) \to (\Omega_N \times \mathbb{R}^N, \nu)$ is a symplectomorphism  by restricting the $2$-form $ \omega - \Phi_N^*\nu $ to a special Lagrangian submanifold $\Lambda_N:=\bigcap_{j=1}^N \gamma_j^{-1}(0) \subset \mathcal{U}_N$. \\

\noindent In appendix $\mathbf{\ref{section appendix}}$, we establish the simple connectedness of $\mathcal{U}_N$ and a covering map from $\mathcal{U}_N$ to the manifold of $N$-gap potentials from their constructions without using the integrability theorems.

\subsection{Related work}
The BO equation has been extensively studied for nearly sixty years in the domain of partial differential equations. We refer to Saut $[\ref{Saut BO recent results}]$ for an excellent account of these results. Besides the global well-posedness problem, various properties of its multi-soliton solutions has been investigated in details. Matsuno $[\ref{Matsuno Multisoliton BO}]$ has found the explicit expression of multi-soliton solutions of $(\ref{Benjamin Ono equation on the line})$ by following the bilinear method of Hirota $[\ref{Hirota Exact solution KdV}]$. The multi-phase solutions (periodic multi-solitons) have been constructed by Satsuma--Ishimori $[\ref{Satsuma Ishimori Periodic wave BO eq}]$ at first. We point out the work of Amick--Toland $[\ref{Amick Toland B o equation traveling wave classification}]$ on the characterization of $1$-soliton solutions which can also be revisited by theorem $\ref{principal theorem of this paper}$ and proposition $\ref{prop introduction polynomial characterization}$. In Dobrokhotov--Krichever $[\ref{Dobrokhotov Krichever Multiphase solutions }]$, the multi-phase solutions are constructed by finite zone integration and they have also established an inversion formula for multi-phase solutions. Compared to their work, we give a geometric description of the inverse spectral transform by proving the real bi-analyticity and the symplectomorphism property of the action--angle map.  Furthermore, the inverse spectral formula 
\begin{equation}\label{inverse spectral formula in related work}
\Pi u(x) = i \frac{Q_u'(x)}{Q_u(x)}, \qquad Q_u(x) =\det(x - G|_{\mathscr{H}_{\mathrm{pp}}(L_u)}) = \det(x - M(u)) , \qquad \forall x \in \mathbb{R}.
\end{equation}provides a spectral connection between the Lax operator $L_u$ and the infinitesimal generator $G$. The idea of introducing generating function $\mathcal{H}_{\lambda}$ has also been used for the quantum BO equation in Nazarov--Sklyanin $[\ref{Nazarov Sklyanin Integrable hierarchy of the quantum Integrable hierarchy of the quantum BO eq}]$. Their method has also been developed by Moll $[\ref{Moll Finite gap conditions BO}]$ for the classical BO equation. The asymptotic stability of soliton solutions and of solutions starting with sums of widely separated soliton profiles is obtained by Kenig--Martel $[\ref{Kenig Martel Asymptotic stability BO Eq}]$. \\

\noindent Concerning the investigation of integrability for the BO equation on $\mathbb{R}$ besides the discovery of Lax pair formulation, we mention the pioneering work of Ablowitz--Fokas $[\ref{Ablowitz Fokas IST for BO equation}]$, Coifman--Wickerhauser $[\ref{Coifman Wickerhauser The scattering transform for the Benjamin Ono equation}]$, Kaup--Matsuno $[\ref{Kaup Matsuno The inverse scattering BO}]$ and Wu $[\ref{Wu Simplicity and finiteness of discrete spectrum of the Benjamin--Ono}, \ref{Wu Jost solutions and the direct scattering problem of the Bo eq }]$ for the inverse scattering transform. In the space-periodic regime, the BO equation on the torus $\mathbb{T}$ admits global Birkhoff coordinates on $L^2_{r,0}(\mathbb{T}):=\{v \in L^2(\mathbb{T}, \mathbb{R}) : \int_{\mathbb{T}}v=0\}$ in G\'erard--Kappeler $[\ref{Gerard kappeler Benjamin ono birkhoff coordinates}]$. We refer to G\'erard--Kappeler--Topalov $[\ref{Gerard kappeler Topalov Benjamin ono flow map}]$ to see that the Birkhoff coordinates of the BO equation on the torus can be extended to a larger Sobolev space $H^s_{r,0}(\mathbb{T}):= \{v \in H^s(\mathbb{T}, \mathbb{R}) : \int_{\mathbb{T}}v=0\}$, for every $-\frac{1}{2} < s< 0$.  We point out that  both Korteweg--de Vries equation on $\mathbb{T}$ (see  Kappeler--P\"oschel $[\ref{Kappeler Poschel KdV KAM}]$) and the defocusing cubic Sch\"odinger equation on $\mathbb{T}$ (see Gr\'ebert--Kappeler $[\ref{Kappeler Grebert}]$)  admit global Birkhoff coordinates. The theory of finite-dimensional Hamiltonian system is transferred to the BO, KdV and dNLS equation on $\mathbb{T}$ through the submanifolds of corresponding finite-gap potentials, which are introduced to solve the periodic KdV initial problem. We refer to Matveev $[\ref{Matveev  finite-gap integration}]$ for details.  \\

\noindent Moreover, the cubic Szeg\H{o} equation both on $\mathbb{T}$ (see G\'erard--Grellier $[\ref{gerardgrellier1 ann ens}, \ref{gerardgrellier2 invention invariant tori}, \ref{Gerard-grellier explicit formula szego equation}, \ref{Gerard grellier book cubic szego equation and hankel operators}]$) and on $\mathbb{R}$ (see Pocovnicu $[\ref{pocovnicu Traveling waves for the cubic Szego eq},\ref{pocovnicu Explicit formula for the cubic Szego eq}]$) admit global (generalized) action--angle coordinates on all finite-rank generic rational function manifolds, denoted respectively by $\mathcal{M}(N)_{\mathrm{gen}}^{\mathbb{T}}$ and $\mathcal{M}(N)_{\mathrm{gen}}^{\mathbb{R}}$. Moreover, the cubic Szeg\H{o} equation both on $\mathbb{T}$ and on $\mathbb{R}$ have inverse spectral formulas which permit the Szeg\H{o} flows to be expressed explicitly in terms of  time-variables and  initial data without using  action--angle coordinates.  The shift semigroup $(S(\eta)^*)_{\eta \geq 0}$ and its infinitesimal generator $G$ are also used in Pocovnicu $[\ref{pocovnicu Explicit formula for the cubic Szego eq}]$ to establish the integrability of the cubic Szeg\H{o} equation on the line.\\

\noindent The BO equation admits an infinite hierarchy of conservation laws controlling every $H^s$-norm  (see Ablowitz--Fokas $[\ref{Ablowitz Fokas IST for BO equation}]$, Coifman--Wickerhauser $[\ref{Coifman Wickerhauser The scattering transform for the Benjamin Ono equation}]$ in the case $2s\in \mathbb{N}$ and Talbut $[\ref{Talbut Low regularity conservation laws}]$ in the case $-\frac{1}{2} < s < 0$ and conservation law controlling Besov norms etc.), so does the KdV equation and the NLS equation (see Killip--Vi\c{s}an--Zhang $[\ref{Killip Visan Zhang Low regularity conservation laws for integrable PDE}]$, Koch--Tataru $[\ref{Koch Tataru Conserved energies for the cubic NLS}]$,  Faddeev--Takhtajan $[\ref{Faddeev-Takhtajan Hamiltonian Methods in the Theory of Solitons}]$, G\'erard $[\ref{Gerard defocusing NLS integrals}]$ and Sun $[\ref{Sun Master thesis}]$ etc.) \\

\noindent \textbf{Throughout this paper, the main results of each section are stated at the beginning. Their proofs are left inside the corresponding subsections.}

\bigskip
\bigskip

\section{The Lax operator}\label{Section Lax operator}
This section is dedicated to studying the Lax operator $L_u$ in the Lax pair formulation of the BO equation $(\ref{Benjamin Ono equation on the line})$, discovered by  Nakamura $[\ref{Nakamura Backlund transform and conservation laws of the Benjamin Ono equation}]$ and Bock--Kruskal $[\ref{Bock  Kruskal A two parameter Miura transformation BO eq}]$. Then we describe the location and revisit the simplicity of eigenvalues of $L_u$. At last, we introduce a generating functional $\mathcal{H}_{\lambda}$ which encodes the entire BO hierarchy. The equation $\partial_t u = \partial_x \nabla_u \mathcal{H}_{\lambda}(u)$ also enjoys  a Lax pair structure with the same Lax operator $L_u$.

\begin{defi}\label{definition of L u and B u}
Given $u \in L^2(\mathbb{R}, \mathbb{R})$, its associated Lax operator $L_u$ is an unbounded operator on $L^2_+$, given by $L_u:= \mathrm{D}- T_u$, where $\mathrm{D}: h \in H^1_+ \mapsto -i \partial_x h \in L^2_+$ and $T_u$ denotes the Toeplitz operator of symbol $u$, defined by $T_u : h \in H^1_+ \mapsto  \Pi (uh) \in L^2_+$, where the Szeg\H{o} projector $\Pi : L^2 (\mathbb{R})\to L^2_+$ is given by $(\ref{Szego projector})$. We set $B_u:=i(T_{|\mathrm{D}|u} - T_u^2)$. 
\end{defi}

\noindent Both $\mathrm{D}$ and $T_u$ are densely defined symmetric operators on $L^2_+$ and $\|T_u(h) \|_{L^2} \leq \|u\|_{L^2} \|h\|_{L^{\infty}}$, for every $h \in H^1_+$ and $u\in L^2(\mathbb{R}, \mathbb{R})$. Moreover, the Fourier--Plancherel transform implies that $\mathrm{D}$ is a self-adjoint operator on $L^2_+$, whose domain of definition is $H^1_+$. 
\begin{prop}\label{Self adjoint ness of Lu}
If $u\in L^2(\mathbb{R}, \mathbb{R})$, then $L_u$ is an unbounded self-adjoint operator on $L^2_+$, whose domain of definition is $\mathbf{D}(L_u)=H^1_+$. Moreover, $L_u$ is bounded from below. The essential spectrum of $L_u$ is $\sigma_{\mathrm{ess}}(L_u)=\sigma_{\mathrm{ess}}(\mathrm{D})=[0,+\infty)$ and its pure point spectrum  satisfies $\sigma_{\mathrm{pp}}(L_u) \subset [-\frac{C^2}{4} \|u\|_{L^2}^2, +\infty)$, where $C=\inf_{f \in H^1_+ \backslash \{0\}}\frac{\||\mathrm{D}|^{\frac{1}{4}}f\|_{L^2}}{\| f\|_{L^4}}$ denotes the Sobolev constant.
\end{prop}

\noindent Thanks to an identity firstly found by Wu $[\ref{Wu Simplicity and finiteness of discrete spectrum of the Benjamin--Ono}]$ in the negative eigenvalue case,  we show the simplicity of the pure point spectrum $\sigma_{\mathrm{pp}}(L_{u})$, if $u \in L^2(\mathbb{R}, (1+x^2)\mathrm{d}x)$ is real-valued.
\begin{prop}\label{proposition showing identity of wu}
Assume that $u\in L^2(\mathbb{R}; \mathbb{R})$ and $x \mapsto xu(x) \in L^2(\mathbb{R})$. For every $\lambda \in \mathbb{R}$ and $\varphi \in \mathrm{Ker}(\lambda-L_u)$,  we have $\widehat{u \varphi} \in C^1(\mathbb{R}) \bigcap H^1(\mathbb{R})$ and the following identity holds,
\begin{equation}\label{identity of Wu for simplicity of spectrum of Lu}
\Big|\int_{\mathbb{R}}u\varphi\Big|^2 = -2\pi  \lambda   \int_{\mathbb{R}}|\varphi|^2.
\end{equation}Thus $\sigma_{\mathrm{pp}}(L_u) \subset (-\infty, 0)$ and for every $\lambda \in \sigma_{\mathrm{pp}}(L_u)$, we have 
\begin{equation}\label{regularity of Hpp functions}
\mathrm{Ker}(\lambda-L_u) \subset \{\varphi \in H^1_+ : \hat{\varphi}_{|\mathbb{R}_+} \in C^1(\mathbb{R}_+) \bigcap H^{1}(\mathbb{R}_+) \quad \mathrm{and} \quad \xi \mapsto \xi [ \hat{\varphi}(\xi)+  \partial_{\xi} \hat{\varphi}(\xi)] \in L^2(\mathbb{R}_+)\}.
\end{equation}
 \end{prop}

\begin{cor}\label{Wu's result on simplicity of eigenvalues}
Assume that $u\in L^2(\mathbb{R}; \mathbb{R})$ and $x \mapsto xu(x) \in L^2(\mathbb{R})$. Then every eigenvalue of $L_u$ is simple. If $u \in L^{\infty}(\mathbb{R})$ in addition, then $\sigma_{\mathrm{pp}}(L_u)$ is a finite subset of $[-\frac{C^2\|u\|_{L^2}^2}{4}, 0)$.
\end{cor}

\begin{proof}
Fix $\lambda \in \sigma_{\mathrm{pp}}(L_u)$ and set $V_{\lambda}=\mathrm{Ker}(\lambda-L_u)$, then $\mathrm{dim}_{\mathbb{C}}(V_{\lambda}) \geq 1$. We define a linear form $ A:V_{\lambda} \to \mathbb{C}$ such that
 \begin{equation*}
 A(\varphi):= \int_{\mathbb{R}}u\varphi
 \end{equation*}Then identity $(\ref{identity of Wu for simplicity of spectrum of Lu})$ yields that $\mathrm{Ker}(A)=\{0\}$. Thus $V \cong V\slash \mathrm{Ker}(A)\cong \mathrm{Im}(A) \hookrightarrow \mathbb{C}$. So we have $\mathrm{dim}_{\mathbb{C}}(V_{\lambda}) = 1$. When $u \in L^{\infty}(\mathbb{R})$ in addition, the finiteness of $\sigma_{\mathrm{pp}}(L_u)\bigcap (-\infty, 0)$ is given by Theorem 1.2 of Wu $[\ref{Wu Simplicity and finiteness of discrete spectrum of the Benjamin--Ono}]$. 
\end{proof}

\noindent We recall some known results of global well-posedness of the BO equation on the line. 
\begin{prop}\label{GWP for H^s solution of BO eq}
For every $s\geq 0$,  the Fr\'echet space $C(\mathbb{R}, H^s(\mathbb{R}))$ is endowed with the topology of uniform convergence on every compact subset of $\mathbb{R}$. There exists a unique continuous mapping $u_0 \in H^s(\mathbb{R}) \mapsto u\in C(\mathbb{R}, H^s(\mathbb{R}))$ such that $u$ solves the BO equation $(\ref{Benjamin Ono equation on the line})$ with initial datum $u(0)=u_0$. 
\end{prop}
\begin{proof}
See Tao $[\ref{Tao GWP of BO eq R}]$, Burq--Planchon $[\ref{Burq Planchon  BO GWP s bigger than 0.25}]$, Ionescu--Kenig $[\ref{Ionescu Kenig GWP of BO low regularity}]$, Molinet--Pilod $[\ref{Molinet Pilod L2 GWP of BOeq}]$, Ifrim--Tataru $[\ref{Ifrim Tataru Well-posedness and dispersive decay of BO eq}]$ etc.
\end{proof}
\begin{prop}\label{prop of conservation law controling every sobolev norms}
For every $n \in \mathbb{N}$, if $u_0 \in H^{\frac{n}{2}}(\mathbb{R}, \mathbb{R})$, let $u : t\in \mathbb{R} \mapsto u(t) \in H^{\frac{n}{2}}(\mathbb{R}, \mathbb{R})$ solves  equation $(\ref{Benjamin Ono equation on the line})$ with initial datum $u(0)=u_0$, then $ C(\|u_0\|_{H^{\frac{n}{2}}}):=\sup_{t\in \mathbb{R}}\|u(t)\|_{H^{\frac{n}{2}}} <+\infty$. 
\end{prop}
\begin{proof}
See Ablowitz--Fokas $[\ref{Ablowitz Fokas IST for BO equation}]$, Coifman--Wickerhauser $[\ref{Coifman Wickerhauser The scattering transform for the Benjamin Ono equation}]$.
\end{proof}

\noindent When $u \in H^2(\mathbb{R}, \mathbb{R})$, the Toeplitz operators $T_{|\mathrm{D}|u}$ and $T_{ u}$ are bounded  both on $L^2_+$ and on $H^1_+$. So $B_u$ is a bounded skew-adjoint operator both on $L^2_+$ and on $H^1_+$.
\begin{prop}\label{Lax pair structure of Benjamin ono equation}
Let $u : t\in \mathbb{R} \mapsto u(t)\in H^2(\mathbb{R}, \mathbb{R})$ denote the unique solution of  equation $(\ref{Benjamin Ono equation on the line})$, then 
\begin{equation}\label{Lax equation of Lu Bu}
\partial_t L_{u(t)}= [B_{u(t)}, L_{u(t)}] \in \mathfrak{B}(H^1_+, L^2_+), \qquad \forall t \in \mathbb{R}.
\end{equation}
\end{prop}
\noindent Let $U : t \mapsto U(t) \in \mathfrak{B}(L^2_+):= \mathfrak{B}(L^2_+, L^2_+)$ denote the unique solution of the following equation 
\begin{equation}\label{Unitary operator U(t)}
U'(t)=B_{u(t)} U(t), \qquad U(0)=\mathrm{Id}_{L^2_+}, 
\end{equation}if $u : t\in \mathbb{R} \mapsto u(t)\in H^2(\mathbb{R}, \mathbb{R})$ denote the unique solution of  equation $(\ref{Benjamin Ono equation on the line})$. The system $(\ref{Unitary operator U(t)})$ is globally well-posed in $\mathfrak{B}(L^2_+)$, thanks to proposition $\ref{prop of conservation law controling every sobolev norms}$, the following estimate
\begin{equation*}
\|B_u (h)\|_{L^2} \lesssim (\|u\|_{H^2}+\|u\|_{H^1}^2)\|h\|_{L^2}, \qquad \forall h \in L^2_+, \quad \forall u\in H^2(\mathbb{R}, \mathbb{R}).
\end{equation*}and a classical Cauchy theorem (see for instance lemma 7.2 of Sun $[\ref{Sun Master thesis}]$). Since $B_u^*=-B_u$, the operator $U(t)$ is unitary for every $t \in \mathbb{R}$. Thus, the Lax pair formulation $(\ref{Lax equation of Lu Bu})$ of the BO equation $(\ref{Benjamin Ono equation on the line})$ is equivalent to the unitary equivalence between $L_{u(t)}$ and $L_{u(0)}$, 
\begin{equation}\label{Unitary equivalence between Lu Lu0}
 L_{u(t)} = U(t) L_{u(0)} U(t)^* \in \mathfrak{B}(H^1_+, L^2_+).
\end{equation}On the one hand, the spectrum of $L_u$ is invariant under the BO flow. In particular, we have $\sigma_{\mathrm{pp}}(L_{u(t)}) = \sigma_{\mathrm{pp}}(L_{u(0)})$. On the other hand, there exists a sequence of conservation laws controlling every Sobolev norms $H^{\frac{n}{2}}(\mathbb{R})$, $n\geq 0$.  Furthermore, the Lax operator in the Lax pair formulation is not unique. If $f \in L^{\infty}(\mathbb{R})$ and $p$ is a polynomial with complex coefficients, then 
\begin{equation}\label{Unitary equivalence between f Lu and f Lu0}
 f(L_{u(t)}) = U(t) f(L_{u(0)}) U(t)^*  \in \mathfrak{B}(L^2_+), \qquad p(L_{u(t)}) = U(t) p(L_{u(0)}) U(t)^*  \in \mathfrak{B}(H^{N}_+, L^2_+),
\end{equation}where $N$ is the degree of the polynomial $p$.

\begin{prop}\label{Conservation law sequence controling h n/2 norm}
Given $n \in \mathbb{N}$, let  $u : t \in \mathbb{R} \mapsto u(t)  \in H^{\frac{n}{2}}(\mathbb{R}, \mathbb{R})$ denote the solution of equation $(\ref{Benjamin Ono equation on the line})$, we set
\begin{equation}\label{conservation law of En}
E_n (u) := \langle L_u^n \Pi u, \Pi u \rangle_{H^{-\frac{n}{2}}, H^{\frac{n}{2}}}.
\end{equation}Then $E_n (u(t)) = E_n (u(0))$, for every $t \in \mathbb{R}$. In particular, $E_1 = E$ on $H^{\frac{1}{2}}(\mathbb{R},\mathbb{R})$, where the energy functional $E$ is given by $(\ref{Hamiltonian energy of BO eq and Hamiltonian form of BO})$.
\end{prop}

\begin{defi}\label{def of generating function}
Given $u \in L^2(\mathbb{R}, \mathbb{R})$ and $\lambda \in \mathbb{C} \backslash \sigma(-L_u)$, 
the $\mathbb{C}$-linear transformation $\lambda + L_{u}$ is invertible in $\mathfrak{B}(H^1_+ , L^2_+)$ and the generating function  is defined by $\mathcal{H}_{\lambda }(u)= \langle (L_u +\lambda  )^{-1} \Pi u , \Pi u \rangle_{L^2}$. The subset $\mathcal{X}:=\{(\lambda, u)\in \mathbb{R}\times L^2(\mathbb{R}, \mathbb{R}) : 4\lambda >C^2 \|u\|_{L^2}^2\} $ is open in the $\mathbb{R}$-Banach space $\mathbb{R}\times L^2(\mathbb{R}, \mathbb{R})$, where the Sobolev constant is given by $C=\inf_{f \in H^1_+ \backslash \{0\}}\frac{\||\mathrm{D}|^{\frac{1}{4}}f\|_{L^2}}{\| f\|_{L^4}}$  and we have $\sigma(L_u)\subset [-\tfrac{C^2\|u\|_{L^2}^2}{4},+\infty)$  by proposition $\ref{Self adjoint ness of Lu}$. 
\end{defi}

\noindent The map $(\lambda ,u)\in \mathcal{X} \mapsto \mathcal{H}_{\lambda }(u)= \langle (L_u +\lambda  )^{-1} \Pi u , \Pi u \rangle_{L^2} \in \mathbb{R}$ is   real analytic.  
\begin{prop}\label{Conservation law of the generating function H lambda}
Let $u : t \in \mathbb{R} \mapsto u(t)  \in L^{2}(\mathbb{R}, \mathbb{R})$ denote the solution of the BO equation $(\ref{Benjamin Ono equation on the line})$ and we choose $\lambda> \frac{C^2 \|u(0)\|_{L^2}^2}{4}$, then  $\mathcal{H}_{\lambda}(u(t))= \mathcal{H}_{\lambda}(u(0))$, for every $t \in \mathbb{R}$.
\end{prop}

\noindent Given $(\lambda, u)\in \mathcal{X}$,  there exists a neighbourhood of $u$ in $L^2(\mathbb{R},\mathbb{R})$, denoted by $\mathcal{V}_u$ such that the restriction $\mathcal{H}_{\lambda} : v \in \mathcal{V}_u \mapsto \mathcal{H}_{\lambda}(v)  \in \mathbb{R}$ is real analytic. The Fr\'echet derivative of $\mathcal{H}_{\lambda}$ at $u$ is computed as follows,
\begin{equation*}
\mathrm{d} \mathcal{H}_{\lambda }(u)(h) = \langle w_{\lambda}  , \Pi h \rangle_{L^2} + \overline{\langle w_{\lambda} , \Pi h  \rangle}_{L^2}  + \langle  T_h w_{\lambda}   , w_{\lambda}  \rangle_{L^2} = \langle h, w_{\lambda}  + \overline{w}_{\lambda} + |w_{\lambda} |^2 \rangle_{L^2}, \quad \forall h \in L^2(\mathbb{R}, \mathbb{R}).
\end{equation*}where  $w_{\lambda} \in H^1_+$ is given by $w_{\lambda}\equiv w_{\lambda}(u) \equiv w_{\lambda}(x, u)= [(L_u +\lambda  )^{-1} \circ \Pi] u(x)$, for every $x \in \mathbb{R}$. Then 
\begin{equation}\label{frechet derivative of H lambda}
\nabla_u \mathcal{H}_{\lambda}(u)=|w_{\lambda}(u) |^2 +w_{\lambda}(u)  + \overline{w}_{\lambda}(u).
\end{equation}Given $(\lambda, u_0) \in \mathcal{X}$ fixed, the pseudo-Hamiltonian equation associated to $H_{\lambda}$ is defined by
\begin{equation}\label{Pseudo Hamiltonian eq of H lambda}
\partial_t u = \partial_x \nabla_u \mathcal{H}_{\lambda}(u) = \partial_x \left( |w_{\lambda}(u) |^2 +w_{\lambda}(u)  + \overline{w}_{\lambda}(u)\right), \qquad u(0)=u_0.
\end{equation}There exists an open subset $\mathcal{V}_{u_0}$ of $L^2(\mathbb{R}, \mathbb{R})$ such that $v \in \mathcal{V}_{u_0} \mapsto   \partial_x \left( |w_{\lambda}(v) |^2 +w_{\lambda}(v)  + \overline{w}_{\lambda}(v)\right) \in L^2_+$ is real analytic and $u_0 \in \mathcal{V}_{u_0}$. Hence $(\ref{Pseudo Hamiltonian eq of H lambda})$ admits a local solution by Cauchy--Lipschitz theorem.
\begin{rem}
The word 'pseudo-Hamiltonian' is used here because no symplectic form has been defined on $L^2(\mathbb{R}, \mathbb{R})$ until now. In section $\mathbf{\ref{section of Manifold}}$, we show that $\partial_x \nabla f(u)$ is exactly the Hamiltonian vector field of the smooth function  $f : \mathcal{U}_N \to \mathbb{R}$ with respect to the symplectic form $\omega$ on the $N$-soliton manifold $\mathcal{U}_N$ defined in $(\ref{definition of the symplectic form omega})$.
\end{rem}
\begin{prop}\label{lax pair structure of Hamiltonian equation associated to H lambda}
Given $(\lambda, u_0) \in \mathcal{X}$ fixed, there exists $\varepsilon>0$ such that $(\lambda, u(t)) \in \mathcal{X}$, for every $t \in (-\varepsilon, \varepsilon)$, where $u : t \in (-\varepsilon, +\varepsilon) \mapsto u(t) \in L^2(\mathbb{R}, \mathbb{R})$ denotes the local solution of $(\ref{Pseudo Hamiltonian eq of H lambda})$ with initial datum $u(0)=u_0$. We have
\begin{equation}\label{Lax equation of Hamiltonian equation associated to H lambda}
\partial_t L_{u(t)}= [B_{u(t) }^{\lambda}, L_{u(t)}] , \quad \mathrm{where} \quad B_{v }^{\lambda}:=i(T_{w_{\lambda}(v)}T_{\overline{w}_{\lambda}(v)} +T_{w_{\lambda}(v)}+ T_{\overline{w}_{\lambda}(v)} ), \quad \mathrm{if} \quad (\lambda, v)\in \mathcal{X}.
\end{equation}i.e. $(L_u, B_u^{\lambda})$ is a Lax pair of equation $(\ref{Pseudo Hamiltonian eq of H lambda})$.
\end{prop}
\begin{rem}
The  Toeplitz operators  $T_{w_{\lambda}(v)}$ and $T_{\overline{w}_{\lambda}(v)}$ are bounded both on $L^2_+$ and on $H^1_+$, so is the skew-adjoint operator $B_v^{\lambda}$, if $(\lambda, v) \in \mathcal{X}$.  
\end{rem}

\noindent For every $u \in H^{\infty}(\mathbb{R}, \mathbb{R})$ and $\epsilon \in (0, \frac{4}{C^2 \|u\|_{L^2}^2})$, we set $\tilde{\mathcal{H}}_{\epsilon}(u):=\frac{1}{\epsilon}\mathcal{H}_{\frac{1}{\epsilon}}(u)$ and $\tilde{B}_{\epsilon,u} :=\frac{1}{\epsilon}B_u^{\frac{1}{\epsilon}}$. Recall that $E_n(u)=\langle L_u^n \Pi u, \Pi u\rangle_{L^2}$, we have the following Taylor expansion
\begin{equation}\label{Taylor expansion of tilde H }
\tilde{\mathcal{H}}_{\epsilon}(u) =\sum_{k=0}^{M}(- \epsilon)^n E_n(u) - (-\epsilon)^M \langle (L_u + \tfrac{1}{\epsilon})^{-1} \Pi u, L_u^M \Pi u\rangle_{L^2}, \quad \forall M \in \mathbb{N}.
\end{equation}Proposition $\ref{lax pair structure of Hamiltonian equation associated to H lambda}$ then leads to a Lax pair formulation for the equations corresponding to the conservation laws in the BO hierarchy, 
\begin{equation*}
\partial_t L_u = [\frac{\mathrm{d}^n}{\mathrm{d}\epsilon^n}\Big|_{\epsilon=0}\tilde{B}_{\epsilon,u}, L_u], 
\end{equation*}where now $u$ evolves according to the pseudo-Hamiltonian flow of $E_n = (-1)^n\frac{\mathrm{d}^n}{\mathrm{d}\epsilon^n}\big|_{\epsilon=0}\tilde{\mathcal{H}}_{\epsilon}$. In the case $n=1$, we have $E_1=E$ and $B_u = \frac{\mathrm{d} }{\mathrm{d}\epsilon }\big|_{\epsilon=0}\tilde{B}_{\epsilon,u}$. \\

\noindent This section is organized as follows. In subsection $\mathbf{\ref{subsection unitary equivalence}}$, we recall some basic facts concerning unitarily equivalent self-adjoint operators on different Hilbert spaces. The subsection $\mathbf{\ref{subsection of spectral analysis in Lax operator}}$ is dedicated to the proofs of proposition $\ref{Self adjoint ness of Lu}$ and $\ref{proposition showing identity of wu}$. Proposition $\ref{Conservation law sequence controling h n/2 norm}$ and $\ref{Conservation law of the generating function H lambda}$ that concern the conservation laws are proved in subsection $\mathbf{\ref{subsection of conservation laws}}$. Proposition $\ref{Lax pair structure of Benjamin ono equation}$ and proposition $\ref{lax pair structure of Hamiltonian equation associated to H lambda}$ that indicate the Lax pair structures are proved in subsection $\mathbf{\ref{subsection Lax pair formulation}}$.

\subsection{Unitary equivalence}\label{subsection unitary equivalence}
\noindent Generally, if $\mathcal{E}_1$ and $\mathcal{E}_2$ are two Hilbert spaces, let $\mathcal{A}$ be a self-adjoint operator defined on $\mathbf{D}(\mathcal{A}) \subset \mathcal{E}_1$ and $\mathcal{B}$ be a self-adjoint operator defined on $\mathbf{D}(\mathcal{B}) \subset \mathcal{E}_2$. Both $\mathcal{A}$ and $\mathcal{B}$ have spectral decompositions 
\begin{equation}\label{spectral decomposition of A and B general}
\mathcal{E}_1 = \mathscr{H}_{\mathrm{ac}}(\mathcal{A}) \bigoplus \mathscr{H}_{\mathrm{sc}}(\mathcal{A}) \bigoplus \mathscr{H}_{\mathrm{pp}}(\mathcal{A}), \qquad \mathcal{E}_2 = \mathscr{H}_{\mathrm{ac}}(\mathcal{B}) \bigoplus \mathscr{H}_{\mathrm{sc}}(\mathcal{B}) \bigoplus \mathscr{H}_{\mathrm{pp}}(\mathcal{B}).
\end{equation}If $\mathcal{A}$ and $\mathcal{B}$ are unitarily equivalent i.e. there exists a unitary operator $\mathcal{U}: \mathcal{E}_1 \to \mathcal{E}_2$ such that
\begin{equation}\label{def of unitary equivalence}
\mathcal{B} = \mathcal{U} \mathcal{A} \mathcal{U}^*, \qquad \mathbf{D}(\mathcal{B})=\mathcal{U}\mathbf{D}(\mathcal{A}),
\end{equation}then we have the following identification result.
\begin{prop}\label{proposition on unitary equivalence}
The operators $\mathcal{A}$ and $\mathcal{B}$ have the same spectrum and $\mathcal{U}\mathscr{H}_{\mathrm{xx}}(\mathcal{A})=\mathscr{H}_{\mathrm{xx}}(\mathcal{B})$, for every $\mathrm{xx} \in \{\mathrm{ac}, \mathrm{sc}, \mathrm{pp}\}$. Moreover, for every bounded borel function $f : \mathbb{R} \to \mathbb{C}$, $f(\mathcal{A})$ is a bounded operator on $\mathcal{E}_1$,  $f(\mathcal{B})$ is a bounded operator on $\mathcal{E}_2$, we have $f(\mathcal{B}) = \mathcal{U} f(\mathcal{A}) \mathcal{U}^*$.
 \end{prop}
\begin{proof}
If $f$ is a bounded Borel function, $\psi \in \mathcal{E}_1$, consider the spectral measure of $\mathcal{A}$ associated to the vector $\psi \in \mathcal{E}_1$, denoted by $\mu_{\psi}^{\mathcal{A}}$. Similarly, we denote by $\mu_{\mathcal{U}\psi}^{\mathcal{B}}$ the spectral measure of $\mathcal{B}$ associated to the vector $\mathcal{U}\psi \in \mathcal{E}_2$. Clearly, we have
\begin{equation*}
\mathrm{supp}(\mu_{\psi}^{\mathcal{A}}) \subset \sigma (\mathcal{A})\subset \mathbb{R}, \qquad \mathrm{supp}(\mu_{\mathcal{U}\psi}^{\mathcal{B}}) \subset \sigma (\mathcal{B})\subset \mathbb{R}.
\end{equation*}For every $\lambda \in \mathbb{C}\backslash \sigma(\mathcal{A}) = \mathbb{C}\backslash \sigma(\mathcal{B})$, formula $(\ref{def of unitary equivalence})$ implies that $\mathcal{U}(\lambda - \mathcal{A})^{-1} \mathcal{U}^*=(\lambda - \mathcal{B})^{-1}$. So the Borel--Cauchy transforms of these two spectral measures are the same. 
\begin{equation*}
\int_{\mathbb{R}} \frac{\mathrm{d}\mu_{\psi}^{\mathcal{A}}(\xi)}{\lambda - \xi} = \langle (\lambda-\mathcal{A})^{-1}\psi, \psi \rangle_{\mathcal{E}_1} = \langle (\lambda-\mathcal{B})^{-1}\mathcal{U}\psi, \mathcal{U}\psi \rangle_{\mathcal{E}_2} = \int_{\mathbb{R}} \frac{\mathrm{d}\mu_{\mathcal{U}\psi}^{\mathcal{B}}(\xi)}{\lambda - \xi}.
\end{equation*}Both of these two spectral measures have finite total variations : $\mu_{\psi}^{\mathcal{A}}(\mathbb{R})=\mu_{\mathcal{U}\psi}^{\mathcal{B}}(\mathbb{R}) = \|\psi\|_{\mathcal{E}_1}^2$. Since every finite Borel measure is uniquely determined by its Borel--Cauchy transform (see Theorem 3.21 of Teschl $[\ref{Teschl spectral theory book}]$ page 108), we have $\mu_{\psi}^{\mathcal{A}} =\mu_{\mathcal{U}\psi}^{\mathcal{B}}$. So the restriction $\mathcal{U}|_{\mathscr{H}_{\mathrm{xx}}(\mathcal{A})} : \mathscr{H}_{\mathrm{xx}}(\mathcal{A}) \to \mathscr{H}_{\mathrm{xx}}(\mathcal{B})$ is a linear isomorphism, for  every $\mathrm{xx} \in \{\mathrm{ac}, \mathrm{sc}, \mathrm{pp}\}$. Finally, we use the definition of the spectral measures to obtain 
\begin{equation*}
\langle f(\mathcal{A}) \psi, \psi \rangle_{\mathcal{E}_1} = \int_{\mathbb{R}} f(\xi)\mathrm{d}\mu_{\psi}^{\mathcal{A}}(\xi) = \int_{\mathbb{R}} f(\xi)\mathrm{d}\mu_{\mathcal{U}\psi}^{\mathcal{B}}(\xi) = \langle f(\mathcal{B}) \mathcal{U}\psi, \mathcal{U}\psi \rangle_{\mathcal{E}_2} 
\end{equation*}We may assume that $f$ is real-valued, so that $f(\mathcal{A})$ is self-adjoint. The polarization identity implies that $\langle f(\mathcal{A}) \psi, \phi \rangle_{\mathcal{E}_1} = \langle f(\mathcal{B}) \mathcal{U}\psi, \mathcal{U}\phi \rangle_{\mathcal{E}_2}$, for every $\psi, \phi \in \mathcal{E}_1$. So we obtain $f(\mathcal{B}) = \mathcal{U} f(\mathcal{A}) \mathcal{U}^*$ in the case $f$ is real-valued bounded Borel function. In the general case, it suffices to use $f = \mathrm{Re}f + i\mathrm{Im}f$.

\end{proof}

\subsection{Spectral analysis \uppercase\expandafter{\romannumeral1}}\label{subsection of spectral analysis in Lax operator}
In this subsection, we study the essential spectrum and discrete spectrum of the Lax operator $L_u$ by proving proposition $\ref{Self adjoint ness of Lu}$ and $\ref{proposition showing identity of wu}$. The spectral analysis of $L_u$ such that $u$ is a multi-soliton in definition $\ref{definition of The N soliton in introduction}$, will be continued in subsection $\mathbf{\ref{subsection proof of spectral decomposition caracterisation for Hpp Hac}}$.

\begin{proof}[Proof of proposition $\ref{Self adjoint ness of Lu}$]
For every $h\in L^2_+$, let $\mu^{\mathrm{D}}_h$ denote the spectral measure of $\mathrm{D}$ associated to  $h$, then 
\begin{equation*}
\langle f(\mathrm{D})h, h \rangle_{L^2} = \int_0^{+\infty} \hat{f}(\xi) \frac{|\hat{h}(\xi)|^2}{2\pi} \mathrm{d}\xi \Longrightarrow \mathrm{d}\mu^{\mathrm{D}}_h(\xi) = \frac{ \mathbf{1}_{[0,+\infty)}(\xi)|\hat{h}(\xi)|^2}{2\pi} \mathrm{d}\xi.
\end{equation*}Thus we have $\sigma(\mathrm{D})=\sigma_{\mathrm{ess}}(\mathrm{D})=\sigma_{\mathrm{ac}}(\mathrm{D}) =[0, +\infty)$. If $u\in L^2(\mathbb{R}, \mathbb{R})$,  we claim that $\mathcal{P}_u:=T_u \circ (\mathrm{D}+i)^{-1}$ is a Hilbert--Schmidt operator on $L^2_+$.\\

\noindent Recall that $ \mathbb{R}_+^*=(0,+\infty)$. In fact, let $\mathscr{F} : h\in L^2_+ \mapsto \frac{\hat{h}}{\sqrt{2\pi}}\in L^2 (\mathbb{R}_+^*)$  denotes the renormalized Fourier--Plancherel transform, then $\mathcal{A}_u := \mathscr{F} \circ \mathcal{P}_u \circ \mathscr{F}^{-1}$ is an operator on $L^2( \mathbb{R}_+^*)$. Then we have
\begin{equation*}
\mathcal{A}_u g(\xi) = \int_{0}^{+\infty} K_u(\xi, \eta) g(\eta)\mathrm{d}\eta, \qquad K_u(\xi, \eta):=\frac{\hat{u}(\xi -\eta)}{2\pi(\eta+i)}, \quad \forall \xi, \eta \in  \mathbb{R}_+^*.
\end{equation*}Hence its Hilbert--Schmidt norm $\|\mathcal{A}_u\|_{\mathcal{HS}(L^2(\mathbb{R}_+^*))} \leq \|K\|_{L^2(\mathbb{R}_+^*  \times \mathbb{R}_+^*)} \leq \frac{\|u\|_{L^2}}{2}$. Since $\mathcal{P}_u$ is unitarily equivalent to $\mathcal{A}_u$, we have $\|\mathcal{P}_u\|_{\mathcal{HS}( L^2_+)}^2=\sum_{\lambda \in \sigma(\mathcal{P}_u)} \lambda^2 = \sum_{\lambda \in \sigma(\mathcal{A}_u)}  \lambda^2=\|\mathcal{A}_u\|^2_{\mathcal{HS}(L^2(\mathbb{R}_+^*))} \leq \frac{\|u\|_{L^2}^2}{4}$. \\

\noindent Then the symmetric operator $T_u$ is relatively compact with respect to $\mathrm{D}$ and Weyl's essential spectrum theorem (Theorem \uppercase\expandafter{\romannumeral13}.14 of Reed--Simon $[\ref{Reed Simon book 4}]$) yields that $\sigma_{\mathrm{ess}}(L_u)=\sigma_{\mathrm{ess}}(\mathrm{D})$ and $L_u$ is  self-adjoint with $\mathbf{D}(L_u)=\mathbf{D}(\mathrm{D})=H^1_+$. An alternative proof of the self-adjointness of $L_u$ can be given by Kato--Rellich theorem (Theorem  \uppercase\expandafter{\romannumeral10}.12 of Reed--Simon $[\ref{Reed Simon book 2}]$) and the following estimate, for every $f  \in H^1_+$,
\begin{equation*}
2\pi \|f\|_{L^{\infty}} \leq \|\hat{f}\|_{L^1} \leq \|\hat{f}\|_{L^2}\sqrt{A} + \|\widehat{\partial_x f}\|_{L^2}\sqrt{A^{-1}} \leq 2  \left(\|\hat{f}\|_{L^2}\|\widehat{\partial_x f}\|_{L^2}\right)^{\frac{1}{2}}, \qquad A=\sqrt{\tfrac{\| \partial_x f \|_{L^2}}{\|f\|_{L^2}}}.
\end{equation*}So $\|T_u(f)\|_{L^2} \leq \|u\|_{L^2}\|f\|_{L^{\infty}} \leq \frac{2}{\pi} \|\partial_x f\|_{L^2} + \frac{\|u\|_{L^2}^2}{4}\|f\|_{L^2}$.\\

\noindent Moreover, $|\langle T_u f, f \rangle_{L^2}| = |\int_{\mathbb{R}}u|f|^2|\leq \|u\|_{L^2} \|f\|_{L^4}^2 \leq C \|u\|_{L^2} \|f\|_{L^2}  \||\mathrm{D}|^{\frac{1}{2}}f\|_{L^2}$ holds by Sobolev embedding $\|f\|_{L^4} \leq C  \||\mathrm{D}|^{\frac{1}{4}}f\|_{L^2}$, for every $f \in H^1_+$. Then $L_u$ is bounded from below, precisely
\begin{equation*}
\langle L_u f, f \rangle_{L^2} =  \||\mathrm{D}|^{\frac{1}{2}}f\|_{L^2}^2 -\langle T_u f, f \rangle_{L^2} \geq  - \tfrac{C^2 \|u\|_{L^2}^2 \|f\|_{L^2}^2}{4}.
\end{equation*}When $\lambda < - \tfrac{C^2 \|u\|_{L^2}^2  }{4}$, the map $L_u - \lambda : H^1_+ \to L^2_+$ is injective. Hence $\sigma_{\mathrm{pp}}(L_u) \subset [-\frac{C^2}{4} \|u\|_{L^2}^2, +\infty)$.
\end{proof}

\bigskip

\noindent Before the proof of proposition $\ref{proposition showing identity of wu}$, we recall a lemma concerning the regularity of convolutions.

 \begin{lem}\label{convolution lemma p p'}
For every $p \in (1,+\infty)$ and $m,n \in \mathbb{N}$, we have 
\begin{equation}\label{convolution injection}
W^{m, p}(\mathbb{R}) * W^{n, \frac{p}{p-1}}(\mathbb{R}) \hookrightarrow C^{m+n}(\mathbb{R})\bigcap W^{m+n, +\infty}(\mathbb{R}).
\end{equation}For every $f \in W^{m, p}(\mathbb{R}) * W^{n, \frac{p}{p-1}}(\mathbb{R})$, we have $\lim_{|x|\to +\infty}\partial_x^{\alpha}f(x)=0$, for every $\alpha =0,1, \cdots, m+n$.
\end{lem}
\begin{proof}
In the case $m=n=0$, it suffices use H\"older's inequality and the density argument of the Schwartz class $\mathscr{S}(\mathbb{R}) \subset W^{m,p}(\mathbb{R})$. In the case $m=0$ and $n=1$, recall that a continuous function whose weak-derivative is continuous is of class $C^1$ and $\langle f ,  \varphi\rangle_{\mathscr{D}(\mathbb{R}) ',\mathscr{D}(\mathbb{R}) }=f * \check{ \varphi}(0)$, we use the density argument of the test function class $\mathscr{D}(\mathbb{R}) \subset L^p(\mathbb{R})$.  We conclude by induction on $n\geq 1$ and $m\in \mathbb{N}$.
\end{proof}

\begin{rem}
Identity $(\ref{identity of Wu for simplicity of spectrum of Lu})$ was firstly found by Wu $[\ref{Wu Simplicity and finiteness of discrete spectrum of the Benjamin--Ono}]$ in the case $\lambda <0$. We show that $(\ref{identity of Wu for simplicity of spectrum of Lu})$ still holds in the case $\lambda \geq 0$. Hence the operator $L_u$ has no eigenvalues in $[0,+\infty)$.
\end{rem}
\begin{proof}[Proof of proposition $\ref{proposition showing identity of wu}$]
We choose $u\in L^2(\mathbb{R}; (1+x^2)\mathrm{d}x)$ such that $u(\mathbb{R})\subset\mathbb{R}$, $\lambda\in \mathbb{R}$ and $\varphi \in L^2_+$ such that $L_u (\varphi)= \lambda \varphi$. Applying the Fourier--Plancherel transform, we obtain
\begin{equation}\label{Fourier modes of Lu phi = lambda phi}
\widehat{u\varphi}(\xi)\mathbf{1}_{\xi\geq 0}=(\xi-\lambda)\hat{\varphi}(\xi)=:g_{\lambda}(\xi).
\end{equation}Since $\hat{u}\in H^1(\mathbb{R})$ and $\hat{\varphi} \in L^2(\mathbb{R})$, their convolution $\widehat{u\varphi} =\frac{1}{2\pi}\hat{u} * \hat{\varphi} \in C^1(\mathbb{R})\bigcap C_0(\mathbb{R})$, where $C_0(\mathbb{R})$ denotes the uniform closure of $C_c(\mathbb{R})$ with respect to the $L^{\infty}(\mathbb{R})$-norm, by lemma $\ref{convolution lemma p p'}$. Recall $\mathbb{R}_+=[0, +\infty)$.\\

\noindent We claim that
\begin{equation*}
\begin{cases}
\mathrm{if}\quad\lambda<0, \qquad \mathrm{then}\quad\hat{\varphi} \in C^1(\mathbb{R}_+); \\
\mathrm{if}\quad\lambda \geq 0, \qquad \mathrm{then}\quad\hat{\varphi} \in C(\mathbb{R}_+)\bigcap C^1(\mathbb{R}_+ \backslash \{\lambda\}).\\
\end{cases}
\end{equation*}In fact, if $\lambda \geq 0$, we have $g_{\lambda}(\lambda)=0$. Otherwise, $\lambda$ would be a singular point of $\hat{\varphi}$ that prevents $\hat{\varphi}$ from being a $L^2$ function on $\mathbb{R}_+$, because $\xi \to \frac{1}{\xi-\lambda} \notin L^2(\mathbb{R}_+)$. By using the fact $g  \in C^1(\mathbb{R}_+)$ ($g$ is right differentiable at $\xi=0$ and the derivative $g'$ is right continuous  at $\xi=0$), we have
\begin{equation*}
\hat{\varphi}(\xi)= \frac{g_{\lambda}(\xi)-g_{\lambda}(\lambda)}{\xi-\lambda}\to\begin{cases}
g'_{\lambda}(\lambda), \qquad \mathrm{if} \quad \lambda>0;\\
g'_{\lambda}(0^+), \qquad \mathrm{if} \quad \lambda=0;\\
\end{cases}
\end{equation*}when $\xi \to \lambda$. So $\hat{\varphi} \in C(\mathbb{R}_+)$ and $\lim_{\xi \to +\infty} \hat{\varphi}(\xi)=0$. Then we derive formula $(\ref{Fourier modes of Lu phi = lambda phi})$ with respect to $\xi$ to get the following
\begin{equation}\label{derivative of g lambda}
-i \widehat{xu}* \hat{\varphi}(\xi)=g'_{\lambda}(\xi)= (\widehat{u\varphi})'(\xi) = \hat{\varphi}(\xi) + (\xi-\lambda)(\hat{\varphi})'(\xi), \qquad \forall \xi \in [0,+\infty)\backslash \{\lambda\}.
\end{equation}Thus we have
\begin{equation}\label{derivative of Wu inequality}
\frac{\mathrm{d}}{\mathrm{d}\xi}[(\xi-\lambda)|\hat{\varphi}(\xi)|^2] = |\hat{\varphi}(\xi)|^2 + 2\mathrm{Re}[((\xi-\lambda)(\hat{\varphi})'(\xi))\overline{\hat{\varphi}}(\xi) ] = 2\mathrm{Re}[(\widehat{u\varphi})'(\xi)\overline{\hat{\varphi}}(\xi)]  - |\hat{\varphi}(\xi)|^2.
\end{equation}When $\lambda<0$, it suffices to integrate equation $(\ref{derivative of Wu inequality})$ on $[0, +\infty)$ and use the Plancherel formula
\begin{equation*}
\int_0^{+\infty} (\widehat{u\varphi})'(\xi)\overline{\hat{\varphi}}(\xi) \mathrm{d}\xi = -2\pi i\int_{\mathbb{R}}x u(x)| \varphi(x) |^2 \mathrm{d}x.
\end{equation*}We also use the fact $(\xi-\lambda)|\hat{\varphi}(\xi)|^2 = \widehat{u\varphi}(\xi) \overline{\hat{\varphi}}(\xi)\to 0$, as $\xi \to +\infty$. Thus,
\begin{equation*}
\lambda |\hat{\varphi}(0)|^2 = \int_0^{+\infty} \frac{\mathrm{d}}{\mathrm{d}\xi}[(\xi-\lambda)|\hat{\varphi}(\xi)|^2]  \mathrm{d}\xi = 4\pi \mathrm{Im} \int_{\mathbb{R}}x u(x)| \varphi(x) |^2 \mathrm{d}x - \int_0^{+\infty}|\hat{\varphi}(\xi)|^2  \mathrm{d}\xi = -2\pi \| \varphi \|_{L^2(\mathbb{R})}^2 .
\end{equation*}When $\lambda > 0$, there may be some problem of derivability of $\hat{\varphi}$ at $\xi =\lambda$. We replace the integral $\int_0^{+\infty}$ by two integrals $\int_0^{\lambda-\epsilon}$ and $\int_{\lambda + \epsilon}^{+\infty}$, for some $\epsilon\in (0,\lambda)$. Set
\begin{equation*}
\begin{split}
\mathcal{I}(\epsilon):=&\lambda |\hat{\varphi}(0)|^2 - \epsilon |\hat{\varphi}(\lambda-\epsilon)|^2- \epsilon |\hat{\varphi}(\lambda +\epsilon)|^2 \\
=&  2\mathrm{Re} \left(\int_0^{+\infty} (\widehat{u\varphi})'(\xi)\overline{\hat{\varphi}}(\xi) \mathrm{d}\xi - \int_{\lambda-\epsilon}^{\lambda + \epsilon} (\widehat{u\varphi})'(\xi)\overline{\hat{\varphi}}(\xi) \mathrm{d}\xi\right)- \int_0^{+\infty}|\hat{\varphi}(\xi)|^2  \mathrm{d}\xi + \int_{\lambda-\epsilon}^{\lambda + \epsilon}|\hat{\varphi}(\xi)|^2  \mathrm{d}\xi
\end{split}
\end{equation*}Thanks to the continuity of $\hat{\varphi}$ on $\mathbb{R}_+$, we have $\lambda |\hat{\varphi}(0)|^2=\lim_{\epsilon \to 0^{+}}\mathcal{I}(\epsilon) = -2\pi \| \varphi \|_{L^2(\mathbb{R})}^2$.\\

\noindent When $\lambda=0$, we use the same idea and integrate $(\ref{derivative of Wu inequality})$ over interval $[\epsilon, +\infty)$, for some $\epsilon>0$. Then

\begin{equation*}
\begin{split}
\mathcal{J}(\epsilon):= -   \epsilon |\hat{\varphi}( \epsilon)|^2  = 2\mathrm{Re}  \int_{\epsilon}^{+\infty} (\widehat{u\varphi})'(\xi)\overline{\hat{\varphi}}(\xi) \mathrm{d}\xi  - \int_{\epsilon}^{+\infty}|\hat{\varphi}(\xi)|^2  \mathrm{d}\xi \to 0, 
\end{split}
\end{equation*}as $\epsilon \to 0$. So we always have
\begin{equation}
-2\pi \| \varphi \|_{L^2(\mathbb{R})}^2=\lambda |\hat{\varphi}(0)|^2, \qquad \mathrm{if} \quad \varphi \in \mathrm{Ker}(\lambda-L_u).
\end{equation}As a consequence $L_u$ has only negative eigenvalues, if the real-valued function $u\in L^2(\mathbb{R}, (1+x^2) \mathrm{d}x)$. Finally we use $\widehat{u\varphi}(0)=-\lambda \hat{\varphi}(0)$ to get identity $(\ref{identity of Wu for simplicity of spectrum of Lu})$.  If $\lambda \in \sigma_{\mathrm{pp}}(L_u)$ and $\varphi \in \mathrm{Ker}(\lambda - L_u) \backslash \{0\}$, we want to prove that
\begin{equation}\label{regularity of derivative of hat varphi}
\xi \mapsto (1+|\xi |)\partial_{\xi}\hat{\varphi}(\xi) \in L^2(0, +\infty).
\end{equation}In fact, since $\varphi \in H^1_+ \hookrightarrow L^{\infty}(\mathbb{R})$ and $u \in L^2(\mathbb{R}, (1+x^2)\mathrm{d}x)$, we have $\widehat{u \varphi}=\frac{\hat{u}*\hat{\varphi}}{2\pi} \in H^1(\mathbb{R})$. Formula $(\ref{Fourier modes of Lu phi = lambda phi})$ yields that $\xi \mapsto (|\lambda| + \xi)\hat{\varphi}(\xi) \in L^2(\mathbb{R})$ and we have $\hat{\varphi} \in L^1(\mathbb{R})$. The hypothesis $u \in L^2(\mathbb{R},  x^2 \mathrm{d}x)$ implies that the convolution term $ \widehat{xu}*\hat{\varphi} \in L^2(\mathbb{R})$. Since $\lambda <0$, we obtain $(\ref{regularity of derivative of hat varphi})$ by using formula $(\ref{derivative of g lambda})$.
\end{proof}

\subsection{Conservation laws}\label{subsection of conservation laws}
Proposition $\ref{Conservation law sequence controling h n/2 norm}$ and $\ref{Conservation law of the generating function H lambda}$ are proved in this subsection. We begin with the following proposition.

\begin{prop}
If $u : t\in \mathbb{R} \mapsto u(t) \in H^2(\mathbb{R}, \mathbb{R})$ denotes the unique solution of the BO equation $(\ref{Benjamin Ono equation on the line})$, then we have
\begin{equation}\label{Identity of Pi u along the Benjamin ono equation}
\partial_t \Pi u(t) = B_{u(t)}(\Pi u(t)) + i L_{u(t)}^2(\Pi u(t)) \in L^2_+.
\end{equation}
\end{prop}

\begin{proof}
For every  $u\in H^2(\mathbb{R}, \mathbb{R})$ is real-valued, $B_u$ is a bounded operator on both $L^2_+$ and $H^1_+$, $\Pi u \in \mathbf{D}(L_u)=H^1_+$. We have $\hat{u}(-\xi) = \overline{\hat{u}(\xi)}$, $u = \Pi u + \overline{\Pi u}$ and $|\mathrm{D}| u = \mathrm{D} \Pi u -\mathrm{D}\overline{\Pi u}$. Since $\mathrm{D} \overline{\Pi u} \in L^2_-$, we have $\Pi (\Pi u \mathrm{D} \overline{\Pi u}) = \Pi (u \mathrm{D} \overline{\Pi u})$. Thus the following two formulas hold,
\begin{equation*}
\begin{split}
&B_u(\Pi u) = i(T_{|\mathrm{D}|u} - T_u^2)(\Pi u) = i (\Pi u )(\mathrm{D} \Pi u) - i \Pi (u \mathrm{D}\overline{\Pi u}) - i T_u^2(\Pi u) = \Pi u \partial_x \Pi u - \Pi(u \partial_x \overline{\Pi u}) - i T_u^2(\Pi u),\\
&i L_u^2(\Pi u) = i \mathrm{D}^2 \Pi u - i T_u(\mathrm{D} \Pi u) -i \mathrm{D} \circ T_u(\Pi u) + i T_u^2(\Pi u) = - i \partial_x^2 \Pi u -  T_u(\partial_x \Pi u) - \partial_x [T_u(\Pi u)] + i T_u^2(\Pi u).
\end{split}
\end{equation*}Then we add them together to get the following 
\begin{equation*}
\begin{split}
& B_u(\Pi u)+i L_u^2(\Pi u) = - i \partial_x^2 \Pi u-2 \Pi[\Pi u \partial_x \Pi u + \Pi u\partial_x \overline{\Pi u} + \overline{\Pi u} \partial_x \Pi u]
\end{split}
\end{equation*}Finally we replace $u$ by $u(t)$, where $u:t \in \mathbb{R} \mapsto u(t) \in H^2(\mathbb{R}, \mathbb{R})$ solves equation $(\ref{Benjamin Ono equation on the line})$ to obtain $(\ref{Identity of Pi u along the Benjamin ono equation})$.
\end{proof}

\begin{proof}[Proof of proposition $\ref{Conservation law sequence controling h n/2 norm}$]
It suffices to prove $(\ref{conservation law of En})$ in the case $u_0 \in H^{\infty}(\mathbb{R}, \mathbb{R})$. Then we use the density argument and the continuity of the flow map
\begin{equation*}
u_0 \in H^s(\mathbb{R}) \mapsto u \in C([-T,T]; H^s(\mathbb{R})) \quad \mathrm{with} \quad T>0, \quad s \geq 0,
\end{equation*}in proposition $\ref{GWP for H^s solution of BO eq}$. We choose $u=u(t) \in H^{\infty}(\mathbb{R},\mathbb{R}) = \bigcap_{s \geq 0} H^{s}(\mathbb{R},\mathbb{R})$, so the functions $L_u^n \Pi u$, $\partial_t \Pi u$ and $ \partial_t(L_u^n )\Pi u =[B_u, L_u^n ]\Pi u$ are in $H^{\infty}(\mathbb{R}, \mathbb{C})$. Thus
\begin{equation*}
\partial_t E_n(u) = 2 \mathrm{Re} \langle L_u^n \Pi u, \partial_t \Pi u \rangle_{L^2} +  \langle  \partial_t (L_u^n)  \Pi u, \Pi u \rangle_{L^2}.
\end{equation*}Since $B_u + i L_u^2$ is skew-adjoint, we use formula $(\ref{Identity of Pi u along the Benjamin ono equation})$ to get the following
\begin{equation*}
 2 \mathrm{Re} \langle  L_u^n \Pi u, \partial_t \Pi u \rangle_{L^2}  =  \langle [ L_u^n , B_u + i L_u^2 ] \Pi u,   \Pi u \rangle_{L^2} =  \langle [ L_u^n , B_u] \Pi u,   \Pi u \rangle_{L^2}.
\end{equation*}Since $( L_u^n , B_u)$ is also a Lax pair of the Benjamin--Ono equation $(\ref{Benjamin Ono equation on the line})$, we have
\begin{equation*}
\partial_t E_n(u) =   \langle ( [ L_u^n , B_u  ]+\partial_t (L_u^n) )\Pi u, \partial_t \Pi u \rangle_{L^2} = 0.
\end{equation*}In the case $n=1$, we assume that $u\in H^1(\mathbb{R}, \mathbb{R})$. Since $u=\Pi u + \overline{\Pi u}$, $|\mathrm{D}|u=\mathrm{D}\Pi u - \mathrm{D} \overline{\Pi u}$ and $\int_{\mathbb{R}}(\Pi u)^3 = 0$, we have $\langle |\mathrm{D}| u, u\rangle_{L^2} = 2 \langle\mathrm{D}\Pi u, \Pi u \rangle_{L^2}$ and $\int_{\mathbb{R}}u^3 = 3 \int_{\mathbb{R}}( \Pi u + \overline{\Pi u} )|\Pi u|^2= 3 \int_{\mathbb{R}}u |\Pi u|^2$. In the general case $u \in H^{\frac{1}{2}}(\mathbb{R}, \mathbb{R})$, we use the density argument.
 
\end{proof}

\begin{proof}[Proof of proposition $\ref{Conservation law of the generating function H lambda}$]
It suffices to prove the case $u(0) \in H^{\infty}(\mathbb{R}, \mathbb{R})$ and we use the density argument. Let $u : t \mapsto u(t) \in H^{\infty}(\mathbb{R}, \mathbb{R})$ solve   equation $(\ref{Benjamin Ono equation on the line})$. Since $\|u(t)\|_{L^2}=\|u(0)\|_{L^2}$ by proposition $\ref{Conservation law sequence controling h n/2 norm}$ and $4\lambda >C^2 \|u(0)\|_{L^2}^2$, we have $(\lambda, u(t)) \in \mathcal{X}$, $\partial_t L_{u(t)}  =[B_{u(t)}, L_{u(t)} +\lambda]$ and 
\begin{equation}\label{derivative of H lambda1}
\partial_t \mathcal{H}_{\lambda} (u) = 2 \mathrm{Re} \langle (L_u +\lambda  )^{-1} \Pi u, \partial_t \Pi u \rangle_{L^2}  - \langle (L_u +\lambda  )^{-1} \partial_t L_u (L_u +\lambda  )^{-1} \Pi u, \Pi u \rangle_{L^2}.
\end{equation}Formula $(\ref{Identity of Pi u along the Benjamin ono equation})$ yields that
\begin{equation*}
 2 \mathrm{Re} \langle (L_u +\lambda  )^{-1} \Pi u, \partial_t \Pi u \rangle_{L^2}  =  \langle [(L_u +\lambda  )^{-1}, B_u + i L_u^2 ] \Pi u,   \Pi u \rangle_{L^2} =  \langle [(L_u +\lambda  )^{-1}, B_u] \Pi u,   \Pi u \rangle_{L^2},
\end{equation*}
\begin{equation*}
\begin{split}
     \langle [(L_u +\lambda  )^{-1}, B_u] \Pi u,   \Pi u \rangle_{L^2} 
=& \langle    B_u    \Pi u, (L_u +\lambda  )^{-1}   \Pi u \rangle_{L^2} + \langle  (L_u +\lambda  )  B_u   (L_u +\lambda  )^{-1} \Pi u, (L_u +\lambda  )^{-1}   \Pi u \rangle_{L^2}\\
 = &  \langle (L_u +\lambda  )^{-1} [B_u, L_u +\lambda] (L_u +\lambda  )^{-1} \Pi u,  \Pi u \rangle_{L^2}.
\end{split}
\end{equation*}Then $(\ref{derivative of H lambda1})$ yields that $\partial_t  H_{\lambda}(u(t))=0$. In the general case $u(t) \in L^2(\mathbb{R}, \mathbb{R})$, we proceed as in the proof of  proposition $\ref{Conservation law sequence controling h n/2 norm}$ and use the continuity of the generating functional 
\begin{equation*}
\mathcal{H}_{\lambda} : u \in \{v \in L^2(\mathbb{R}, \mathbb{R}): \|v\|_{L^2}< \tfrac{2\sqrt{\lambda}}{C}\} \mapsto \mathcal{H}_{\lambda}(u) \in \mathbb{R}.
\end{equation*} 
\end{proof}

\subsection{Lax pair formulation}\label{subsection Lax pair formulation}
In this subsection, we prove proposition $\ref{lax pair structure of Hamiltonian equation associated to H lambda}$ and $\ref{Lax pair structure of Benjamin ono equation}$.  The Hankel operators whose symbols are in $L^2(\mathbb{R}) \bigcup L^{\infty}(\mathbb{R})$ will be used to calculate the commutators of Toeplitz operators. We notice that the Hankel operators are $\mathbb{C}$-anti-linear and the Toeplitz operators are $\mathbb{C}$-linear. For every symbol $v\in L^2(\mathbb{R}) \bigcup L^{\infty}(\mathbb{R})$, we define its associated Hankel operator to be $H_v(h)=T_{\overline{h}}v  = \Pi(v \overline{h})$, for every $h\in H^1_+$. If $v \in L^{\infty}(\mathbb{R})$, then $H_v : L^2_+ \to L^2_+$ is a bounded operator. If $v \in L^2(\mathbb{R})$, then $H_v$ may be an unbounded operator on $L^2_+$ whose domain of definition contains $H^1_+$. For every $b\in H^1(\mathbb{R})$, we have  $\|T_b (h)\|_{H^1} + \|H_b (h)\|_{H^1} \lesssim \|b\|_{H^1} \|h\|_{H^1}$, for every $h\in H^1_+$, so both $T_b$ and $H_b$ are bounded on $L^2_+$ and on $H^1_+$.

\begin{lem}
For every $v, w \in L^2_+\bigcap L^{\infty}(\mathbb{R})$ and $u \in L^2(\mathbb{R})$, we have
\begin{equation}\label{commutator of Tv T bar(w)}
[T_v, T_{\overline{w}}] =-H_v \circ H_w \in \mathfrak{B}(L^2_+).
\end{equation}If $w \in H^1_+$ in addition, then we have $T_u(w) \in L^2_+$ and 
\begin{equation}\label{HTu(w) 2 expressions}
H_{T_{u}w}= T_{w} \circ H_{\Pi u} + H_w \circ T_{\overline{u}} = T_u \circ H_w + H_{\Pi u} \circ T_{\overline{w}}\in \mathfrak{B}(H^1_+, L^2_+).
\end{equation}
\end{lem}

\begin{proof}
For every $v, w \in L^2_+\bigcap L^{\infty}(\mathbb{R})$ and $h\in L^2_+$, we have $\overline{w}h = \Pi(\overline{w}h)  + \overline{\Pi(w\overline{h} )} \in L^2_+$. Thus,   
\begin{equation*}
[T_v, T_{\overline{w}}]h=\Pi(v\Pi(\overline{w}h) - \overline{w}\Pi(vh) ) = \Pi( v\overline{w}h - v\overline{\Pi(w\overline{h})}- v\overline{w} h  ) =-\Pi( v\overline{\Pi(w\overline{h})}) = - H_v \circ H_w (h) \in L^2_+.
\end{equation*}Given $u \in L^2(\mathbb{R})$ and  $w \in H^1_+$, for every $h\in H^1_+$, we have  $w \overline{h} =\Pi(w\overline{h} )  + \overline{\Pi( \overline{w}h )} \in H^1(\mathbb{R})$ and  $H_{w}(h), T_{\overline{w}} (h) \in H^1_+$. So $\Pi(u\overline{\Pi( \overline{w} h)}) =\Pi(\overline{\Pi( \overline{w} h)} \Pi u)=H_{\Pi u} \circ T_{\overline{w}}   (h) \in L^2_+$ and  we have
\begin{equation*}
H_{T_{u}w}(h)= \Pi(\Pi(uw)\overline{h}) = \Pi( uw \overline{h}) = \Pi( u \Pi(w \overline{h}) + u\overline{\Pi( \overline{w} h)})   = (T_u \circ H_w + H_{\Pi u} \circ T_{\overline{w}} ) (h)\in L^2_+.
\end{equation*}Similarly, we have $u \overline{h} =\Pi(u\overline{h} )  + \overline{\Pi( \overline{u}h )} \in L^2(\mathbb{R})$ and $\Pi (u \overline{h}) = \Pi(\overline{h}\Pi u) = H_{\Pi u}(h) \in L^2_+$. Thus,
\begin{equation*}
H_{T_{u}w}(h) = \Pi(  w u\overline{h}) = \Pi( w \Pi(u \overline{h}) + w\overline{\Pi( \overline{u} h)})   = (T_w \circ H_{\Pi u} + H_{w} \circ T_{\overline{u}} ) (h)\in L^2_+.
\end{equation*}

\end{proof}

\begin{lem}\label{Lax identity lemma of generating functional}
Given $(\lambda, u) \in \mathcal{X}$ given in definition $\ref{def of generating function}$, set $w_{\lambda}(u)= \left(L_u+\lambda\right)^{-1} \circ \Pi (u) \in H^1_+$, then
\begin{equation}\label{equivalence of lax identity for Hamiltonian equation associated to H lambda}
[\mathrm{D} - T_u, T_{w_\lambda (u)} T_{\overline{w}_\lambda (u)} + T_{w_\lambda (u)} + T_{\overline{w}_\lambda (u)}] = T_{\mathrm{D} [|w_{\lambda}(u)|^2 + w_{\lambda}(u) + \overline{w}_{\lambda}(u)]}\in \mathfrak{B}(H^1_+ , L^2_+).
\end{equation}
\end{lem}

\begin{proof}
We use abbreviation $w_{\lambda}:=w_{\lambda}(u) \in H^1_+$, then $\overline{w}_{\lambda} \in H^1_-$. If $f^+, g^+   \in H^1_+$ and $f^-, g^- \in H^1_-$, then we have $[T_{f^+}, T_{g^+}]=[T_{f^-}, T_{g^-}]=0$, because for every $h \in L^2_+$, we have  
\begin{equation*}
T_{f^+}[T_{g^+}(h)] =  f^+ g^+ h = T_{g^+}[T_{f^+}(h)], \qquad \forall h \in L^2_+.
\end{equation*}and $T_{f^-}[T_{g^-}(h)] =\Pi(f^- \Pi(g^- h )) = \Pi(f^-  g^- h  ) =\Pi(g^- \Pi(f^- h ))   = T_{g^-}[T_{f^-}(h)]$. Since $\Pi u \in L^2_+$ and $\overline{\Pi u} \in L^2_-$, we use  Leibnitz's rule and formula $(\ref{commutator of Tv T bar(w)})$ to obtain that
\begin{equation}\label{formula for commutator Lu,  T w  lambda  + T overline w  lambda} 
\begin{split}
[\mathrm{D} - T_u,  T_{w_\lambda} + T_{\overline{w}_\lambda}] = & T_{\mathrm{D} w_\lambda} + T_{\mathrm{D} \overline{w}_\lambda} -   [T_u,  T_{w_\lambda}] - [T_u, T_{\overline{w}_\lambda}]\\
= &  T_{\mathrm{D} w_\lambda} + T_{\mathrm{D} \overline{w}_\lambda} -   [T_{\overline{\Pi u}},  T_{w_\lambda}] - [T_{\Pi u}, T_{\overline{w}_\lambda}]\\
= &  T_{\mathrm{D} w_\lambda} + T_{\mathrm{D} \overline{w}_\lambda} - H_{w_{\lambda}}H_{\Pi u} + H_{\Pi u}H_{w_{\lambda}}  . \\
\end{split}
\end{equation}Similarly, formula $(\ref{commutator of Tv T bar(w)})$ implies that 
\begin{equation}\label{commutator of Tu (Tw Tw bar)}
\begin{split}
[T_u, T_{w_\lambda} T_{\overline{w}_\lambda}] = &[T_u, T_{w_\lambda}] T_{\overline{w}_\lambda} + T_{w_\lambda} [T_u, T_{\overline{w}_\lambda}]\\
 =&  [T_{\overline{\Pi u}}, T_{w_\lambda}] T_{\overline{w}_\lambda} + T_{w_\lambda} [T_{\Pi u}, T_{\overline{w}_\lambda}]\\
  =&  H_{w_\lambda} H_{\Pi u} T_{\overline{w}_\lambda} - T_{w_\lambda} H_{\Pi u} H_{w_\lambda} . \\
\end{split}
\end{equation}For every $h \in H^1_+$, since $\overline{w}_\lambda, \mathrm{D} \overline{w}_\lambda \in L^2_-$, we have 
\begin{equation*}
\begin{split}
[\mathrm{D},   T_{\overline{w}_\lambda} T_{w_\lambda}]h =&  [\mathrm{D},   T_{\overline{w}_\lambda}] T_{w_\lambda}h  +T_{\overline{w}_\lambda}[\mathrm{D},    T_{w_\lambda}]  h \\
=& T_{\mathrm{D} \overline{w}_\lambda}  (T_{w_\lambda}h ) +T_{\overline{w}_\lambda} (T_{\mathrm{D} w_\lambda}h) \\
= & \Pi [\mathrm{D} \overline{w}_\lambda \Pi(w_\lambda h) + \overline{w}_\lambda \Pi(\mathrm{D} w_\lambda h) ]   
=  \Pi [ (w_\lambda \mathrm{D} \overline{w}_\lambda  +\overline{w}_\lambda \mathrm{D} w_\lambda  )h  ] \in L^2_+.
\end{split}
\end{equation*}So $[\mathrm{D},   T_{\overline{w}_\lambda} T_{w_\lambda}] = T_{\mathrm{D}  |w _\lambda|^2 } \in \mathfrak{B}(H^1_+, L^2_+)$. We use formula $(\ref{commutator of Tv T bar(w)})$ and Leibnitz's Rule to obtain that
\begin{equation}\label{leibnitz rule for hankel operators}
[\mathrm{D},    T_{w_\lambda}T_{\overline{w}_\lambda}] = [\mathrm{D},   T_{\overline{w}_\lambda} T_{w_\lambda}] - [\mathrm{D}, H_{w_{\lambda}}^2] =  T_{\mathrm{D}   |w_\lambda|^2 } - H_{\mathrm{D} {w_\lambda}} H_{w_\lambda} +H_{w_\lambda} H_{\mathrm{D} {w_\lambda}}
\end{equation}Recall that $w_\lambda = (\lambda + L_u)^{-1} \Pi u$, then we have
 \begin{equation}\label{formula D w lambda}
 \mathrm{D}w_{\lambda} = T_u (w_{\lambda}) -\lambda w_{\lambda} + \Pi u.
 \end{equation}The formula $(\ref{HTu(w) 2 expressions})$ and $(\ref{formula D w lambda})$ yield that
 \begin{equation}\label{difference H Dw - Tw Hpiu}
 H_{\mathrm{D} w_{\lambda}} - T_{w_{\lambda}} H_{\Pi u} = H_{T_u  w_{\lambda} } -\lambda H_{w_{\lambda}} + H_{\Pi u} - T_{w_{\lambda}} H_{\Pi u} = H_{w_{\lambda}} T_u  -\lambda H_{w_{\lambda}} + H_{\Pi u} 
 \end{equation}and
  \begin{equation}\label{difference H Dw - Hpiu Tw bar}
  H_{\mathrm{D} w_{\lambda}} -H_{\Pi u} T_{\overline{w}_{\lambda}} =H_{T_u  w_{\lambda} } -\lambda H_{w_{\lambda}} + H_{\Pi u} - H_{\Pi u}  T_{\overline{w}_{\lambda}} =T_u H_{w_{\lambda}}-\lambda H_{w_{\lambda}} + H_{\Pi u}.
  \end{equation}We use formulas $(\ref{commutator of Tu (Tw Tw bar)})$, $(\ref{leibnitz rule for hankel operators})$, $(\ref{difference H Dw - Tw Hpiu})$ and $(\ref{difference H Dw - Hpiu Tw bar})$ to get the following formula
  \begin{equation}\label{commutator of Lu and Tw Tw bar}
\begin{split}
& [\mathrm{D} - T_u , T_{w_\lambda}T_{\overline{w}_\lambda}]\\
 = &  T_{\mathrm{D}  |w_\lambda|^2 } -( H_{\mathrm{D} {w_\lambda}}-T_{w_\lambda} H_{\Pi u}) H_{w_\lambda} +H_{w_{\lambda}} (H_{\mathrm{D} {w_\lambda}} -  H_{\Pi u} T_{\overline{w}_\lambda})\\
= &  T_{\mathrm{D}  |w_\lambda|^2 } -(H_{w_{\lambda}} T_u H_{w_\lambda} -\lambda H_{w_{\lambda}}^2 + H_{\Pi u}H_{w_\lambda} ) + (H_{w_{\lambda}} T_u H_{w_{\lambda}}-\lambda H_{w_{\lambda}}^2 + H_{w_{\lambda}} H_{\Pi u})\\
= & T_{\mathrm{D}  |w_\lambda|^2 } - H_{\Pi u}H_{w_\lambda} + H_{w_{\lambda}} H_{\Pi u}
\end{split}
  \end{equation}At last, we combine formulas $(\ref{formula for commutator Lu,  T w  lambda  + T overline w  lambda})$ and $(\ref{commutator of Lu and Tw Tw bar})$ to obtain formula $(\ref{equivalence of lax identity for Hamiltonian equation associated to H lambda})$.
\end{proof}

\begin{proof}[End of the proof of proposition $\ref{lax pair structure of Hamiltonian equation associated to H lambda}$]
Since $L: u \in L^2(\mathbb{R}, \mathbb{R}) \mapsto L_u = \mathrm{D}- T_u \in \mathfrak{B}(H^1_+ , L^2_+)$ is $\mathbb{R}$-affine, for every $u \in L^2_+$, we have
\begin{equation*}
 \frac{\mathrm{d}}{\mathrm{d}t} (L\circ u)(t) = - T_{\partial_t u(t)} = -i T_{\mathrm{D}\left( w_{\lambda}(u(t)) \overline{w}_{\lambda}(u(t))  + w_{\lambda}(u(t)) +  \overline{w}_{\lambda}(u(t))  \right)}.
\end{equation*}Thus the Lax equation $(\ref{Lax equation of Hamiltonian equation associated to H lambda})$ is equivalent to identity $(\ref{equivalence of lax identity for Hamiltonian equation associated to H lambda})$ in lemma $\ref{Lax identity lemma of generating functional}$. 
\end{proof} 

\bigskip

\noindent The proof of proposition $\ref{Lax pair structure of Benjamin ono equation}$ can be found in G\'erard--Kappeler $[\ref{Gerard kappeler Benjamin ono birkhoff coordinates}]$, Wu $[\ref{Wu Simplicity and finiteness of discrete spectrum of the Benjamin--Ono}]$ etc. In order to make this paper self contained, we recall it here.

\begin{proof}[Proof of proposition $\ref{Lax pair structure of Benjamin ono equation}$]
Since the Lax map $L : u \in H^2(\mathbb{R}, \mathbb{R}) \mapsto \mathrm{D}- T_u \in \mathfrak{B}(H^1_+, L^2_+)$ is $\mathbb{R}$-affine,  
\begin{equation*}
\frac{\mathrm{d}}{\mathrm{d}t}(L\circ u)(t) = -T_{\partial_t u(t)} = -T_{ \mathrm{H}\partial_x^2 u(t)- \partial_x \left(u(t)^2 \right)}.
\end{equation*}It suffices to prove $[B_u, L_u]+T_{ \mathrm{H}\partial_x^2 u - \partial_x \left(u^2 \right)}=0$ for every $u \in H^2(\mathbb{R}, \mathbb{R})$. \\

\noindent In fact, $u$ is real-valued, we have $\hat{u}(-\xi) = \overline{\hat{u}(\xi)}$, $u = \Pi u + \overline{\Pi u}$ and $|\mathrm{D}| u = \mathrm{D} \Pi u - \mathrm{D} \overline{\Pi u}$. Since both $T_u$ and $B_u$ are bounded operators $L^2_+ \to L^2_+$ and bounded operators $H^1_+ \to H^1_+$ , their Lie Bracket $[B_u, L_u]$ is given by
\begin{equation}\label{Bu Lu - Lu Bu}
\begin{split}
[B_u, L_u]f = &  - \Pi(f \partial_x |\mathrm{D}| u) + i \Pi[u\Pi(f |\mathrm{D}|u) - |\mathrm{D}|u \Pi(uf)] + \Pi[\partial_x u \Pi(uf) + u \Pi(f \partial_x u)] \\
= & -\Pi(f H \partial_x^2 u) + \mathcal{I}_1  + \mathcal{I}_2 \in L^2_+, 
\end{split}
\end{equation}for every $f \in H^1_+$, where the terms $\mathcal{I}_1$ and $\mathcal{I}_2$ are given by
\begin{equation*}
\begin{split}
 \mathcal{I}_1 
 := & i \Pi[u\Pi(f |\mathrm{D}|u) - |\mathrm{D}|u \Pi(uf)]  \\
= & \Pi[f  \overline{\Pi u}  \partial_x \Pi u + f \Pi u \partial_x \overline{\Pi u}] - \Pi u \Pi(f \partial_x \overline{\Pi u}) -\Pi(f \overline{\Pi u}) \partial_x \Pi u  + \Pi[\Pi(f \overline{\Pi u})\partial_x \overline{\Pi u} - \overline{\Pi u}\Pi(f \partial_x \overline{\Pi u})],
\end{split}
\end{equation*}
\begin{equation*}
\begin{split}
 \mathcal{I}_2 
:=& \Pi[\partial_x u \Pi(uf) + u \Pi(f \partial_x u)] 
=   \Pi(f \overline{\Pi u}) \partial_x \Pi u+ \Pi u \Pi(f \partial_x \overline{\Pi u}) + \Pi(\overline{\Pi u} \Pi(f \partial_x \overline{\Pi u}))  \\
 &  \qquad \qquad  \qquad  \qquad \qquad \qquad \qquad + 2 f \Pi u \partial_x \Pi u + \Pi[f  \Pi u \partial_x \overline{\Pi u} + f \overline{\Pi u} \partial_x \Pi u +  \Pi(f \overline{\Pi u}) \partial_x \overline{\Pi u}].
\end{split}
\end{equation*}If $h_1 \in H^1_-$ and $h_2 \in L^2_-$, then $h_1 h_2 \in L^2_-$. Since $\partial_x \overline{\Pi u} \in L^2_-$, we have $\Pi[ \Pi(f \overline{\Pi u}) \partial_x \overline{\Pi u}] = \Pi[  f \overline{\Pi u}  \partial_x \overline{\Pi u}]$. Thus
\begin{equation}\label{I1 + I2}
 \mathcal{I}_1  + \mathcal{I}_2 
= 2 f \Pi u \partial_x \Pi u + 2\Pi[f  \Pi u \partial_x \overline{\Pi u} + f \overline{\Pi u} \partial_x \Pi u + \Pi(f \overline{\Pi u})\partial_x \overline{\Pi u} ]  
 =    \Pi[f\partial_x (u^2)]   \in H^1_+.
 \end{equation}Formulas $(\ref{Bu Lu - Lu Bu})$ and $(\ref{I1 + I2})$ yield that $[B_u, L_u]f =  \Pi[f(\partial_x( u^2) - \mathrm{H} \partial_x^2 u)]$. Thus equation $(\ref{Lax equation of Lu Bu})$ holds along the evolution of equation $(\ref{Benjamin Ono equation on the line})$.
 \end{proof}
\begin{rem}
As indicated in G\'erard--Kappeler $[\ref{Gerard kappeler Benjamin ono birkhoff coordinates}]$, there are many choices of the operator $B_u$. We can replace $B_u$ by any operator of the form $B_u + P_u$ such that $P_u$ is a skew-adjoint operator commuting with $L_u$. For instance, we set $C_u:=B_u+ i L_u^2$ and we obtain $C_u = i\mathrm{D}^2 - 2i \mathrm{D}  T_u +2i T_{\mathrm{D} \Pi u}$. So $(L_u, C_u)$ is also a Lax pair of the BO equation $(\ref{Benjamin Ono equation on the line})$. The advantage of the operator $B_u = i(T_{|\mathrm{D} |u}- T_u^2)$ is that $B_u : L^2_+ \to L^2_+$ is bounded if $u$ is sufficiently regular. For instance, $u \in H^2(\mathbb{R}, \mathbb{R})$.
\end{rem}

\bigskip
\bigskip

\section{The action of the shift semigroup}\label{section infinitesimal generators}
In this section, we introduce the semigroup of shift operators $(S(\eta)^*)_{\eta \geq 0}$ acting on the Hardy space $L^2_+$ and classify all finite-dimensional translation-invariant subspaces of $L^2_+$.\\

\noindent For every $\eta \geq 0$, we define the operator $S(\eta): L^2_+ \to L^2_+$ such that $S(\eta)f = e_{\eta}f$, where $e_{\eta}(x)=e^{i\eta x}$. Its adjoint is given by $S(\eta)^* = T_{e_{-\eta}}$. We have 
\begin{equation*}
 S(\eta)^* \circ L_u \circ S(\eta) = L_u + \eta \mathrm{Id}_{L^2_+},  \qquad \forall \eta \geq 0.
\end{equation*}Since $\|S(\eta)^*\|_{\mathfrak{B}(L^2_+)}=\|S(\eta)\|_{\mathfrak{B}(L^2_+)}=1$, $(S(\eta)^*)_{\eta \geq 0}$ is a contraction semi-group. Let $-i G$ be its infinitesimal generator, i.e. $Gf = i \frac{\mathrm{d}}{\mathrm{d}\eta}  \Big|_{\eta=0^+} S(\eta)^* f \in L^2_+$, $\forall f \in  \mathbf{D}(G)$, where
\begin{equation}\label{domaine of G}
\begin{split}
\mathbf{D}(G):= & \{f \in L^2_+ : \hat{f}_{|\mathbb{R}_+} \in H^1(0, +\infty)\},
\end{split}
\end{equation}because $\lim_{\epsilon \to 0} \|\frac{\psi -\tau_{\epsilon}\psi}{\epsilon} - \partial_x \psi\|_{L^2(0, +\infty)} =0$, where $\tau_{\epsilon}\psi(x)= \psi(x-\epsilon)$ and $ \psi \in H^1(0, +\infty)$. Every function $f \in  \mathbf{D}(G)$ has bounded H\"older continuous Fourier transform by Morrey's inequality and Sobolev extension operator yields the existence of $\hat{f}(0^+):=\lim_{\xi \to 0^+}\hat{f}(\xi)$.  The operator $G$ is densely defined and closed. The Fourier transform of $Gf$ is given by  
\begin{equation}\label{definition of G}
 \widehat{Gf}(\xi) = i \partial_{\xi}\hat{f}(\xi) , \qquad \forall f \in \mathbf{D}(G), \qquad \forall \xi >0.
\end{equation}In accordance with the Hille--Yosida theorem, we have
\begin{equation}\label{Hille Yosida theorem}
(-\infty, 0) \subset \rho(iG),  \qquad \|(G-\lambda i)^{-1}\|_{\mathfrak{B}(L^2_+)} \leq \lambda^{-1}, \quad \forall \lambda>0.
\end{equation}
\bigskip

\begin{lem}\label{Domaine of G is invariant by Tb}
For every $b \in L^2(\mathbb{R})\bigcap L^{\infty}(\mathbb{R})$, we have $T_b (\mathbf{D}(G)) \subset \mathbf{D}(G)$ and the following identity 
\begin{equation}\label{formula of commutator of G T_b}
[G, T_{b}] \varphi = \frac{i\hat{\varphi}(0^+)}{2\pi}\Pi b 
\end{equation}holds for every $\varphi\in \mathbf{D}(G)$. 
\end{lem}

\begin{proof}
For every $\eta>0$ and $\varphi\in \mathbf{D}(G)$, both $S(\eta)^*$ and $T_b$ are bounded operators, so we have
\begin{equation*}
\left([\tfrac{S(\eta)^*}{\eta}, T_b]\varphi \right)^{\wedge}(\xi)  = \frac{\hat{b}*\hat{\varphi}(\xi+\eta) - \hat{b}*[\mathbf{1}_{\mathbb{R}_+}( \tau_{-\eta}\hat{\varphi})](\xi)}{2\pi \eta}  = \frac{1}{2\pi \eta} \int_{\xi}^{\xi +\eta} \hat{b}(\zeta) \hat{\varphi}(\xi + \eta - \zeta) \mathrm{d}\zeta, \quad \forall \xi >0,
\end{equation*}where $\tau_{-\eta}\hat{\varphi} (x)=\hat{\varphi} (x+\eta)$, for every $x \in \mathbb{R}$. Then we change the variable $\zeta = \xi + t \eta$, for $0\leq t \leq 1$, 
\begin{equation}\label{quotient of difference shift operator}
\left([\tfrac{S(\eta)^* - \mathrm{Id}_{L^2_+}}{\eta}, T_b]\varphi \right)^{\wedge}(\xi)= \frac{1}{2\pi  } \int_{0}^{1} \hat{b}(\xi + t\eta) \hat{\varphi}((1-t)\eta  ) \mathrm{d}\zeta=  a_{\eta}\widehat{  b}(\xi) + \widehat{\phi_{\eta}}(\xi), \quad \forall \xi >0,
\end{equation}where $a_{\eta} :=  \frac{1}{2\pi  } \int_{0}^{1}   \hat{\varphi}((1-t)\eta  ) \mathrm{d}\zeta \in \mathbb{C}$ and $\phi_{\eta} \in L^2_+$ such that 
\begin{equation*}
\widehat{\phi_{\eta}}(\xi) := \frac{1}{2\pi  } \int_{0}^{1} [\hat{b}(\xi + t\eta) - \hat{b}(\xi)]\hat{\varphi}((1-t)\eta  ) \mathrm{d}t, \qquad  \forall \xi > 0.
\end{equation*}Since $\hat{\varphi}|_{\mathbb{R}_+} \in H^1(0, +\infty)$, $\hat{\varphi}$ is bounded and $\lim_{\eta \to 0^+}\hat{\varphi}(\eta) = \hat{\varphi}(0^+)$, Lebesgue's dominated convergence theorem yields that $\lim_{\eta \to 0^+} a_{\eta} = \frac{\hat{\varphi}(0^+)}{2\pi}$. Since $b \in L^2(\mathbb{R})$, we have $\lim_{\epsilon\to 0} \|\tau_{\epsilon}\hat{b} - \hat{b} \|_{L^2} = 0$. By using Cauchy--Schwarz inequality and Fubini's theorem, we have
\begin{equation*}
\|\phi_{\eta}\|_{L^2}^2 \lesssim \|\hat{\varphi}\|_{L^{\infty}}^2\int_{0}^{1} \int_{0}^{+\infty}|\hat{b}(\xi + t\eta) - \hat{b}(\xi)|^2\mathrm{d}\xi  \mathrm{d}t = \|\hat{\varphi}\|_{L^{\infty}}^2\int_{0}^{1} \|\tau_{-t\eta}\hat{b} - \hat{b} \|_{L^2}^2\mathrm{d}t \to 0,
\end{equation*}when $\eta \to 0^+$, by Lebesgue's dominated convergence theorem. Thus $(\ref{quotient of difference shift operator})$ implies that
\begin{equation*}
  [\tfrac{S(\eta)^* - \mathrm{Id}_{L^2_+}}{\eta}, T_b]\varphi  =  a_{\eta} \Pi b + \phi_{\eta} \to \frac{\hat{\varphi}(0^+)}{2\pi} \Pi b, \qquad \mathrm{in} \quad L^2_+, \qquad \mathrm{when} \quad \eta \to 0^+.
\end{equation*}Since $\varphi \in \mathbf{D}(G)$ and $T_b$ is bounded, we have $\tfrac{1}{\eta} T_b [ ( S(\eta)^* - \mathrm{Id}_{L^2_+} )\varphi ] \to ( T_b G )\varphi$ in $L^2_+$, consequently 
\begin{equation*}
\tfrac{1}{\eta}   ( S(\eta)^* - \mathrm{Id}_{L^2_+} )(T_b \varphi) \to ( T_b G )\varphi + \frac{\hat{\varphi}(0^+)}{2\pi} \qquad \mathrm{in} \quad L^2_+, \qquad \mathrm{when} \quad \eta \to 0^+.  
\end{equation*}So $T_b \varphi \in \mathbf{D}(G)$ and $(\ref{formula of commutator of G T_b})$ holds.
\end{proof}

\bigskip

\noindent The following scalar representation theorem of Lax $[\ref{Lax Translation invariant spaces}]$ allows to classify all translation-invariant subspaces of the Hardy space $L^2_+$, which plays the same role as Beurling's theorem in the case of Hardy space on the circle (see Theorem 17.21 of Rudin $[\ref{Rudin Real and complex analysis}]$).
\begin{thm}[Beurling--Lax]\label{Beurling Lax thm}
Every nonempty closed subspace of $L^2_+$ that is invariant under the semigroup of shift operators $(S(\eta))_{\eta \geq 0}$ is of the form $\Theta  L^2_+$, where $\Theta$ is a holomorphic function in the upper-half plane $\mathbb{C}_+ = \{z \in \mathbb{C} : \mathrm{Im}z >0\}$. We have $|\Theta(z)|\leq 1$, for all $z \in \mathbb{C}_+$  and $|\Theta(x)|=1$, $\forall x \in \mathbb{R}$. Moreover, $\Theta$ is uniquely determined up to multiplication by
a complex constant of absolute value $1$.
\end{thm}

\noindent The following lemma classifies all finite-dimensional subspaces that are invariant under the semi-group $(S(\eta)^*)_{\eta \geq 0}$, which is a weak version of  theorem $\ref{Beurling Lax thm}$.

\begin{lem}\label{lemma of caracterization of G invariant subspace of finite dimension}
Let $M$ be a subspace of $L^2_+$ of finite dimension $N =\dim_{\mathbb{C}} M \geq 1$ and $G(M) \subset M$. Then there exists a unique monic polynomial $Q \in \mathbb{C}_{N}[X]$ such that $Q^{-1}(0) \subset \mathbb{C}_-$ and $M=\frac{\mathbb{C}_{\leq N-1}[X]}{Q}$, where $\mathbb{C}_{\leq N-1}[X]$ denotes all the polynomials whose degrees are at most $N-1$. $Q$ is the characteristic polynomial of the operator $G|_{M}$. 
\end{lem}

\begin{proof}
We set $\hat{M}=\{\hat{f} \in L^2(0,+\infty) : f \in M\}$, then  $\dim_{\mathbb{C}}\hat{M}=N$. Since $\widehat{Gf}=i\partial_{\xi}\hat{f}$ on $\mathbb{R}\backslash \{0\}$, the restriction $G|_M$ is unitarily equivalent to $i\partial_{\xi}|_{\hat{M}}$ by the renormalized Fourier--Plancherel transformation. So the characteristic polynomial $Q \in \mathbb{C}_N [X]$  of $i\partial_{\xi}|_{\hat{M}}$ is well defined, let $\{\overline{\beta}_1, \overline{\beta}_2, \cdots, \overline{\beta}_n\} \subset \mathbb{C}$ denote the distinct roots of $Q$ and $m_j$ denote the multiplicity of $\overline{\beta}_j$, we have $\sum_{j=1}^n m_j=N$ and
\begin{equation*}
Q(z) = \det(z - i \partial_{\xi}|_{\hat{M}})= \prod_{j=1}^n (z-\overline{\beta}_j)^{m_j} = z^d + \sum_{k=0}^{N-1}c_k z^k, \qquad c_k \in \mathbb{C}.
\end{equation*}The Cayley--Hamilton theorem implies that $Q(i\partial_{\xi}  ) = 0$ on the subspace $\hat{M}$. If $\psi \in \hat{M} \subset L^2(0,+\infty)$, then $\psi$ is a weak-solution of the following differential equation
\begin{equation}\label{differential eq associated to G restrict in M}
i^{-N}Q(-\mathrm{D})\psi=\partial_{\xi}^N \psi + \sum_{k=0}^{N-1}  i^{k-N}c_k \partial_{\xi}^k \psi =0  \quad \mathrm{on} \quad (0,+\infty), \qquad \psi \equiv 0 \quad \mathrm{on}  \quad (-\infty, 0).
\end{equation}The differential operator $Q(-\mathrm{D})$ is elliptic is on the open interval $(0,+\infty)$ in the following sense: the symbol of the principal part of $Q(-\mathrm{D})$, denoted by $\mathit{a}_{Q} :  (x,\xi) \in (0,+\infty)\times \mathbb{R}  \mapsto (-\xi)^N$, does not vanish except for $\xi=0$. Theorem $8.12$ of Rudin $[\ref{Rudin Functional Analysis}]$ yields that $\psi$ is a smooth function. The solution space  
\begin{equation}\label{fourier transform of  basis sol of differential eq}
\mathrm{Sol}(\ref{differential eq associated to G restrict in M}) =\mathrm{Span}_{\mathbb{C}}\{\hat{f}_{j,l}\}_{0\leq l \leq m_j-1,  1\leq j \leq n}, \qquad \hat{f}_{j,l}(\xi)= \xi^l e^{-i \overline{\beta}_j \xi} \mathbf{1}_{\mathbb{R}_+}.
\end{equation}has complex dimension  $\sum_{j=1}^n m_j=N$ so we have  $\mathrm{Sol}(\ref{differential eq associated to G restrict in M})=\hat{M} \subset L^2_+$ and $\mathrm{Im}\beta_j = \mathrm{Re}(i \overline{\beta}_j) >0$ and $Q^{-1}(0) \subset \mathbb{C}_-$. At last, we have $M=\mathrm{Span}_{\mathbb{C}}\{f_{j,l}\}_{0\leq l \leq m_j-1,  1\leq j \leq n} =\frac{\mathbb{C}_{\leq N-1}[X]}{Q} $, where
\begin{equation}\label{definition of basis sol of differential eq}
f_{j,l}(x)= \frac{l !}{2\pi [(-i)(x-\overline{\beta}_j)]^{l+1}}, \quad \forall x \in \mathbb{R}.
\end{equation}The uniqueness is obtained by identifying all the roots.
\end{proof}

\bigskip
\bigskip

\section{The manifold of multi-solitons}\label{section of Manifold}
This section is dedicated to a geometric description of the multi-soliton subsets in definition $\ref{definition of The N soliton in introduction}$. We give at first a polynomial characterization then a spectral characterization for the real analytic symplectic manifold of $N$-solitons in order to prove the global well-posedness of the BO equation with $N$-soliton solutions $(\ref{Hamiltonian energy of BO eq and Hamiltonian form of BO})$.\\

\noindent Recall that every $N$-soliton has the form $u(x)= \sum_{j=1}^N\mathcal{R}_{\eta_j^{-1}}(x-x_j)=\sum_{j=1}^N \frac{2\eta_j}{(x-x_j)^2 + \eta_j^2}$ with $x_j\in \mathbb{R}$ and $\eta_j >0$, then we have the following polynomial characterization of the $N$-solitons.
\begin{prop}\label{Polynomial characterization of N soliton}
The $N$-soliton subset $\mathcal{U}_N   \subset H^{\infty}(\mathbb{R}, \mathbb{R}) \bigcap L^2(\mathbb{R}, x^2\mathrm{d}x)$ and $\mathcal{U}_N \bigcap \mathcal{U}_M = \emptyset$, for every $M\ne N$. Moreover,   each of the following three properties implies the others: \\

\noindent $\mathrm{(a)}$. $u \in \mathcal{U}_N$.\\
\noindent $\mathrm{(b)}$. There exists a unique monic polynomial $Q_u \in \mathbb{C}_N[X]$ whose roots are contained in the lower half-plane $\mathbb{C}_-$ such that $\Pi u = i \frac{Q_u'}{Q_u}$. \\
\noindent $\mathrm{(c)}$. There exists $Q \in \mathbb{C}_N[X]$ such that $Q^{-1}(0) \subset \mathbb{C}_-$ and $\Pi u = i \frac{Q'}{Q}$.  
\end{prop}
\begin{proof}
We only prove the uniqueness in $\mathrm{(a)} \Rightarrow  \mathrm{(b)}$. If $\Pi u = i \frac{Q_u'}{Q_u} = i\frac{P'}{P}$, then we have $\left(\frac{P}{Q_u} \right)' \equiv 0$ on $\mathbb{R}$. Since $P$ and $Q_u$ are monic polynomials, we have $P=Q_u$.  The other assertions are consequences of $u=\Pi u+ \overline{\Pi u}$.
\end{proof}

\begin{defi}\label{definition of char poly of soliton and translation scaling parameters}
For every $u\in \mathcal{U}_N$, the unique monic polynomial $Q_u \in \mathbb{C}_N[X]$ given by proposition $\ref{Polynomial characterization of N soliton}$ is called the characteristic polynomial of $u$. Its roots  are denoted by $z_j=x_j-i \eta_j \in \mathbb{C}_-$, for $1\leq j\leq N$ (not necessarily all distinct). The unordered $N$-uplet $\mathbf{cl}(z_1, z_2, \cdots, z_N) \in \mathbb{C}^N_- \slash \mathrm{S}_N$ is called the translation--scaling parameters of $u$, where $\mathbb{C}^N_- \slash \mathrm{S}_N$ denotes the orbit space of the action $(\ref{action of SN on CN})$ of symmetric group $\mathrm{S}_N$ on $\mathbb{C}^N_-$.
\end{defi}

\noindent The real analytic structure of $\mathcal{U}_N$ is given in the next proposition.
\begin{prop}\label{prop differential structure of UN}
Equipped with the subspace topology of $L^2(\mathbb{R}, \mathbb{R})$, the subset $\mathcal{U}_N$ is a connected, real analytic,  embedded submanifold of the $\mathbb{R}$-Hilbert space $L^2(\mathbb{R}, \mathbb{R})$ and $\dim_{\mathbb{R}}\mathcal{U}_N =2N$. For every $u\in \mathcal{U}_N$, its translation--scaling parameters are denoted by $\mathbf{cl}(x_1-i\eta_1, x_2-i\eta_2, \cdots, x_N-i\eta_N)$ for some  $x_j  \in \mathbb{R}$ and $\eta_j >0$, then the tangent space to $\mathcal{U}_N$ at $u$ is given by  
\begin{equation}\label{tangent space of u in UN}
\mathcal{T}_u (\mathcal{U}_N) = \bigoplus_{j=1}^N (\mathbb{R}f_j^u\bigoplus \mathbb{R}g_j^u), \qquad \mathrm{where}\quad f_{j}^u(x)=  \tfrac{2[(x-x_j)^2 - \eta_j^2]}{[(x-x_j)^2 + \eta_j^2]^2}, \quad g_j^u(x)= \tfrac{4 \eta_j (x-x_j) }{[(x-x_j)^2 + \eta_j^2]^2}.
\end{equation}
\end{prop}
\noindent Every tangent space $\mathcal{T}_u (\mathcal{U}_N)$ is contained in the auxiliary space $\mathcal{T}$ defined by $(\ref{Auxiliary space containing all tangent spaces})$ in which the global $2$-covector $\boldsymbol{\omega} \in \boldsymbol{\Lambda}^2(\mathcal{T}^*)$ is well defined. Recall that the
nondegenerate $2$-form $\omega$ on $\mathcal{U}_N$ is given by
\begin{equation}\label{definition of the symplectic form omega}
 \omega_u (h_1, h_2)  = \boldsymbol{\omega}(h_1, h_2)=\frac{i}{2\pi} \int_{\mathbb{R}} \frac{\hat{h}_1(\xi) \overline{\hat{h}_2(\xi)}}{\xi}\mathrm{d}\xi, \quad \forall h_1, h_2 \in \mathcal{T}_u(\mathcal{U}_N).
\end{equation}It provides the symplectic structure of the manifold $\mathcal{U}_N$.

\begin{prop}\label{prop omega is closed symplectic manifold UN}
The nondegenerate real analytic  $2$-form $\omega$ is closed on $\mathcal{U}_N$. Endowed with the symplectic form $\omega$, the real analytic manifold $(\mathcal{U}_N, \omega)$ is a symplectic manifold. 
\end{prop}

\noindent For every smooth real-valued function $f : \mathcal{U}_N \to \mathbb{R}$, let $X_f \in \mathfrak{X}(\mathcal{U}_N)$ denote its Hamiltonian vector field, defined as follows: for every $u \in \mathcal{U}_N$ and $h \in \mathcal{T}_u(\mathcal{U}_N)$, 
\begin{equation*}
\mathrm{d}f(u)(h)=  \langle h, \nabla_u f(u) \rangle_{L^2} =  \frac{i}{2\pi} \int_0^{+\infty}\frac{\hat{h}(\xi)}{\xi} \overline{i \xi (\nabla_u f(u))^{\wedge}(\xi)}\mathrm{d}\xi  = \omega_u (h, X_{f}(u)).
\end{equation*}Then we have 
\begin{equation}\label{formula to obtain relation between XF and nabla F}
X_{f}(u) = \partial_x \nabla_u f(u) \in \mathcal{T}_u (\mathcal{U}_N), \qquad \forall u \in  \mathcal{U}_N.
\end{equation}

\begin{rem}
There are several ways to prove the simple connectedness of $\mathcal{U}_N$. Firstly, it is irrelevant to the proof of proposition $\ref{prop of real bi analytic diffeo}$. In subsection $\mathbf{\ref{subsection diffeo property}}$, we show that the real analytic manifold $\mathcal{U}_N$ is diffeomorphic to some open convex subset of $\mathbb{R}^{2N}$, hence $\mathcal{U}_N$ is homotopy equivalent to a one-point space. On the other hand, the simple connectedness of the K\"ahler manifold $\Pi(\mathcal{U}_N)$ can be directly obtained from its construction (see proposition $\ref{simply connected property for viete map}$). 
\end{rem}

\noindent Then, we return back to spectral analysis in order to establish a spectral characterization of the manifold $\mathcal{U}_N$. For every monic polynomial $Q \in \mathbb{C}_N [X]$ with roots in $\mathbb{C}_-$, we set $\Theta=\Theta_{Q} :=\frac{\overline{Q}}{Q} \in \mathrm{Hol}(\mathbb{C}_+)$, where
\begin{equation*}
\overline{Q}(x):=\sum_{j=0}^{N-1}\overline{a}_j x^j + x^N  , \qquad \mathrm{if} \qquad Q(x)=\sum_{j=0}^{N-1}a_j x^j + x^N.
\end{equation*}Then $\Theta$ is an inner function on the upper half-plane $\mathbb{C}_+$, because $|\Theta | \leq 1$ on $\mathbb{C}_+$ and $|\Theta| = 1$ on $\mathbb{R}$. Recall the shift operator $S(\eta) : L^2_+ \to L^2_+$ defined in section $\mathbf{\ref{section infinitesimal generators}}$, we have $S(\eta)[ \Theta h]= \Theta [S(\eta) h]$, for every $h \in L^2_+$, so $\Theta L^2_+$ is a closed subspace of $L^2_+$ that is invariant by  the semigroup $(S(\eta))_{\eta\geq 0}$ (see also the Beurling--Lax theorem $\ref{Beurling Lax thm}$ of the complete classification of the translation-invariant subspaces of the Hardy space $L^2_+$). We define $K_{\Theta}$ to be the orthogonal complement of $\Theta L^2_+$,  thus
\begin{equation}\label{k  Theta invariant by G}
 L^2_+ =  \Theta L^2_+ \bigoplus K_{\Theta}, \quad  S(\eta)^* ( K_{\Theta}) \subset K_{\Theta}  \quad \mathrm{and} \quad G(\mathbf{D}(G) \bigcap K_{\Theta}) \subset K_{\Theta}.
\end{equation}where the infinitesimal generator $G$ is defined in $(\ref{definition of G})$. Recall that the $\mathbb{C}$-vector space $\mathbb{C}_{\leq N-1}[X]$ consists of all polynomials with complex coefficients of degree at most $N-1$. So $\frac{\mathbb{C}_{\leq N-1}[X]}{Q}$ is an $N$-dimensional subspace of $L^2_+$. \\

\noindent The Lax map $L: u\in L^2(\mathbb{R}, \mathbb{R}) \mapsto L_u =\mathrm{D}-T_u \in\mathfrak{B}(H^1_+, L^2_+)$ is $\mathbb{R}$-affine. Defined on $\mathbf{D}(L_u)=H^1_+$, the unbounded self-adjoint operator $L_u$ has the following spectral decomposition 
\begin{equation}\label{Spectral decomposition of Lu}
L^2_+ = \mathscr{H}_{\mathrm{ac}}(L_u) \bigoplus \mathscr{H}_{\mathrm{sc}}(L_u) \bigoplus \mathscr{H}_{\mathrm{pp}}(L_u).
\end{equation}

\noindent The following proposition gives an identification of these subspaces in the spectral decomposition $(\ref{Spectral decomposition of Lu})$.
 
\begin{prop}\label{spectral decomposition caracterisation for Hpp Hac}
If $u \in \mathcal{U}_N$, then $L_u$ has exactly $N$ simple negative eigenvalues. Let $Q_u$ denote the characteristic polynomial of the $N$-soliton $u$ given in definition $\ref{definition of char poly of soliton and translation scaling parameters}$ and $\Theta_u := \Theta_{Q_u}=\frac{\overline{Q}_u}{Q_u}$ denote the associated inner function. Then we have the following identification,
\begin{equation}\label{spectral subspace identification}
\mathscr{H}_{\mathrm{ac}}(L_u) = \Theta_u L^2_+,  \qquad\mathscr{H}_{\mathrm{sc}}(L_u)=\{0\},  \qquad \mathscr{H}_{\mathrm{pp}}(L_u) = K_{\Theta_u} = \frac{\mathbb{C}_{\leq  N-1}[X]}{Q_u}.
\end{equation}
\end{prop}
\noindent For every $u \in \mathcal{U}_N$, we have the following spectral decomposition of $L_u$:
\begin{equation}\label{spectrum set splits ac sc pp}
\sigma(L_u)=\sigma_{\mathrm{ac}}(L_u) \bigcup \sigma_{\mathrm{sc}}(L_u) \bigcup\sigma_{\mathrm{pp}}(L_u), \qquad \mathrm{where} \quad \sigma_{\mathrm{ac}}(L_u) = [0, +\infty), \quad \sigma_{\mathrm{sc}}(L_u) =\emptyset
\end{equation}and $\sigma_{\mathrm{pp}}(L_u) = \{\lambda_1^u,\lambda_2^u, \cdots, \lambda_N^u\}$ consists of all eigenvalues of $L_u$. Proposition $\ref{Self adjoint ness of Lu}$ yields that $L_u$ is bounded from below and $-\frac{C^2}{4} \|u\|_{L^2}^2 \leq \lambda_1^u <  \cdots< \lambda_N^u < 0$, where $C=\inf_{f \in H^1_+ \backslash \{0\}}\frac{\||\mathrm{D}|^{\frac{1}{4}}f\|_{L^2}}{\| f\|_{L^4}}$ denotes the Sobolev constant. Hence the min-max principle (Theorem \uppercase\expandafter{\romannumeral13}.1 of Reed--Simon $[\ref{Reed Simon book 4}]$) yields that 
\begin{equation}\label{Min-Max formula of lambda N}
\lambda_n^u = \sup_{\dim_{\mathbb{C}}F=n-1 }\mathfrak{I}(F, L_u), \qquad \mathfrak{I}(F, L_u)=\inf \{\langle L_u h, h\rangle_{L^2} : h \in H^1_+ \bigcap F^{\perp}, \|h\|_{L^2}=1\}
\end{equation}where, the above supremum, $F$ describes all subspaces of $L^2_+$ of complex dimension $n$, $1 \leq n\leq N$. When $n\geq N+1$,  $\sup_{\dim_{\mathbb{C}}F=n }\mathfrak{I}(F, L_u) =\inf \sigma_{\mathrm{ess}}(L_u)=0$. Proposition $\ref{proposition showing identity of wu}$ and corollary $\ref{Wu's result on simplicity of eigenvalues}$ yield that there exist  eigenfunctions $\varphi_j : u\in \mathcal{U}_N \mapsto\varphi_j^{u} \in \mathscr{H}_{\mathrm{pp}}(L_u)$ such that 
\begin{equation}\label{definition of varphi j u a basis of H pp}
\mathrm{Ker}(\lambda_j^u - L_u)= \mathbb{C}   \varphi_j^u ,  \qquad \|\varphi_j^{u}\|_{L^2} =1, \qquad \langle \varphi_j^{u}, u \rangle_{L^2}=\int_{\mathbb{R}}u \varphi_j^u= \sqrt{2\pi |\lambda_j^u |},
\end{equation}for every $j=1, 2, \cdots, N$. Then $\{\varphi_1^u, \varphi_2^u, \cdots, \varphi_N^u\}$ is an orthonormal basis of the subspace $\mathscr{H}_{\mathrm{pp}}(L_u)$. We have the following result.

\begin{prop}\label{prop eigenvalue is analytic on UN}
For every $j=1,2,\cdots, N$, the j th eigenvalue $\lambda_j: u \in \mathcal{U}_N\mapsto \lambda_j^u \in \mathbb{R}$ is  real analytic. 
\end{prop}

\noindent We refer to proposition $\ref{Hpp is invariant by G}$ and formula $(\ref{k  Theta invariant by G})$ to see that the subspace $\mathscr{H}_{\mathrm{pp}}(L_u) \subset \mathbf{D}(G)$ is invariant by $G$. The matrix representation of $G|_{\mathscr{H}_{\mathrm{pp}}(L_u)}$ with respect to the orthonormal basis  $\{\varphi_1^u, \varphi_2^u, \cdots, \varphi_N^u\}$ is given in proposition $\ref{Coefficients of Matrix of G}$. Then the following theorem gives the spectral characterization for $N$-solitons.

\begin{thm}\label{characterization of multisoliton manifold}
A function $u \in  \mathcal{U}_N$ if and only if $u\in L^2(\mathbb{R}, (1+x^2)\mathrm{d}x)$ is real-valued, $\dim_{\mathbb{C}}\mathscr{H}_{\mathrm{pp}}(L_u) = N$ and $\Pi u \in \mathscr{H}_{\mathrm{pp}}(L_u)$. Moreover, we have the following inversion formula
\begin{equation}\label{inversion formula for Pi u and Q char polynomial }
\Pi u(x) =  i \frac{\frac{\mathrm{d}}{\mathrm{d}x}\det(x - G|_{\mathscr{H}_{\mathrm{pp}}(L_u)})}{\det(x - G|_{\mathscr{H}_{\mathrm{pp}}(L_u)})}, \qquad \forall x \in \mathbb{R}.
\end{equation} 
\end{thm}

\noindent Then $Q_u$ in definition $\ref{definition of char poly of soliton and translation scaling parameters}$ is the characteristic polynomial of  $G|_{\mathscr{H}_{\mathrm{pp}}(L_u)}$. The translation--scaling parameters of $u$ can be identified as the spectrum of $G|_{\mathscr{H}_{\mathrm{pp}}(L_u)}$. Finally the invariance of $\mathcal{U}_N$ under the BO flow is obtained by its spectral characterization, so we have the global well-posedness of the BO equation in the $N$-soliton manifold $(\ref{Hamiltonian energy of BO eq and Hamiltonian form of BO})$.
\begin{prop}\label{Invariance of UN under BO flow}
If $u_0 \in \mathcal{U}_N$, we denote by $u : t \in \mathbb{R} \mapsto u(t) \in H^{\infty}(\mathbb{R}, \mathbb{R})$ the solution of the BO equation $(\ref{Benjamin Ono equation on the line})$ with initial datum $u(0) =u_0$. Then $u(t) \in \mathcal{U}_N$, for every $t  \in \mathbb{R}$.
\end{prop}

\noindent This section is organized as follows. The real analytic structure and the symplectic structure are given in subsection $\mathbf{\ref{subsection differential structure}}$.  Then the spectral decomposition of the Lax operator $L_u$ and the real analyticity of its eigenvalues are given in subsection $\mathbf{\ref{subsection proof of spectral decomposition caracterisation for Hpp Hac}}$, for every $u \in \mathcal{U}_N$. The characterization theorem $\ref{characterization of multisoliton manifold}$ is proved in subsection $\mathbf{\ref{proof of thm characterization of multisoliton manifold}}$. Finally, we show the stability of $\mathcal{U}_N$ under the BO flow in subsection $\mathbf{\ref{subsection invariance}}$.

\bigskip

\subsection{Differential structure}\label{subsection differential structure}
The construction of real analytic structure and symplectic structure of $\mathcal{U}_N$ is divided into three steps. Firstly, we describe the complex structure of $\Pi (\mathcal{U}_N)$. Then the Hermitian metric $\mathfrak{H}$ for the complex manifold $\Pi (\mathcal{U}_N)$ is introduced in $(\ref{definition of hermitian form H})$ and we establish a real analytic diffeomorphism between $\mathcal{U}_N$ and $\Pi (\mathcal{U}_N)$. The third step is to prove $\mathrm{d}\omega=0$ on $\mathcal{U}_N$. Since $\omega = - \Pi^* (\mathrm{Im}  \mathfrak{H})$, $(\Pi (\mathcal{U}_N),  \mathfrak{H})$ is a K\"ahler manifold.  \\

\noindent Step $\mathbf{\uppercase\expandafter{\romannumeral1}}$. The Vi\`ete map $\mathbf{V} : (\beta_1,  \beta_2, \cdots, \beta_N) \in \mathbb{C}^N \mapsto (a_0, a_1, \cdots, a_{N-1}) \in \mathbb{C}^N$ is defined as follows
\begin{equation}\label{definition of Vieta map}
 \prod_{j=1}^N (X- \beta_j) = \sum_{k=0}^{N-1} a_k X^k + X^N.
\end{equation}Both addition and multiplication of two complex numbers are open continuous maps $\mathbb{C}^2 \to \mathbb{C}$, the Vi\`ete  map $\mathbf{V} : \mathbb{C}^N \to \mathbb{C}^N$ is an open quotient map. So $\mathbf{V}(\mathbb{C}_-^N)$ is an open connected subset of $\mathbb{C}^N$ (see also proposition $\ref{simply connected property for viete map}$). With the subspace topology and the Hermitian form $ \mathfrak{H}_{\mathbb{C}^N}(X,Y)=\langle X, Y \rangle_{\mathbb{C}^N} = X^T    \overline{Y}$, the subset $(\mathbf{V}(\mathbb{C}_-^N),  \mathfrak{H}_{\mathbb{C}^N})$ is a connected K\"ahler manifold  of complex dimension $N$. 

\begin{lem}\label{Lemma of complex analytic manifold Pi UN}
Equipped with the subspace topology of $L^2_+$, the subset $\Pi (\mathcal{U}_N)$ is a connected topological manifold of complex dimension $N$ and it has a unique complex analytic structure making it into an embedded submanifold of the $\mathbb{C}$-Hilbert space $L^2_+$. For every $u\in \mathcal{U}_N$, its translation--scaling parameters are denoted by $\mathbf{cl}(x_1-i\eta_1, x_2-i\eta_2, \cdots, x_N-i\eta_N)$, for some  $x_j  \in \mathbb{R}$ and $\eta_j >0$, then the tangent space to $\Pi (\mathcal{U}_N)$ at $\Pi u$ is given by
\begin{equation}\label{tangent space of u in Pi UN}
\mathcal{T}_{\Pi u} (\Pi (\mathcal{U}_N)) = \bigoplus_{j=1}^N  \mathbb{C}h_j^u , \qquad \mathrm{where}\quad h_{j}^u(x)=  \frac{1}{( x-x_j  + \eta_j i)^2}.
\end{equation}
\end{lem}

\begin{proof}
We define $\Gamma_N : \mathbf{a}= (a_0, a_1, \cdots, a_{N-1}) \in \mathbf{V}(\mathbb{C}_-^N) \mapsto \Pi u = i\frac{Q'}{Q} \in \Pi(\mathcal{U}_N) \subset L^2_+$ such that 
\begin{equation*}
Q(X)=\sum_{k=0}^{N-1} a_k X^k + X^N. 
\end{equation*}The surjectivity of $\Gamma_N$ is given by the definition of $\mathcal{U}_N$. Since the monic polynomial $Q$ is uniquely determined by $u \in \mathcal{U}_N$, the map $\Gamma_N$ is injective. For every $\mathbf{h}=(h_0, h_1, \cdots, h_{N-1})\in \mathbb{C}^N$, we have
\begin{equation*}
\mathrm{d} \Gamma_N (a_0, a_1, \cdots, a_{N-1}) \mathbf{h} = i\frac{Q H'- Q'H}{Q^2}, \qquad \mathrm{where} \quad H(X)=\sum_{k=0}^{N-1} h_k X^k.
\end{equation*}If $\mathrm{d} \Gamma_N (a_0, a_1, \cdots, a_{N-1}) \mathbf{h} =0$, then $(\frac{H}{Q})' \equiv 0$. Since $\deg H \leq   \deg Q-1$, we have $H=0$. Thus $\Gamma_N : \mathbf{V}(\mathbb{C}_-^N) \to L^2_+$ is a complex analytic immersion. We claim that $\Gamma_N$ is a topological embedding. \\

\noindent In fact we set $\mathbf{a}^{(n)}=(a_0^{(n)}, a_1^{(n)}, \cdots, a_{N-1}^{(n)}) \in \mathbf{V}(\mathbb{C}_-^N)$ such that 
\begin{equation*}
\frac{\partial_x Q_n}{Q_n} \to \frac{\partial_x Q }{Q } \quad \mathrm{in} \quad L^2_+, \qquad \mathrm{as} \quad n \to +\infty, \qquad \mathrm{where} \quad Q_n(x)= \sum_{j=0}^{N-1} a_j^{(n)} x^j + x^N, \quad \forall x \in \mathbb{R}.
\end{equation*}Since $\mathbf{a}^{(n)} \in \mathbf{V}(\mathbb{C}_-^N)$, we have $a_0^{(n)}=Q_n(0) \ne 0$. For every $x \in \mathbb{R}$, we have
\begin{equation}\label{convergence simple of Qn}
\frac{Q_n(x)}{Q_n(0)}=  \exp(\int_0^x \frac{\partial_y Q_n(y)}{Q_n(y)}\mathrm{d}y) \to \exp(\int_0^x \frac{\partial_y Q (y)}{Q(y)}\mathrm{d}y) = \frac{Q (x)}{Q (0)}, \qquad \mathrm{as} \quad n \to +\infty.
\end{equation}Every coefficient of $\frac{Q_n}{Q_n(0)}$ converges to the corresponding coefficient of $\frac{Q(x)}{Q(0)}$. Since $Q_n, Q$ are monic, we have $\lim_{n \to +\infty}\frac{1}{Q_n(0)}=\frac{1}{Q (0)}$ and $\lim_{n\to+\infty} \mathbf{a}^{(n)} = \mathbf{a}$. Then $\Gamma_N^{-1} : \Pi(\mathcal{U}_N) \subset L^2_+ \to \mathbf{V}(\mathbb{C}_-^N)$ is continuous. Since $\Gamma_N$ is a complex analytic embedding, with the subspace topology of $L^2_+$, there exists a unique complex analytic structure making  $\Pi (\mathcal{U}_N)= \Gamma_N \circ \mathbf{V}(\mathbb{C}_-^N)$ into an embedded complex analytic  submanifold of $L^2_+$. The map $\Gamma_N :  \mathbf{V}(\mathbb{C}_-^N) \to \Pi (\mathcal{U}_N)$ is biholomorphic. Set $u(x) = \sum_{j=1}^N  \tfrac{2 \eta_j}{(x- x_j )^2 +  \eta_j^2}$ for some $x_j  =x_j(u)\in \mathbb{R}$ and $\eta_j=\eta_j(u) >0$. Then every $h \in \mathcal{T}_{\Pi u}(\Pi(\mathcal{U}_N))$ is identified as the velocity of the smooth curve $c : t \in (-1, 1) \to \Pi(\mathcal{U}_N)$ such that $c(0)=\Pi u$ at $t=0$. If we choose 
\begin{equation*}
c(t,x)=\sum_{j=1}^N  \frac{i}{ x- x_j(t)  +  \eta_j(t) i} \qquad \mathrm{where} \quad x_j(t) \in \mathbb{R},  \quad \eta_j(t) >0.
\end{equation*}Then we have $x_j(0)=x_j$, $\eta_j(0)=\eta_j$ and 
\begin{equation}\label{formula of tangent vector  h in T Pi u Pi UN}
h(x)= \partial_t \big|_{t=0} c(t,x) = \sum_{j=1}^N   \frac{\eta_j'(0) + i x_j'(0)}{( x-x_j  + \eta_j i)^2} .
\end{equation}We have $h_j^u = \Pi f_j^u = -i \Pi g_j^u$ and $(h_j^u)^{\wedge}(\xi)= -2\pi \mathbf{1}_{\xi \geq 0} \xi  e^{-(i x_j(u) + \eta_j(u) ) \xi } $. For every $h \in \mathcal{T}_{\Pi  u}(\Pi(\mathcal{U}_N))$, we have $\xi \mapsto \xi^{-1}\hat{h}(\xi) \in L^2(\mathbb{R})$ (see also Hardy's inequality $(\ref{Hardy inequality H1 R})$).
\end{proof}

\noindent Step $\mathbf{\uppercase\expandafter{\romannumeral2}}$. Given $u\in \mathcal{U}_N$, the Hermitian metric $\mathfrak{H}_{\Pi u}$ is defined as follows 
\begin{equation}\label{definition of hermitian form H}
\mathfrak{H}_{\Pi u}(h_1, h_2) =  \int_0^{+\infty}\frac{\hat{h}_1(\xi) \overline{\hat{h}_2(\xi)}}{\pi \xi } \mathrm{d}\xi, \qquad \forall h_1, h_2 \in \mathcal{T}_{\Pi u}(\Pi(\mathcal{U}_N)).
\end{equation}The sesquilinear form $\mathfrak{H}_{\Pi u}$ is positive definite because $\mathfrak{H}_{\Pi u}(h,h)=\int_0^{+\infty} \frac{|\hat{h}(\xi)|^2}{\pi \xi}\mathrm{d}\xi >0$, if $h \ne 0$. Hence the smooth symmetric covariant $2$-tensor field $\mathrm{Re}\mathfrak{H}$ is positive definite on $\Pi(\mathcal{U}_N)$, so $(\Pi(\mathcal{U}_N), \mathrm{Re}\mathfrak{H})$ is a  Riemannian manifold of real dimension $2N$. \\

\noindent We consider the $\mathbb{R}$-linear isomorphism between the Hilbert spaces  
\begin{equation*}
\Pi : u \in L^2(\mathbb{R}, \mathbb{R}) \mapsto \Pi u \in L^2_+, \qquad f \in L^2_+ \mapsto 2 \mathrm{Re}f  \in L^2(\mathbb{R}, \mathbb{R}).
\end{equation*}Then $\Pi \circ 2\mathrm{Re} = \mathrm{Id}_{L^2_+}$ and $ 2\mathrm{Re} \circ \Pi = \mathrm{Id}_{L^2(\mathbb{R}, \mathbb{R})} $ and $\|u\|_{L^2} = \sqrt{2} \|\Pi u\|_{L^2}$. Then $\mathcal{U}_N = 2\mathrm{Re} \circ \Pi (\mathcal{U}_N)$ is a real analytic manifold of real dimension $2N$. Furthermore we have $f_j^u = 2 \mathrm{Re}h_j^u$,  $g_j^u = 2 i\mathrm{Re}h_j^u$ and 
\begin{equation}\label{linear isomorphism 2 Re}
2 \mathrm{Re} : \mathcal{T}_{\Pi u}(\Pi (\mathcal{U}_N)) \to  \mathcal{T}_{u}(\mathcal{U}_N)
\end{equation}is an $\mathbb{R}$-linear isomorphism. Since $\mathfrak{H}$ is Hermitian, the $2$-form $\omega = - \Pi^* (\mathrm{Im}  \mathfrak{H}  )$ is nondegenerate  on $\mathcal{U}_N$.   \\

\noindent Step $\mathbf{\uppercase\expandafter{\romannumeral3}}$. We set $\mathcal{E}:=  L^2(\mathbb{R}, \mathbb{R}) \bigcap L^2(\mathbb{R}, x^2\mathrm{d}x)$, $\mathcal{E}_{c} : = \{u \in \mathcal{E} : \int_{\mathbb{R}}u=c\}$, for every $c\in \mathbb{R}$. Then  
\begin{equation*}
\mathcal{U}_N \subset \mathcal{E}_{2\pi N}, \qquad \mathcal{T}_u(\mathcal{U}_N) \subset \mathcal{T}:=\mathcal{E}_0, \quad \forall u\in \mathcal{U}_N.
\end{equation*}The nondegenerate  $2$-form $\omega$ can be extended to a $2$-covector of the subspace $\mathcal{T}$. Recall that
\begin{equation}\label{omega def in T}
\boldsymbol{\omega} (h_1, h_2) =\frac{i}{2\pi} \int_{\mathbb{R}} \frac{\hat{h}_1(\xi) \overline{\hat{h}_2(\xi)}}{\xi}\mathrm{d}\xi, \qquad \forall h_1, h_2 \in \mathcal{T}.
\end{equation}If $h \in \mathcal{T}$, then we have $\hat{h}(0)=0$ and $\hat{h}\in H^1(\mathbb{R})$. Hence the Hardy's inequality (see Brezis $[\ref{Brezis Functional analysis PDEs}]$, Bahouri--Chemin--Danchin $[\ref{Bahouri Chemin Danchin book}]$ etc.) yields that
\begin{equation}\label{Hardy inequality H1 R}
\int_{\mathbb{R}} \frac{|\hat{h}(\xi)|^2}{|\xi|^2} \mathrm{d}\xi \leq 4 \|\partial_{\xi} \hat{h}\|_{L^2}^2 \Longrightarrow \xi \mapsto \frac{\hat{h}(\xi)}{\xi} \in L^2(\mathbb{R}),
\end{equation}so the $2$-covector $\boldsymbol{\omega} \in \boldsymbol{\Lambda}^2(\mathcal{T}^*)$ is well defined and $\omega_u (h_1, h_2) = \boldsymbol{\omega}(h_1, h_2)$. For every smooth vector field $X \in \mathfrak{X}(\mathcal{U}_N)$, let $X \lrcorner \omega \in \boldsymbol{\Omega}^1(\mathcal{U}_N)$ denote the interior multiplication by $X$, i.e. $(X\lrcorner \omega) (Y) = \omega(X,Y)$, for every $Y \in \mathfrak{X}(\mathcal{U}_N)$.   We shall prove that $\mathrm{d}\omega=0$ on $\mathcal{U}_N$ by using Cartan's formula: 
\begin{equation}\label{Cartan's Magic formula}
\mathscr{L}_X \omega =  X\lrcorner  (\mathrm{d}\omega)  + \mathrm{d}( X\lrcorner \omega).
\end{equation}

\begin{proof}[Proof of proposition $\ref{prop omega is closed symplectic manifold UN}$]
For any smooth vector field $X \in \mathfrak{X}(\mathcal{U}_N)$, let $\phi$ denote the smooth maximal flow of $X$. If $t$ is sufficiently close to $0$, then $\phi_t : u\in \mathcal{U}_N \mapsto \phi(t,u) \in \mathcal{U}_N$ is a local diffeomorphism by the fundamental theorem on flows (see Theorem 9.12 of Lee $[\ref{Lee book on smooth manifold}]$). For every $u \in \mathcal{U}_N$, $h_1, h_2 \in \mathcal{T}_u(\mathcal{U}_N)$, we compute the Lie derivative of $\omega$ with respect to $X$,
\begin{equation*}
\begin{split}
(\mathscr{L}_X \omega)_u(h_1,h_2) = & \lim_{t \to 0} \frac{\omega_{\phi_t(u)} (\mathrm{d} \phi_t(u) h_1, \mathrm{d} \phi_t(u) h_2)-\omega_{u} (    h_1,   h_2)}{t} \\
= & \lim_{t \to 0}\boldsymbol{\omega}  \left(\frac{ \mathrm{d} \phi_t(u) h_1 - h_1}{t}, \mathrm{d} \phi_t(u) h_2 \right) + \lim_{t \to 0}\boldsymbol{\omega} \left(h_1,   \frac{ \mathrm{d} \phi_t(u) h_2 - h_2}{t} \right).
\end{split}
\end{equation*}Since $\lim_{t \to 0} \frac{ \mathrm{d} \phi_t(u) h_j  - h_j }{t} = \mathrm{d} X(u) h_j \in \mathcal{T}_u(\mathcal{U}_N)$, for every $j=1,2$, we have
\begin{equation*}
(\mathscr{L}_X \omega)_u(h_1,h_2) = \boldsymbol{\omega}  (\mathrm{d} X(u) h_1, h_2) + \boldsymbol{\omega}  ( h_1, \mathrm{d} X(u) h_2) = \left(h_1 \omega (X , h_2)\right)(u) -  \left( h_2 \omega  (X , h_1)\right)(u).
\end{equation*}We choose $(V, x^i)$ a smooth local chart for $\mathcal{U}_N$ such that $u \in V$ and the tangent vector $h_k$ has the coordinate expression $h_k= \sum_{j=1}^{2N} h_k^{(j)}\frac{\partial}{\partial x^j}\big|_u$, for some $h_k^{(j)} \in \mathbb{R}$, $j=1,2 \cdots, 2N$ and $k=1,2$. The tangent vector $h_k$ can be identified as some locally constant vector field $Y_k \in \mathfrak{X}(\mathcal{U}_N)$  defined by
\begin{equation*}
Y_k : v \in V \mapsto  \sum_{j=1}^{2N} h_k^{(j)}\frac{\partial}{\partial x^j}\Big|_v \in \mathcal{T}_v(\mathcal{U}_N), \quad Y_k :u \mapsto (Y_k)_u= h_k,   \qquad k=1,2.
\end{equation*}Then the vector field $[Y_1, Y_2]$ vanishes in the open subset $V$. The exterior derivative of the $1$-form $\beta= X \lrcorner \omega$ is computed as $\mathrm{d}\beta (Y_1,Y_2) = Y_1 \left( \beta (Y_2)\right) - Y_2 \left( \beta (Y_1)\right) + \beta([Y_1,Y_2])$. Thus
\begin{equation*}
\mathrm{d}( X \lrcorner \omega)_u (h_1,h_2) =h_1 \omega_u (X_u, h_2) -  h_2 \omega_u (X_u, h_1) + \omega_u (X_u, [Y_1, Y_2]_u)= (\mathscr{L}_X \omega)_u(h_1,h_2).
\end{equation*}Then Cartan's formula $(\ref{Cartan's Magic formula})$ yields that $ X \lrcorner (\mathrm{d}\omega)=0$. Since $X \in \mathfrak{X}(\mathcal{U}_N)$ is arbitrary, we have $\mathrm{d}\omega =0$. As a consequence, the real analytic $2$-form $\omega : u \in \mathcal{U}_N \mapsto \boldsymbol{\omega} \in \boldsymbol{\Lambda}^2(\mathcal{T}^*)$ is a symplectic form. 
\end{proof}
\noindent Since $ \mathrm{Im}\mathfrak{H} =  (-2 \mathrm{Re})^* \omega$, where $-2 \mathrm{Re} : \Pi (\mathcal{U}_N) \to \mathcal{U}_N$ is a real  analytic diffeomorphism, the associated $2$-form $\mathrm{Im}\mathfrak{H}$ is closed. So $(\Pi (\mathcal{U}_N),   \mathfrak{H})$ is a K\"ahler manifold. The simple connectedness of $\Pi (\mathcal{U}_N)$ is proved in subsection $\mathbf{\ref{subsection simple connectedness}}$. \\

\bigskip

\subsection{Spectral analysis \uppercase\expandafter{\romannumeral2}}\label{subsection proof of spectral decomposition caracterisation for Hpp Hac}
We continue to study the spectrum of the Lax operator $L_u$ introduced in definition $\ref{definition of L u and B u}$. The general cases $u\in L^2(\mathbb{R, \mathbb{R})}$ and $u\in L^2(\mathbb{R}, (1+x^2)\mathrm{d}x)$ have been studied in subsection $\mathbf{\ref{subsection of spectral analysis in Lax operator}}$.  We restrict our study to the case $u\in  \mathcal{U}_N$ in this subsection. Let $Q=Q_u$ denote the characteristic polynomial of $u$ and $\Theta:=\frac{\overline{Q}}{Q}$, $K_{\Theta}=( \Theta L^2_+)^{\perp}$. Since $L_u$ is an unbounded self-adjoint operator of $L^2_+$, we have the following 
\begin{equation*}
L^2_+ = \Theta L^2_+ \bigoplus K_{\Theta} = \mathscr{H}_{\mathrm{ac}}(L_u) \bigoplus \mathscr{H}_{\mathrm{sc}}(L_u) \bigoplus \mathscr{H}_{\mathrm{pp}}(L_u).
\end{equation*}We shall at first identify those subspaces by proving proposition $\ref{spectral decomposition caracterisation for Hpp Hac}$ and formula $(\ref{spectrum set splits ac sc pp})$. Then we turn to study the real analyticity of each eigenvalue $\lambda_j : u \in \mathcal{U}_N \mapsto \lambda_j^u \in \mathbb{R}$.
 
\begin{proof}[Proof of proposition $\ref{spectral decomposition caracterisation for Hpp Hac}$]
The first step is to prove $K_{\Theta} = \frac{\mathbb{C}_{\leq  N-1}[X]}{Q}$. In fact, for every $h \in L^2_+$ and $f = \frac{P}{Q} \in \frac{\mathbb{C}_{\leq N-1}[X]}{Q}$, for some $P \in \mathbb{C}_{\leq N-1}[X]$, we have
\begin{equation*}
\langle f, \Theta h\rangle_{L^2} = \int_{\mathbb{R}}\frac{P(x)\overline{\Theta}(x)\overline{h}(x)}{Q(x)} \mathrm{d}x = \int_{\mathbb{R}}\frac{P(x) \overline{h}(x)}{\overline{Q}(x)} \mathrm{d}x = \langle \frac{P}{\overline{Q}},  h\rangle_{L^2}.
\end{equation*}Since $\overline{Q}(x)=\prod_{j=1}^N (x- \alpha_j)$ with $\mathrm{Im}(\alpha_j)>0$, the meromorphic function $\frac{P}{\overline{Q}}$ has poles in $\mathbb{C}_+$, so $\frac{P}{\overline{Q}} \in L^2_-$. Thus $\langle f, \Theta h\rangle_{L^2}  = \langle \frac{P}{\overline{Q}},  h\rangle_{L^2} =0$. Thus $\frac{\mathbb{C}_{\leq N-1}[X]}{Q} \subset (\Theta L^2_+)^{\perp} =K_{\Theta}$.\\

\noindent Conversely, if $f \in K_{\Theta}$, then $\langle \Theta^{-1}f, h \rangle_{L^2} =\langle f, \Theta h \rangle_{L^2} = 0$, for every $h \in L^2_+$. Thus $g:= \frac{Q}{\overline{Q}}f \in L^2_-$. It suffices to prove that $P:=Q f = \overline{Q}g \in \mathbb{C}[X]$. In fact, 
\begin{equation*}
\widehat{Qf}= Q (i \partial_{\xi}) \hat{f}  \qquad \mathrm{and} \qquad \mathrm{supp}(\hat{f}) \subset [0,+\infty) \Longrightarrow \mathrm{supp}(\widehat{Qf}) \subset [0,+\infty). 
\end{equation*}Similarly, $\mathrm{supp}((\overline{Q}g)^{\wedge}) \subset (-\infty, 0]$. Thus $\mathrm{supp}(\hat{P})\subset \{0\}$ and $P$ is a polynomial. Since $f = \frac{P}{Q} \in L^2(\mathbb{R})$, we have $\deg P \leq N-1$. So $K_{\Theta} \subset \frac{\mathbb{C}_{\leq N-1}[X]}{Q}$.\\

\noindent The second step is to prove $L_u (\Theta L^2_+) \subset \Theta L^2_+$. Precisely, we have
\begin{equation}\label{precise formula for stability of theta L2+  and Ktheta by L_u }
L_u(\Theta h)= \Theta \mathrm{D} h, \qquad \forall h \in L^2_+.
\end{equation}Since $\frac{\mathbb{C}_{\leq N-1}[X]}{Q} \subset L^2_+$, $\Theta = \frac{\overline{Q}}{Q}$ and $\frac{\mathrm{D} \Theta}{\Theta}=\frac{\mathrm{D}  \overline{Q}}{\overline{Q}} - \frac{\mathrm{D}  Q}{Q} = i \frac{Q'}{Q} - i \frac{\overline{Q}'}{\overline{Q}} = \Pi u + \overline{\Pi u} = u$ on $\mathbb{R}$, we have
\begin{equation*}
\begin{split}
L_u (\Theta h) =(\mathrm{D}-T_u)(\Theta h)= \Theta \mathrm{D} h +  h \left(\mathrm{D}  \Theta  - i  \tfrac{Q'}{Q} \Theta   + i  \tfrac{\overline{Q}'}{Q}  \right) 
=   \Theta \mathrm{D}h +  h \Theta \left(\tfrac{\mathrm{D}  \Theta}{\Theta}  - i  \tfrac{Q'}{Q}    + \tfrac{\overline{Q}'}{\overline{Q}} \right) = \Theta \mathrm{D}  h.
\end{split}
\end{equation*}Recall that $L_u = L_u^*$, so we have $L_u(K_{\Theta}) \subset K_{\Theta}$. Since $\dim_{\mathbb{C}}K_{\Theta} = N$, corollary $\ref{Wu's result on simplicity of eigenvalues}$ yields that the Hermitian matrix $L_{u|{K_{\Theta}}}$ has exactly $N$ distinct eigenvalues. Hence $K_{\Theta} \subset \mathscr{H}_{\mathrm{pp}}(L_u)$. \\

\noindent On the other hand, we set $U_{\Theta} : L^2_+ \to \Theta L^2_+$ such that $U_{\Theta} h = \Theta h$. Thus $\|U_{\Theta}\|_{\mathfrak{B}(L^2_+,  \Theta L^2_+)}=1$ and
\begin{equation*}
U_{\Theta}^{-1}=U_{\Theta}^* : g \in \Theta L^2_+ \mapsto  \Theta^{-1} g \in L^2_+.
\end{equation*}So $U_{\Theta} : L^2_+ \to \Theta L^2_+$ is a unitary operator. $U_{\Theta}(H^1_+)=\Theta H ^1_+ = H^1_+ \bigcap \Theta L^2_+$. Formula $(\ref{precise formula for stability of theta L2+  and Ktheta by L_u })$ yields that 
\begin{equation*}
U_{\Theta}^* L_{u|\Theta L^2_+} U_{\Theta} = \mathrm{D}, \qquad U_{\Theta}[\mathbf{D}(\mathrm{D})]  =\Theta H ^1_+ = H^1_+ \bigcap \Theta L^2_+ = \mathbf{D}( L_{u|\Theta L^2_+} ).
\end{equation*}For every bounded Borel function $f : \mathbb{R} \to \mathbb{C}$, we have $f(L_u)U_{\Theta} = U_{\Theta} f( \mathrm{D})$ by proposition $\ref{proposition on unitary equivalence}$. We denote by $\mu_{\psi}=\mu_{\psi}^{L_u}$ the spectral measure of $L_u$ associated to $\psi \in L^2_+$, then $\forall h \in L^2_+$, we have
\begin{equation*}
\int_{\mathbb{R}}f(\xi) \mathrm{d}\mu_{\Theta h}(\xi)=\langle f(L_u)U_{\Theta}h,  U_{\Theta}h \rangle_{L^2} = \langle \Theta f( \mathrm{D}) h,   \Theta h \rangle_{L^2} =  \langle   f( \mathrm{D}) h,     h \rangle_{L^2} =\frac{1}{2\pi}\int_0^{+\infty} f(\xi)  |\hat{h}(\xi)|^2 \mathrm{d}\xi.
\end{equation*}So $\mathrm{d}\mu_{\Theta h}(\xi) = \frac{\mathbf{1}_{\mathbb{R}_+}|\hat{h}(\xi)|^2}{2 \pi}  \mathrm{d}\xi$. The spectral measure $\mu_{\Theta h}$ is absolutely continuous with respect to the Lebesgue measure on $\mathbb{R}$. Thus $\Theta L^2_+ \subset \mathscr{H}_{\mathrm{ac}}(L_u) \subset \mathscr{H}_{\mathrm{cont}}(L_u) = ( \mathscr{H}_{\mathrm{pp}}(L_u))^{\perp} \subset \Theta L^2_+$ and $(\ref{spectral subspace identification})$ is obtained. We have $\mathrm{supp}(\mu_{\Theta h}) \subset [0,+\infty)$, for every $h \in L^2_+$. $\forall\xi \geq 0$, there exists $h\in L^2_+$ such that $\hat{h}(\xi) \ne 0$. So we have $\sigma_{\mathrm{ess}}(L_u)=\sigma_{\mathrm{cont}}(L_u)=\sigma_{\mathrm{ac}}(L_u)= [0,+\infty)$. 
\end{proof}

\noindent Before proving the real analyticity of each eigenvalue, we show its continuity at first. 
\begin{lem}\label{lemma eigenvalue is continuous}
For every $j=1,2,\cdots, N$, the j th eigenvalue $\lambda_j: u \in \mathcal{U}_N\mapsto \lambda_j^u \in \mathbb{R}$ is Lipschitz continuous on every compact subset of $\mathcal{U}_N$. 
\end{lem}

\begin{proof}
For every $f\in H^1(\mathbb{R})$, the Sobolev embedding $\|f\|_{L^4} \leq C \||\mathrm{D}|^{\frac{1}{4}} f\|_{L^2}$  yields that $\forall u, v \in \mathcal{U}_N$,
\begin{equation}\label{estimates of Lu - Lv}
\big| \langle  L_u h, h\rangle_{L^2} - \langle  L_v h, h\rangle_{L^2}\big| \leq \|u-v\|_{L^2}\|h\|_{L^4}^2 \leq C \|u-v\|_{L^2} \||\mathrm{D}|^{\frac{1}{2}} h\|_{L^2} \|  h\|_{L^2}, \quad \forall h \in H^1_+.
\end{equation}Given $j=1,2, \cdots, N$ and  a subspace $F \subset L^2_+$ with complex dimension $j-1$, we choose 
\begin{equation*}
h \in  F^{\perp} \bigcap \bigoplus_{k=1}^{j} \mathrm{Ker}(\lambda_k^u - L_u) \subset H^1_+, \qquad \|h\|_{L^2}=1, \qquad h = \sum_{k=1}^j h_k \varphi_k^u.
\end{equation*}Then $\langle L_u h, h\rangle_{L^2} = \sum_{k=1}^j |h_k|^2 \lambda_k^u \leq \lambda_j^u<0$, because $\lambda_k^u < \lambda_{k+1}^u$. We have the following estimate
\begin{equation}\label{estimates Dh h}
\|\mathrm{D}|^{\frac{1}{2}} h\|_{L^2}^2 = \langle \mathrm{D}h, h \rangle_{L^2} = \langle L_u h, h\rangle_{L^2} + \langle u h, h\rangle_{L^2} \leq \lambda_j^u + \|u\|_{L^2}\|h\|_{L^4}^2 \leq  C \|u\|_{L^2}  \||\mathrm{D}|^{\frac{1}{2}} h\|_{L^2} \|  h\|_{L^2}.
\end{equation}So estimates $(\ref{estimates of Lu - Lv})$ and $(\ref{estimates Dh h})$ yield that $\langle  L_v h, h\rangle_{L^2} \leq \lambda_j^u +C^2 \|u\|_{L^2} \|u-v\|_{L^2}$. Since $F$ is arbitrary,  the max--min formula $(\ref{Min-Max formula of lambda N})$ implies that 
\begin{equation*}
|\lambda_j^u - \lambda_j^v|\leq C^2 (\|u\|_{L^2}+ \|v\|_{L^2})\|u-v\|_{L^2}.
\end{equation*}Every compact subset $K \subset \mathcal{U}_N$ is bounded in $L^2(\mathbb{R},\mathbb{R})$. Hence $u \in K \mapsto \lambda_j^u \in \mathbb{R}$ is Lipschitz continuous.  
\end{proof}

\begin{proof}[Proof of proposition $\ref{prop eigenvalue is analytic on UN}$]
For every $u \in \mathcal{U}_N$, the Lax operator $L_u$ has $N$ negative simple eigenvalues, denoted by $\lambda_1^u < \lambda_2^u<  \cdots  < \lambda_N^u <0$. Let $\mathbb{P}^{j}_{u}$ denotes the Riesz projector of the eigenvalue $\lambda_j^u$ and  
\begin{equation*}
D(z, \epsilon)=\{\eta \in \mathbb{C} : |\eta - z|< \epsilon\}, \quad \mathscr{C}(z, \epsilon)= \partial D(z, \epsilon)=\{\eta \in \mathbb{C} : |\eta - z| =\epsilon\}, \quad \forall z \in \mathbb{C},\quad \epsilon>0.
\end{equation*}Then there exists $\epsilon_0>0$ such that the family of closed discs $\{\overline{D}(\lambda_j^u,  \epsilon_0)\}_{1\leq j\leq N}\bigcup \{\overline{D}(0,  \epsilon_0)\}$ is mutually disjoint and for every $j, k =1,2 \cdots, N$ and any closed path $\boldsymbol{\Gamma}_j^u$ (piecewise $C^1$ closed curve) in $ D(\lambda_j^u,  \epsilon_0)$ with respect to which the eigenvalue $\lambda_j^u$ has winding number $1$, we have
\begin{equation}\label{Riesz Projector for lambda j}
\mathbb{P}^{j}_{u} = \frac{1}{2\pi i} \oint_{\boldsymbol{\Gamma}_j^u} (\zeta - L_u)^{-1} \mathrm{d}\zeta, \qquad \mathbb{P}^{j}_{u} \circ \mathbb{P}^{j}_{u}=\mathbb{P}^{j}_{u},\qquad \mathbb{P}^{j}_{u} \varphi_k^u = \mathbf{1}_{j=k}\varphi_k^u.
\end{equation}by Theorem \uppercase\expandafter{\romannumeral12}.5 of Reed--Simon $[\ref{Reed Simon book 4}]$. We choose $\boldsymbol{\Gamma}_j^u$ to be the counterclockwise-oriented circle $\mathscr{C}(\lambda_j^u,  \epsilon )$ in $(\ref{Riesz Projector for lambda j})$ for some $ \epsilon\in (0, \epsilon_0)$. We claim that $\mathrm{Im}\mathbb{P}^{j}_{u}  = \mathrm{Ker}(\lambda_j^u - L_u)=\mathbb{C}\varphi_j^u $.\\

\noindent It suffices to show that $\mathbb{P}^{j}_{u}|_{\mathscr{H}_{\mathrm{ac}}(L_u)}=0$. In fact the operator $\mathbb{P}^{j}_{u} = g_{\lambda_j^u}(L_u)$ is self-adjoint by Theorem \uppercase\expandafter{\romannumeral8}.6 of Reed--Simon $[\ref{Reed Simon book 1}]$, where the real-valued bounded Borel function $g_{\lambda} : \mathbb{R} \to \mathbb{R}$ is given by 
\begin{equation*}
g_{\lambda}(x):=  \frac{1}{2\pi i} \oint_{\mathscr{C}(\lambda,  \epsilon )} (\zeta - x)^{-1} \mathrm{d}\zeta = \mathbf{1}_{(\lambda-\epsilon, \lambda+\epsilon)}(x), \qquad \mathrm{a}.\mathrm{e}. \quad \mathrm{on} \quad \mathbb{R},
\end{equation*}for every $\lambda\in \mathbb{R}$. Since $\mathbb{P}^{j}_{u}(\mathscr{H}_{\mathrm{pp}}(L_u)) \subset \mathbb{C}\varphi_j^u \subset \mathscr{H}_{\mathrm{pp}}(L_u)$, we have $\mathbb{P}^{j}_{u}(\mathscr{H}_{\mathrm{ac}}(L_u)) \subset \mathscr{H}_{\mathrm{ac}}(L_u) $. Let $\mu_{\psi}=\mu_{\psi}^{L_u}$ denote the spectral measure of $L_u$ associated to the function $\psi \in\mathscr{H}_{\mathrm{ac}}(L_u) $, whose support is included in $[0, +\infty)$ by formula $(\ref{spectrum set splits ac sc pp})$, we have
\begin{equation*}
\langle \mathbb{P}^{j}_{u} \psi, \psi\rangle_{L^2} = \frac{1}{2\pi i} \oint_{\mathscr{C}(\lambda_j^u,  \epsilon )} \langle (\zeta - L_u)^{-1}\psi, \psi \rangle_{L^2} \mathrm{d}\zeta = \frac{1}{2\pi i}\int_0^{+\infty}\left(\oint_{\mathscr{C}(\lambda_j^u,  \epsilon )} (\zeta - \xi)^{-1} \mathrm{d}\zeta\right) \mathrm{d}\mu_{\psi}(\xi)=0.
\end{equation*}Set $\tilde{\psi}=\mathbb{P}^{j}_{u} \psi \in \mathscr{H}_{\mathrm{ac}}(L_u)$, then $\|\tilde{\psi}\|_{L^2}^2 = \langle   \mathbb{P}^{j}_{u} \tilde{\psi}, \tilde{\psi}\rangle_{L^2} =0$. So the claim is obtained.\\

\noindent For every fixed $j=1,2, \cdots N$, we have $\lambda^u_j= \mathrm{Tr}(L_u \circ \mathbb{P}^{j}_{u})$. Since every eigenvalue $\lambda_k : v \in \mathcal{U}_N \mapsto \lambda_k^v \in \mathbb{R}$ is continuous, there exists an open subset $\mathcal{V}\subset \mathcal{U}_N$ containing $u$ such that $\sup_{v \in \mathcal{V}}\sup_{1\leq k \leq N}|\lambda^v_k-\lambda^u_k| < \frac{\epsilon_0}{3}$. We set $\epsilon=\frac{2\epsilon_0}{3}$, then $\lambda_j^v \in D(\lambda_j^u, \epsilon) \backslash \overline{D}(\lambda_k^u, \epsilon_0)  $, for every $v \in \mathcal{V}$ and $k \ne j$. For example, in the next picture, the dashed circles denote respectively $\mathscr{C}(\lambda_j^u, \epsilon_0)$ and $\mathscr{C}(\lambda_k^u, \epsilon_0)$; the smaller circles denote respectively $\mathscr{C}(\lambda_j^u, \epsilon)$ and $\mathscr{C}(\lambda_k^u, \epsilon)$ with $j<k$. The segments inside small circles denote the possible positions of $\lambda_j^v$ and $\lambda_k^v$.
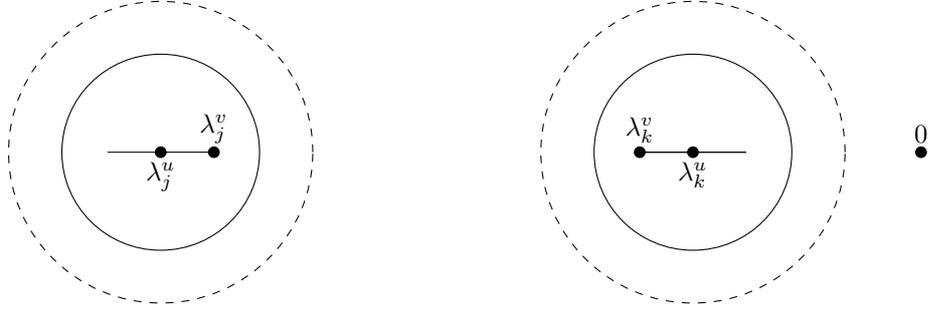
\begin{center}

\begin{tikzpicture} 

\draw [dashed](2,2) circle (2cm);
 \draw (2,2) circle (1.3cm);
\draw [dashed](9,2) circle (2cm);
\draw (9,2) circle (1.3cm);
\draw (1.3,2) -- (2,2);
 \filldraw 
 (2,2) circle (2pt) node[align=center, below] {$\lambda_j^u$}     -- 
(2.7,2) circle (2pt) node[align=right,  above] {$\lambda_j^v$};
\draw (9,2) -- (9.7,2);
  \filldraw 
 (8.3,2) circle (2pt) node[align=center, above] {$\lambda_k^v$}     -- 
(9,2) circle (2pt) node[align=right,  below] {$\lambda_k^u$};
\draw (8.3,2) -- (9.7,2);
 \filldraw 
 (12,2) circle (2pt) node[align=center, above] {$0$};

\end{tikzpicture}
\end{center}Then $\sigma (L_v) \bigcap D(\lambda_j^u, \epsilon_0)=\{\lambda_j^v\}$ and $\mathscr{C}(\lambda_j^u, \epsilon)$ is a closed path in $D(\lambda_j^u, \epsilon_0)$ with respect to which $\lambda_j^v$ has winding number $1$. Thus,
\begin{equation}\label{Riesz Projector of v curve center at u}
\mathbb{P}^{j}_{v} = \frac{1}{2\pi i} \oint_{\mathscr{C}(\lambda_j^u, \epsilon)} (\zeta - L_v)^{-1} \mathrm{d}\zeta, \qquad \lambda_j^v = \mathrm{Tr}(L_v\circ \mathbb{P}^{j}_{v} ), \qquad \forall v \in \mathcal{V}.
\end{equation}Since  $v \in \mathcal{V} \mapsto L_v \in \mathfrak{B}(H^1_+, L^2_+)$ is $\mathbb{R}$-affine and  $\mathbf{i} : \mathcal{A} \in \mathfrak{B}_{\mathfrak{I}}(H^1_+, L^2_+) \mapsto \mathcal{A}^{-1} \in \mathfrak{B}(L^2_+, H^1_+)$ is complex analytic, where $\mathfrak{B}_{\mathfrak{I}}(H^1_+, L^2_+)\subset \mathfrak{B}(H^1_+, L^2_+)$ denotes the open subset of all bijective bounded $\mathbb{C}$-linear transformations $H^1_+ \to L^2_+$, we have the real analyticity of the following map 
\begin{equation}\label{two variables zeta v}
(\zeta, v) \in \left( D(\lambda_j^u, \frac{3}{4}\epsilon_0) \backslash \overline{D}(\lambda_j^u, \frac{1}{2}\epsilon_0) \right) \times \mathcal{V} \mapsto (\zeta - L_v)^{-1} \in \mathfrak{B}(L^2_+, H^1_+).
\end{equation}Hence the maps $\mathbb{P}^{j} : v \in \mathcal{V} \mapsto \mathbb{P}^{j}_{v} \in \mathfrak{B}(L^2_+,  H^1_+)$ and $\lambda_j : v \in \mathcal{V} \mapsto \mathrm{Tr}(L_v\circ \mathbb{P}^{j}_{v} ) \in \mathbb{R}$ are both real analytic by composing $(\ref{Riesz Projector of v curve center at u})$ and $(\ref{two variables zeta v})$.
 
\end{proof}

\noindent Recall that $\mathscr{H}_{\mathrm{pp}}(L_u)= \frac{\mathbb{C}_{\leq N-1}[X]}{Q_u}$, where $Q_u$ denotes the characteristic polynomial of $u \in \mathcal{U}_N$ whose zeros are contained in $\mathbb{C}_-$, so $\mathscr{H}_{\mathrm{pp}}(L_u) \subset \mathbf{D}(G)$ is given by $(\ref{fourier transform of  basis sol of differential eq})$. We have the following consequence.
\begin{cor}\label{G varphi varphi analytic}
For every $j=1,2, \cdots, N$, the map $\mho_j : u \in \mathcal{U}_N \mapsto \langle G \varphi_j^u, \varphi_j^u \rangle_{L^2} \in \mathbb{C}$ is real analytic. 
\end{cor} 
\begin{proof}
For every $u,v  \in \mathcal{U}_N$, we have $\mathbb{P}_v^j \varphi_j^u = \langle \varphi_j^u, \varphi_j^v \rangle_{L^2}\varphi_j^v$. Since the Riesz projector $\mathbb{P}^{j} : v \in \mathcal{U}_N \mapsto \mathbb{P}^{j}_{v} \in \mathfrak{B}(L^2_+,  H^1_+)$ is real analytic in the proof of  proposition $\ref{prop eigenvalue is analytic on UN}$ and $\|\mathbb{P}_u^j \varphi_j^u\|_{L^2}=1$, there exists a  neighbourhood of $u$, denoted by $\mathcal{V}$, such that $\|\mathbb{P}_v^j \varphi_j^u\|_{L^2} > \frac{1}{2}$ for every $v \in \mathcal{V}$ and $\mathbb{P}^{j} : v \in \mathcal{V} \mapsto \mathbb{P}^{j}_{v} \in \mathfrak{B}(L^2_+,  H^1_+)$ can be expressed by power series. Then 
 \begin{equation*}
\varphi_j^v = \frac{\mathbb{P}_v^j \varphi_j^u}{ \langle \varphi_j^u, \varphi_j^v \rangle_{L^2}}, \qquad \mho_j(v)= \frac{ \langle  G \circ \mathbb{P}_v^j( \varphi_j^u), \mathbb{P}_v^j( \varphi_j^u) \rangle_{L^2}}{\|\mathbb{P}_v^j( \varphi_j^u)\|_{L^2}^2}.
 \end{equation*}Hence the restriction $\mho_j : v \in \mathcal{V} \mapsto \|\mathbb{P}_v^j( \varphi_j^u)\|_{L^2}^{-2}  \langle  G \circ \mathbb{P}_v^j( \varphi_j^u), \mathbb{P}_v^j( \varphi_j^u) \rangle_{L^2} \in \mathbb{C}$ is real analytic.
\end{proof}

 \bigskip

\subsection{Characterization theorem}\label{proof of thm characterization of multisoliton manifold}
The characterization theorem $\ref{characterization of multisoliton manifold}$ is proved in this subsection. The direct sense is given by proposition $\ref{Polynomial characterization of N soliton}$ and proposition $\ref{spectral decomposition caracterisation for Hpp Hac}$. Before proving the converse sense of theorem $\ref{characterization of multisoliton manifold}$, we need the following lemmas to prove the invariance of $\mathscr{H}_{\mathrm{pp}}(L_u)$ under $G$, if $u \in L^2(\mathbb{R}, (1+x^2)\mathrm{d}x)$ is real-valued, $\Pi u \in \mathscr{H}_{\mathrm{pp}}(L_u)$ and $\dim_{\mathbb{C}}\mathscr{H}_{\mathrm{pp}}(L_u)=N \geq 1$. The following lemma gives another version of formula of commutators (see also lemma $\ref{Domaine of G is invariant by Tb}$).  
 
 \begin{lem}\label{Lem of formula of commutator calculus G Tu T Du}
For $u \in L^2(\mathbb{R}, (1+x^2)\mathrm{d}x)$, $\varphi \in \mathrm{Ker}(\lambda - L_u)$ for some $\lambda \in \sigma_{\mathrm{pp}}(L_u)$, then we have $\varphi, T_u\varphi, L_u\varphi \in \mathbf{D}(G)$ and 
\begin{equation}\label{formula of commutator of G Lu}
\begin{split}
& [G, T_{ u}] \varphi = \frac{i\hat{\varphi}(0^+)}{2\pi}\Pi u  , \qquad 
  [G, L_{ u}] \varphi =i\varphi -\frac{i\hat{\varphi}(0^+)}{2\pi}\Pi u .\\
\end{split}
\end{equation}where $\Theta=\Theta_u=\frac{\overline{Q}_u}{Q_u}$ with $Q_u$ the characteristic polynomial of $u$.
\end{lem}

\begin{proof}
In proposition $\ref{proposition showing identity of wu}$, we have shown that $\widehat{u \varphi} \in H^1(\mathbb{R})$, so $(T_u \varphi)^{\wedge} =\widehat{u \varphi}  \mathbf{1}_{\mathbb{R}_+} \in H^1(0,+\infty)$ and $T_u \varphi \in \mathbf{D}(G)$. We recall the regularity of eigenfunctions $(\ref{regularity of Hpp functions})$ 
\begin{equation}\label{regularity of Hpp functions recalled}
\mathrm{Ker}(\lambda - L_u) \subset \{\varphi \in H^1_+ : \hat{\varphi}_{|\mathbb{R}_+} \in C^1(\mathbb{R}_+) \bigcap H^{1}(\mathbb{R}_+) \quad \mathrm{and} \quad \xi \mapsto \xi [ \hat{\varphi}(\xi)+  \partial_{\xi} \hat{\varphi}(\xi)] \in L^2(\mathbb{R}_+)\}.
\end{equation}So $G\varphi \in H^1_+ = \mathbf{D}(L_u)=\mathbf{D}(T_u)$. Moreover, we have $\hat{\varphi}$ is right-continuous at $\xi=0^+$ and $\hat{\varphi} \in C^1(0, +\infty)$. The weak-derivative of $\hat{\varphi} $ is  denoted by $\partial^{w}_{\xi} \hat{\varphi} $, $\delta_0$ denotes the Dirac measure with support $\{0\}$, then 
\begin{equation}
 \partial^{w}_{\xi} \hat{\varphi} = \mathbf{1}_{\mathbb{R}_+^* }  \frac{\mathrm{d}}{\mathrm{d}\xi}\hat{\varphi} + \hat{\varphi}(0^+) \delta_0,  \qquad \partial_{\xi}(\hat{u}*\hat{\varphi})= \partial^{w}_{\xi}(\hat{u}*\hat{\varphi}) =\hat{u}*\partial^{w}_{\xi} \hat{\varphi}
\end{equation}by lemma $\ref{convolution lemma p p'}$. Since $\hat{\varphi}  = \mathbf{1}_{\mathbb{R}_+^* }  \hat{\varphi}$ a.e. in $\mathbb{R}$ and $\hat{u} \in H^1(\mathbb{R})$,  we have $\hat{u}* \widehat{G \varphi}(\xi)= \hat{u}*[\mathbf{1}_{\mathbb{R}_+^* } \widehat{G\varphi}](\xi)$, for every $\xi >0$  and $([G, T_u]\varphi)^{\wedge}(\xi) =     \frac{i}{2\pi} \partial_{\xi} (\hat{u}* \hat{\varphi})( \xi) - \frac{i}{2\pi} \hat{u} * [\mathbf{1}_{\mathbb{R}_+^* }  \frac{\mathrm{d}}{\mathrm{d}\xi}\hat{f} ](\xi)  
= \frac{i}{2\pi} \hat{\varphi}(0^+) \widehat{u}(\xi)$. Together with $(\ref{varphi xi in terms of 1-Theta})$, the first formula of $(\ref{formula of commutator of G Lu})$ is obtained. Since $L_u= \mathrm{D}-T_u$, we claim that $\mathrm{D} \varphi \in \mathbf{D}(G)$. In fact, $\partial_{\xi} (\mathrm{D} \varphi)^{\wedge} (\xi) = \hat{\varphi}(\xi) + \xi \partial_{\xi} \hat{\varphi}(\xi)$, $\forall \xi >0$. Thus $(\ref{regularity of Hpp functions recalled})$ implies that $\widehat{\mathrm{D}\varphi }\in H^1(0, +\infty)$. Then
\begin{equation}\label{commutator of G and -i partial x}
\left( [G, \mathrm{D}]\varphi \right)^{\wedge}(\xi) = i \partial_\xi(\xi \hat{\varphi})(\xi) - \xi  \cdot i \partial_{\xi} \hat{\varphi}(\xi) = i \hat{\varphi}(\xi), \qquad \forall \xi >0.
\end{equation}So we have $[\partial_x, G]=  \mathrm{Id}_{L^2_+}$. The second formula of $(\ref{formula of commutator of G Lu})$ holds.
\end{proof}

\begin{prop}\label{Hpp is invariant by G}
If  $u \in L^2(\mathbb{R}, (1+x^2)\mathrm{d}x)$ is real-valued, $\dim_{\mathbb{C}} \mathscr{H}_{\mathrm{pp}}(L_u) =N \geq 1$ and $\Pi u \in \mathscr{H}_{\mathrm{pp}}(L_u)$,  then we have $\mathscr{H}_{\mathrm{pp}}(L_u) \subset \mathbf{D}(G)$ and  $G (\mathscr{H}_{\mathrm{pp}}(L_u)) \subset \mathscr{H}_{\mathrm{pp}}(L_u)$.
\end{prop}

\begin{proof} 
There exists an orthonormal basis of $ \mathscr{H}_{\mathrm{pp}}(L_u)$, denoted by $\{\psi_1, \psi_2, \cdots,  \psi_N\}$, such that
\begin{equation*}
L_u \psi_j = \lambda_j \psi_j, \quad \mathrm{where} \quad \sigma_{\mathrm{pp}}(L_u)=\{\lambda_1, \lambda_2, \cdots, \lambda_N\} \subset (-\infty, 0) , \quad \lambda_j < \lambda_{j+1} . 
\end{equation*}Since $(\ref{regularity of Hpp functions recalled})$ implies that $\mathscr{H}_{\mathrm{pp}}(L_u) \subset G^{-1}(H^1_+) \bigcap \mathbf{D}(G)$,  formula $(\ref{formula of commutator of G Lu})$ gives that  
\begin{equation*}
f_j:=[L_u, G]\psi_j =  -i \psi_j + \frac{i \hat{\psi}_j(0^+)}{2\pi} \Pi u \in \mathscr{H}_{\mathrm{pp}}(L_u), \qquad \forall j=1,2,\cdots, N.
\end{equation*}So we have $\langle f_j, \psi_j\rangle_{L^2} = \langle G \psi_j, L_u\psi_j\rangle_{L^2} - \langle G L_u \psi_j, \psi_j\rangle_{L^2}  = \lambda (\langle G \psi_j, \psi_j \rangle_{L^2} -\langle G \psi_j, \psi_j \rangle_{L^2})=0$. \\

\noindent For every $j=1,2,\cdots, N$, we set $g_j:= \sum_{1\leq k \leq N, k\ne j}    \frac{\langle f_j, \psi_k\rangle_{L^2}}{\lambda_k -\lambda_j} \psi_k$. Since $f_j = \sum_{1\leq k \leq N, k\ne j} \langle f_j, \psi_k\rangle_{L^2} \psi_k$, we have $(L_u - \lambda_j)g_j =f_j = (L_u - \lambda_j)G \psi_j$. Then $G \psi_j - g_j \in \mathrm{Ker}(L_u- \lambda_j) = \mathbb{C} \psi_j$ and
\begin{equation*}
G \psi_j \in g_j + \mathbb{C} \psi_j  \subset \mathscr{H}_{\mathrm{pp}}(L_u) .
\end{equation*}We conclude by $\mathscr{H}_{\mathrm{pp}}(L_u) = \mathrm{Span}_{\mathbb{C}}\{\psi_1$, $\psi_2, \cdots,  \psi_N\}$. (see also formulas $(\ref{k  Theta invariant by G})$ and   $(\ref{spectral subspace identification})$)
 \end{proof}

\noindent Now, we perform the proof of converse sense of theorem $\ref{characterization of multisoliton manifold}$ give the explicit formula of  $Q_u$.
\begin{proof}[End of the proof of theorem $\ref{characterization of multisoliton manifold}$]
\noindent $\Leftarrow$: Proposition $\ref{Hpp is invariant by G}$ yields that $G(\mathscr{H}_{\mathrm{pp}}(L_u)) \subset \mathscr{H}_{\mathrm{pp}}(L_u)$. Let $Q$ denote the characteristic polynomial of the operator $G|_{\mathscr{H}_{\mathrm{pp}}(L_u)}$, then we have $\mathscr{H}_{\mathrm{pp}}(L_u)  = \frac{\mathbb{C}_{\leq  N-1}[X]}{Q}$ by lemma $\ref{lemma of caracterization of G invariant subspace of finite dimension}$. So $\Pi u = \frac{\mathrm{P}_0}{Q}$, for some $\mathrm{P}_0 \in \mathbb{C}[X]$ such that $\deg \mathrm{P}_0 \leq N-1$. It remains to show that $\mathrm{P}_0=iQ'$. Since $\mathscr{H}_{\mathrm{pp}}(L_u)$ is invariant under $L_u$, for every $P \in \mathbb{C}_{\leq N-1}[X]$, we have 
\begin{equation*}
L_u (\frac{P}{Q})= (\mathrm{D}-T_{\frac{\mathrm{P}_0}{Q}} -T_{\frac{\overline{\mathrm{P}_0}}{\overline{Q}}})(\frac{P}{Q})= \frac{\mathrm{D}P}{Q} - \Pi (\frac{\overline{\mathrm{P}}_0 P}{\overline{Q}Q}) + \frac{(i Q'-\mathrm{P}_0)P }{Q^2} \in \frac{\mathbb{C}_{\leq  N-1}[X]}{Q}.
\end{equation*}Partial-fraction decomposition implies that $\Pi (\frac{\overline{\mathrm{P}}_0 P}{\overline{Q}Q}) \in \frac{\mathbb{C}_{\leq  N-1}[X]}{Q}$. So $\frac{(i Q'-\mathrm{P}_0) P  }{Q } \in \mathbb{C}_{\leq  N-1}[X]$ for every $P \in \mathbb{C}_{\leq  N-1}[X]$. Choose $P=\mathbf{1}$, since $\deg (iQ' - \mathrm{P}_0) \leq N-1$, we have $ \mathrm{P}_0 = iQ'$, so $u \in \mathcal{U}_N$. Since $Q \in \mathbb{C}_N[X]$ is monic and $Q^{-1}(0) \subset \mathbb{C}_-$, we have $Q_u(x)=Q(x)=\det(x-G|_{\mathscr{H}_{\mathrm{pp}}(L_u)})$.
\end{proof}

\bigskip

\subsection{The stability under the Benjamin--Ono flow}\label{subsection invariance}
Finally we prove proposition $\ref{Invariance of UN under BO flow}$ in this subsection. Two lemmas will be proved at first in order to obtain the invariance of the property $x \mapsto xu(x) \in L^2(\mathbb{R})$   under the BO flow.
\begin{lem}\label{Lem of invariance of x u(t,x) in L2x}
If $u_0 \in H^{2}(\mathbb{R}, \mathbb{R}) \bigcap L^2(\mathbb{R}, x^2 \mathrm{d}x)$, let $u=u(t,x)$ solves the BO equation $(\ref{Benjamin Ono equation on the line})$ with initial datum $u(0)=u_0$, then $u(t) \in L^2(\mathbb{R}, x^2 \mathrm{d}x)$, for every $t \in \mathbb{R}$.
\end{lem}

\begin{rem}
This result can be strengthened by replacing the assumption $u_0 \in H^{2}(\mathbb{R}, \mathbb{R})$ by a weaker assumption $u_0 \in  H^{\frac{3}{2} +}(\mathbb{R}, \mathbb{R}) = \bigcup_{s >  \frac{3}{2}}H^s(\mathbb{R}, \mathbb{R})$, because one can construct the conservation law of BO equation controlling the $H^s$-norm for every $s>-\frac{1}{2}$ by using the method of perturbation of determinants. We refer to Talbut $[\ref{Talbut Low regularity conservation laws}]$ to see details and Killip--Vi\c{s}an--Zhang $[\ref{Killip Visan Zhang Low regularity conservation laws for integrable PDE}]$ for the KdV and the NLS cases (see also Koch--Tataru $[\ref{Koch Tataru Conserved energies for the cubic NLS}]$). It suffices to use lemma $\ref{Lem of invariance of x u(t,x) in L2x}$ to prove proposition $\ref{Invariance of UN under BO flow}$.
\end{rem}

\noindent Before proving lemma $\ref{Lem of invariance of x u(t,x) in L2x}$, we need some commutator estimates used in G\'erard--Lenzmann--Pocovnicu--Rapha\"el $[\ref{Gerard Lenzmann Pocovnicu Raphael A two-soliton with transient turbulent regime}]$, we recall it here.

\begin{lem}\label{Lemma of commutator estimates}
For a general locally  Lipschitz function $\chi : \mathbb{R} \to \mathbb{R}$ such that $\partial_x \chi, \partial_x^3 \chi, \partial_x^5 \chi \in L^1(\mathbb{R})$, then we have the following commutator estimates
\begin{equation}\label{commutator estimates of D chi}
\begin{split}
&\|[|\mathrm{D}|, \chi]g\|_{L^2} + \|[\partial_x, \chi]g\|_{L^2} \lesssim  (\|\partial_x \chi\|_{L^1}\|\partial_x^3 \chi\|_{L^1})^{\frac{1}{2}} \|g\|_{L^2}, \qquad\qquad\qquad\qquad\qquad \forall g \in \L^2(\mathbb{R}),\\
&\||\mathrm{D}|[\partial_x , \chi]g\|_{L^2}   \lesssim  (\|\partial_x \chi\|_{L^1}\|\partial_x^3 \chi\|_{L^1})^{\frac{1}{2}} \|\partial_x g\|_{L^2} + (\|\partial_x \chi\|_{L^1}\|\partial_x^5 \chi\|_{L^1})^{\frac{1}{2}} \| g\|_{L^2}, \qquad \forall g \in H^1(\mathbb{R}).
\end{split}
\end{equation}
\end{lem}

\begin{proof}
We use $\big||\xi| - |\eta| \big| \leq  | \xi  -  \eta  |$ to estimate the Fourier modes of $[|\mathrm{D}|, \chi]g$.
\begin{equation*}
2\pi \Big|\left( [|\mathrm{D}|, \chi]g \right)^{\wedge} (\xi) \Big| \leq    \int_{\eta \in \mathbb{R}}\big||\xi| - |\eta| \big| |\hat{\chi}(\xi - \eta)| |\hat{g}(\eta)| \mathrm{d}\eta \leq   \int_{\eta \in \mathbb{R}}| \xi  -  \eta  | |\hat{\chi}(\xi - \eta)| |\hat{g}(\eta)| \mathrm{d}\eta = |\widehat{\partial_x \chi}| * |\hat{g}|(\xi).
\end{equation*}Then Young's convolution inequality yields that $\|[|\mathrm{D}|, \chi]g\|_{L^2} \lesssim \|\widehat{\partial_x \chi}| * |\hat{g}|\|_{L^2} \lesssim \|\widehat{\partial_x \chi}\|_{L^1} \|g \|_{L^2}$. In order to estimate $\|\widehat{\partial_x \chi}\|_{L^1}$, we divide the integral as two parts. Wet set $\mathcal{R}_1=\|\partial_x \chi\|_{L^1}^{-\frac{1}{2}}\|\partial_x^3 \chi\|_{L^1}^{\frac{1}{2}}$, so 
\begin{equation*}
\|\widehat{\partial_x \chi}\|_{L^1} \leq  \|\widehat{\partial_x \chi}\|_{L^{\infty}}\int_{|\xi|\leq \mathcal{R}_1}\mathrm{d}\xi + \int_{|\xi|> \mathcal{R}_1} \frac{\|\widehat{\partial_x^3 \chi}\|_{L^{\infty}}}{|\xi|^2}\mathrm{d}\xi \lesssim \| \partial_x \chi \|_{L^{1}}\mathcal{R}_1 + \frac{\| \partial_x^3 \chi \|_{L^1}}{\mathcal{R}_1} = (\|\partial_x \chi\|_{L^1}\|\partial_x^3 \chi\|_{L^1})^{\frac{1}{2}}.
\end{equation*}Similarly, we have $\|[\partial_x, \chi]g\|_{L^2} \lesssim \|\widehat{\partial_x \chi}\|_{L^1} \|g \|_{L^2} \lesssim (\|\partial_x \chi\|_{L^1}\|\partial_x^3 \chi\|_{L^1})^{\frac{1}{2}}$. Thus $(\ref{commutator estimates of D chi})$ is obtained.
\begin{equation*}
\begin{split}
2\pi \Big|\left(|\mathrm{D}|[\partial_x , \chi]g \right)^{\wedge} (\xi) \Big| \leq &  |\xi|\int_{\eta \in \mathbb{R}} |\xi -  \eta|   |\hat{\chi}(\xi - \eta)| |\hat{g}(\eta)| \mathrm{d}\eta \\
\leq & \int_{\eta \in \mathbb{R}} |\xi  -  \eta|^2\big| |\hat{\chi}(\xi - \eta)| |\hat{g}(\eta)| \mathrm{d}\eta + \int_{\eta \in \mathbb{R}} |\xi  -  \eta|   |\hat{\chi}(\xi - \eta)| |\eta| |\hat{g}(\eta)| \mathrm{d}\eta\\
= & |\widehat{\partial_x^2 \chi}| * |\hat{g}|(\xi) + |\widehat{\partial_x \chi}| * |\widehat{\partial_x g}|(\xi)
\end{split}
\end{equation*}So we have $\||\mathrm{D}|[\partial_x , \chi]g\|_{L^2}   \lesssim  \||\widehat{\partial_x^2 \chi}| * |\hat{g}|\|_{L^2} + \||\widehat{\partial_x \chi}| * |\widehat{\partial_x g}|\|_{L^2} \lesssim \|\widehat{\partial_x^2 \chi}\|_{L^1} \| g \|_{L^2} + \|\widehat{\partial_x \chi}\|_{L^1} \|\partial_x g \|_{L^2}$. Then we use the same idea to estimate $\|\widehat{\partial_x^2 \chi}\|_{L^1}$, we set $\mathcal{R}_2 :=\|\partial_x \chi\|_{L^1}^{-\frac{1}{4}}\|\partial_x^5 \chi\|_{L^1}^{\frac{1}{4}}$. Thus,
\begin{equation*}
\|\widehat{\partial_x^2 \chi}\|_{L^1} \leq  \|\widehat{\partial_x \chi}\|_{L^{\infty}}\int_{|\xi|\leq \mathcal{R}_1}|\xi|\mathrm{d}\xi + \int_{|\xi|> \mathcal{R}_1} \frac{\|\widehat{\partial_x^5 \chi}\|_{L^{\infty}}}{|\xi|^3}\mathrm{d}\xi \lesssim \| \partial_x \chi \|_{L^{1}}\mathcal{R}_2^2 + \frac{\| \partial_x^5 \chi \|_{L^1}}{\mathcal{R}_2^2} = (\|\partial_x \chi\|_{L^1}\|\partial_x^5 \chi\|_{L^1})^{\frac{1}{2}}.
\end{equation*}Finally, we add them together to get the second estimate in $(\ref{commutator estimates of D chi})$.
\end{proof}

\noindent Now we prove the invariance of the property $x \mapsto xu(x) \in L^2(\mathbb{R})$ is invariant under the BO flow.
\begin{proof}[Proof of lemma $\ref{Lem of invariance of x u(t,x) in L2x}$]
We choose a cut-off function $\chi \in C^{\infty}_c(\mathbb{R})$ such that $\chi$ decreases in $[0, +\infty)$, $\chi$ is even and
\begin{equation}\label{properties of chi cut off function}
0 \leq \chi \leq 1, \qquad \chi \equiv 1 \quad \mathrm{on} \quad [-1, 1], \qquad \mathrm{supp}(\chi) \subset [-2, 2].
\end{equation}If $u_0 \in H^{2}(\mathbb{R}) \bigcap L^2(\mathbb{R}, x^2 \mathrm{d}x)$, we claim that there exists a constant $\mathcal{C}=\mathcal{C}(\|u(0)\|_{H^1})$ such that
\begin{equation}\label{claim of cut off of integral}
I(R,t):= \int_{\mathbb{R}}\chi^2(\tfrac{x}{R})|x|^2 |u(t,x)|^2 \mathrm{d}x \leq \mathcal{C} e^{|t|}( \int_{\mathbb{R}} |x|^2 |u(0,x)|^2 \mathrm{d}x +1), \qquad \forall t \in \mathbb{R}, \quad \forall R>1,
\end{equation}if $u$ solves the BO equation $\partial_t u = \mathrm{H} \partial_x^2 u -\partial_x (u^2)=|\mathrm{D}|\partial_x u - 2u \partial_x u$.\\

\noindent In fact, we define $\rho(x):=x \chi(x)$. For every $R>0$, we set $\rho_R(x):=R \rho(\tfrac{x}{R})=x\chi (\tfrac{x}{R})$. Thus
\begin{equation*}
\partial_t I(R,t)= 2 \mathrm{Re}\langle \rho^2_R \partial_t u(t), u(t) \rangle_{L^2} = 2 \mathrm{Re}\langle \rho^2_R |\mathrm{D}|\partial_x u(t) - 2\rho^2_R u(t) \partial_x u(t), u(t) \rangle_{L^2} = \mathcal{J}_1(u(t))+\mathcal{J}_2(u(t)),
\end{equation*}where for every $u \in H^2(\mathbb{R})$, we define
\begin{equation}\label{formula J1}
\mathcal{J}_1(u):= -4 \mathrm{Re}\langle \rho^2_R u \partial_x u, u \rangle_{L^2} \Longrightarrow |\mathcal{J}_1(u)|\leq 4 \|\partial_x u\|_{L^{\infty}} \|\rho_R u\|_{L^2}^2 \lesssim \|u\|_{H^2}\|\rho_R u\|_{L^2}^2
\end{equation}and
\begin{equation*}
\mathcal{J}_2(u):=2 \mathrm{Re}\langle \rho^2_R |\mathrm{D}|\partial_x u , u  \rangle_{L^2} =  \langle [\rho^2_R, |\mathrm{D}|\partial_x] u , u  \rangle_{L^2},
\end{equation*}because $|\mathrm{D}|\partial_x = - (|\mathrm{D}|\partial_x)^*$ is an unbounded skew-adjoint operator on $L^2(\mathbb{R})$, whose domain of definition is $H^2(\mathbb{R})$,  $u \mapsto \rho_R u$ is a bounded self-adjoint operator on $H^s(\mathbb{R})$, for every $s \geq 0$. Since
\begin{equation*}
[\rho_R^2, |\mathrm{D}|\partial_x] =\rho_R [\rho_R , |\mathrm{D}|\partial_x] + [\rho_R , |\mathrm{D}|\partial_x] \rho_R, \quad [\rho_R , |\mathrm{D}|\partial_x] =[\rho_R , |\mathrm{D}|\partial_x]^* = [\rho_R , |\mathrm{D}| ]\partial_x + |\mathrm{D}|  [\rho_R , \partial_x  ],
\end{equation*}we have
\begin{equation}
\begin{split}
\mathcal{J}_2(u) 
= &\langle  \rho_R [\rho_R , |\mathrm{D}|\partial_x]  u + [\rho_R , |\mathrm{D}|\partial_x] \rho_R   u , u  \rangle_{L^2} \\
 =  & 2 \mathrm{Re}\langle  [\rho_R , |\mathrm{D}|\partial_x]   u ,\rho_R u  \rangle_{L^2}  \\
=  & 2 \mathrm{Re}\langle    [\rho_R , |\mathrm{D}| ]\partial_x u ,\rho_R u  \rangle_{L^2} + 2 \mathrm{Re}\langle   |\mathrm{D}|  [\rho_R , \partial_x  ]  u ,\rho_R u  \rangle_{L^2}.
\end{split}
\end{equation}Since $\|\partial_x \rho_R\|_{L^1}=R \|\partial_x \rho \|_{L^1}$, $\|\partial_x^3 \rho_R\|_{L^1}=R^{-1} \|\partial_x \rho \|_{L^1}$ and $\|\partial_x^5 \rho_R\|_{L^1}=R^{-3} \|\partial_x \rho \|_{L^1}$, the commutator estimates $(\ref{commutator estimates of D chi})$  yield  that if  $u \in H^2(\mathbb{R})$, then
\begin{equation}\label{J2 estimate}
\begin{split}
  |\mathcal{J}_2(u) | 
\leq  & 2 \|\rho_R u\|_{L^2}^2 + \| [\rho_R , |\mathrm{D}| ]\partial_x u   \|_{L^2}^2+ \| |\mathrm{D}|  [\rho_R , \partial_x  ]  u\|_{L^2}^2 \\
\lesssim & \|\rho_R u\|_{L^2}^2 +   \|\partial_x \rho_R\|_{L^1}\|\partial_x^3 \rho_R\|_{L^1} \|\partial_x u   \|_{L^2}^2+     \|\partial_x \rho_R\|_{L^1}\|\partial_x^5 \rho_R\|_{L^1}  \| u\|_{L^2}^2\\
\lesssim & \|\rho_R u\|_{L^2}^2 +   \|\partial_x \rho  \|_{L^1}\|\partial_x^3 \rho \|_{L^1} \|\partial_x u   \|_{L^2}^2+  R^{-2} \|\partial_x \rho \|_{L^1}\|\partial_x^5 \rho \|_{L^1}  \| u\|_{L^2}^2\\
\lesssim & \|\rho_R u\|_{L^2}^2 +\|u\|_{H^1}^2
\end{split}
\end{equation}for every $R \geq 1$. Proposition $\ref{prop of conservation law controling every sobolev norms}$ and $\ref{Conservation law sequence controling h n/2 norm}$ yield that there exists a conservation law of $(\ref{Benjamin Ono equation on the line})$ controlling $H^2$-norm of the solution. Let $u : t \in \mathbb{R} \mapsto u(t) \in H^2(\mathbb{R})$ denote the solution of the BO equation $(\ref{Benjamin Ono equation on the line})$. Then $\sup_{t \in \mathbb{R}} \|u(t)\|_{H^2} \lesssim_{\|u_0\|_{H^2}} 1 $. Since $I(R,t)=\|\rho_R u(t)\|_{L^2}^2$, estimates $(\ref{formula J1})$ and $(\ref{J2 estimate})$ imply that
\begin{equation*}
|\partial_t I(R,t)| \leq \mathcal{C} (I(R,t) +1) ,  \qquad t \in \mathbb{R},
\end{equation*}for some constant $\mathcal{C}=\mathcal{C}(\|u_0\|_{H^2}) $. Thus $(\ref{claim of cut off of integral})$ is obtained by Gronwall's inequality. Let $R \to+\infty$, we conclude by using Lebesgue's monotone convergence theorem.
\end{proof}

\bigskip

\noindent Since the generating function $\lambda \in \mathbb{C}\backslash \sigma(-L_u) \mapsto \mathcal{H}_{\lambda}(u) \in \mathbb{C}$ is the Borel--Cauchy transform of the spectral measure of $L_u$, the invariance of the $N-$soliton manifold $\mathcal{U}_N$ under BO flow is obtained by using the inverse spectral transform.

\begin{proof}[End of the proof of proposition $\ref{Invariance of UN under BO flow}$]
If $u_0 \in \mathcal{U}_N \subset H^{\infty}(\mathbb{R}, \mathbb{R}) \bigcap L^2(\mathbb{R}, x^2 \mathrm{d}x)$, let $u=u(t,x)$ be the unique solution of the BO equation $(\ref{Benjamin Ono equation on the line})$ with initial datum $u(0)=u_0$, then  $u(t)\in H^{\infty}(\mathbb{R}, \mathbb{R}) \bigcap L^2(\mathbb{R}, x^2 \mathrm{d}x)$ by proposition $\ref{GWP for H^s solution of BO eq}$ and lemma $\ref{Lem of invariance of x u(t,x) in L2x}$. Recall the generating function $\mathcal{H}_{\lambda} : u  \in L^2(\mathbb{R}, \mathbb{R}) \to \mathbb{R}$ defined as
\begin{equation}
\mathcal{H}_{\lambda}(u)=\langle (\lambda+L_u)^{-1} \Pi u, \Pi u\rangle_{L^2} = \int_{\mathbb{R}} \frac{\mathrm{d}\mathbf{m}_u(\xi)}{\xi+\lambda},\quad \mathbf{m}_u:=\mu_{\Pi u}^{L_u},  \qquad \forall \lambda \in \mathbb{C}\backslash \sigma( - L_u), 
\end{equation}where $\mu_{\psi}^{L_u}$ denotes the spectral measure of $L_u$ associated to the function $\psi \in L^2_+$. So the holomorphic function $\lambda \in \mathbb{C}\backslash \sigma (- L_u) \mapsto \mathcal{H}_{\lambda}u $ is the Borel--Cauchy transform of the positive Borel measure $\mathbf{m}_u$. We recall that the total variation $\mathbf{m}_u(\mathbb{R})=\|\Pi u\|_{L^2}^2$ is a conservation law of the BO equation $(\ref{Benjamin Ono equation on the line})$ by proposition $\ref{Conservation law sequence controling h n/2 norm}$ and formula $(\ref{Identity of Pi u along the Benjamin ono equation})$. Every finite Borel measure is uniquely determined by its Borel--Cauchy transform (see Theorem 3.21 of Teschl $[\ref{Teschl spectral theory book}]$ page 108), precisely for every $a\leq b$ real numbers, we use Stieltjes inversion formula to obtain that
\begin{equation*}
\frac{1}{2}\mathbf{m}_u((a,b)) +\frac{1}{2}\mathbf{m}_u([a,b]) = -\frac{1}{\pi }\lim_{\epsilon \to 0^+} \int_{a}^b \mathrm{Im}\mathcal{H}_{x+i \epsilon}(u) \mathrm{d}x. 
\end{equation*}For every $t \in \mathbb{R}$, proposition $\ref{Conservation law of the generating function H lambda}$ yields that $\mathcal{H}_{\lambda}[u(t)]=\mathcal{H}_{\lambda}[u(0)]$, $\forall\lambda \in \mathbb{C}\backslash \sigma_{\mathrm{pp}}(L_{u(0)})=\mathbb{C}\backslash \sigma_{\mathrm{pp}}(L_{u(t)})$. Since $u(0) \in\mathcal{U}_N$, we have $\Pi [u(0)] \in \mathscr{H}_{\mathrm{pp}}(L_{u(0)})$ by proposition $\ref{spectral decomposition caracterisation for Hpp Hac}$ and there exist $c_1, c_2, \cdots, c_N \in \mathbb{R}_+$ such that
\begin{equation*}
\mu_{\Pi [u(t)]}^{L_{u(t)}}=\mathbf{m}_{u(t)}=\mathbf{m}_{u(0)}= \mu_{\Pi [u(0)]}^{L_{u(0)}} = \sum_{j=1}^N c_j \delta_{\lambda_j^{u(0)}} .
\end{equation*}The spectral measure $\mu_{\Pi [u(t)]}^{L_{u(t)}}$ is purely point, so $\Pi [u(t)] \in \mathscr{H}_{\mathrm{pp}}(L_{u(t)})$ for every $t \in \mathbb{R}$. The Lax pair structure yields the unitary equivalence between $L_{u(t)}$ and $L_{u(0)}$. So $\dim_{\mathbb{C}}\mathscr{H}_{\mathrm{pp}}(L_{u(t)}) = \dim_{\mathbb{C}}\mathscr{H}_{\mathrm{pp}}(L_{u(0)})=N$ is given by proposition $\ref{proposition on unitary equivalence}$. We conclude by theorem $\ref{characterization of multisoliton manifold}$.
\end{proof}

\bigskip
\bigskip

\section{The generalized action--angle coordinates}\label{section of action angle coordinates}
In this section, we construct the (generalized) action--angle coordinates $\Phi_N$ in theorem $\ref{principal theorem of this paper}$ of the BO equation $(\ref{Hamiltonian energy of BO eq and Hamiltonian form of BO})$ with solutions in the real analytic symplectic manifold $(\mathcal{U}_N, \omega)$ of real dimension $2N$ given in proposition $\ref{prop differential structure of UN}$. The goal of this section is to establish the diffeomorphism property and the symplectomorphism property of $\Phi_N$.\\

\noindent Recall that the BO equation with $N$-soliton solutions is identified as a globally well-posed Hamiltonian system reading as
\begin{equation}\label{Hamiltonian system expression of BO eq}
\partial_t u(t) = X_{E}(u(t)),  \qquad u(t) \in \mathcal{U}_N,
\end{equation}whose energy functional $E(u)=\langle L_u \Pi u, \Pi u\rangle_{L^2}$ is well defined on $ \mathcal{U}_N$ and the Hamiltonian vector field $X_E : u \in \mathcal{U}_N \mapsto  X_E (u) = \partial_x(|\mathrm{D}|u  -  u^2) \in \mathcal{T}_u(\mathcal{U}_N)$ coincides with the definition $(\ref{formula to obtain relation between XF and nabla F})$. The Poisson bracket of two  smooth functions $f, g :\mathcal{U}_N \to \mathbb{R}$ is given by
\begin{equation}\label{Poisson bracket in intro of GAA coordinates}
 \{f,g\} : u \in \mathcal{U}_N \mapsto \omega_u(X_f(u), X_g(u))=\langle \partial_x \nabla_u f(u), \nabla_u g(u) \rangle_{L^2}\in \mathbb{R}.
\end{equation} Given $u\in \mathcal{U}_N$, proposition $\ref{spectral decomposition caracterisation for Hpp Hac}$ yields that there exist $\lambda_1^u<\lambda_2^u< \cdots<\lambda_N^u<0$ and $\varphi_j^u \in \mathrm{Ker}(\lambda_j^u - L_u) \subset \mathbf{D}(G)$ such that $\|\varphi_j^u\|_{L^2}=1$ and $\langle u, \varphi_j^u \rangle_{L^2} = \sqrt{2\pi|\lambda_j^u|}$, thanks to the spectral analysis in subsection $\mathbf{\ref{subsection proof of spectral decomposition caracterisation for Hpp Hac}}$.
\begin{defi}\label{j th action and angle}
For every $j=1,2, \cdots, N$, the map $I_j :u \in \mathcal{U}_N \mapsto 2 \pi \lambda_j^u \in \mathbb{R}$ is called the j th action. The map $\gamma_j : u \in \mathcal{U}_N \mapsto \mathrm{Re}\langle G \varphi_j^u,  \varphi_j^u\rangle_{L^2} \in \mathbb{R}$ is called the j th (generalized) angle.
\end{defi}

\noindent Set $\Omega_N : = \{(r^1, r^2, \cdots, r^N) \in \mathbb{R}^N : r^1 < r^2 < \cdots < r^N  < 0\} \subset \mathbb{R}^N$, the canonical symplectic form on  $\mathbb{R}^{2N} = \{(r^1, r^2, \cdots, r^N; \alpha^1, \alpha^2, \cdots, \alpha^N) : \forall r^j, \alpha^j \in \mathbb{R} \}$ is given by $\nu =\sum_{j=1}^N \mathrm{d}r^j \wedge \mathrm{d}\alpha^j$. Endowed with the subspace topology and the embedded real analytic structure of $\mathbb{R}^{2N}$, the submanifold $(\Omega_N \times \mathbb{R}^N, \nu)$ is a symplectic manifold of real dimension $2N$. The action--angle map is defined by
\begin{equation}\label{definition of action angle mapping}
\Phi_N : u \in \mathcal{U}_N \mapsto (I_1(u), I_2(u), \cdots, I_N(u); \gamma_1(u), \gamma_2(u), \cdots, \gamma_N(u)) \in \Omega_N \times \mathbb{R}^N.
\end{equation}Theorem $\ref{principal theorem of this paper}$ is restated here.

\begin{thm}\label{action angle coordinate thm}
The map $\Phi_N$ has following properties:  \\

\noindent $(\mathrm{a})$. The map $\Phi_N  :  \mathcal{U}_N \to \Omega_N \times \mathbb{R}^N$ is a real  analytic diffeomorphism. \\
\noindent $(\mathrm{b})$. The pullback of $\nu$ by $\Phi_N$ is $\omega$, i.e. $\Phi_N^* \nu = \omega$. \\
\noindent $(\mathrm{c})$. We have $E\circ\Phi_N^{-1} : (r^1, r^2, \cdots, r^N; \alpha^1, \alpha^2, \cdots, \alpha^N) \in  \Omega_N \times \mathbb{R}^N \mapsto  - \frac{1}{2\pi}\sum_{j=1}^N| r^j|^2 \in (-\infty, 0)$. \\
\end{thm}

\begin{rem}
The real analyticity of $\Phi_N  :  \mathcal{U}_N \to \Omega_N \times \mathbb{R}^N$ is given by proposition $\ref{prop eigenvalue is analytic on UN}$ and corollary $\ref{G varphi varphi analytic}$. The symplectomorphism property $(\mathrm{b})$ is equivalent to the following Poisson bracket characterization (see proposition $\ref{Property symplectomorphism property Poisson bracket}$)
\begin{equation}\label{Poisson bracket all}
\{I_j, I_k\} = 0 , \quad \{I_j, \gamma_k\}  = \mathbf{1}_{j=k}, \quad \{\gamma_j, \gamma_k\} =0  \quad \mathrm{on} \quad \mathcal{U}_N, \qquad \forall j, k =1,2, \cdots, N.
\end{equation}The family $(X_{I_1}, X_{I_2} , \cdots, X_{I_N} ; X_{\gamma_1} , X_{\gamma_2} , \cdots, X_{\gamma_N} )$ is linearly independent in $\mathfrak{X}(\mathcal{U}_N)$ and we have   
\begin{equation*}
\mathrm{d}\Phi_N(u) : X_{I_k}(u) \mapsto \frac{\partial}{\partial \alpha^k}  \Big|_{\Phi_N(u)}, \qquad \mathrm{d}\Phi_N(u) : X_{\gamma_k}(u) \mapsto -\frac{\partial}{\partial r^k}\Big|_{\Phi_N(u)}.
\end{equation*}The assertion $(\mathrm{c})$ is obtained by a direct calculus: $\Pi u = \sum_{j=1}^N \langle \Pi u, \varphi_j^u \rangle_{L^2}\varphi_j^u  $, formula $(\ref{definition of varphi j u a basis of H pp})$ yields that 
\begin{equation*}
E(u) = \langle L_u (\Pi u), \Pi u\rangle_{L^2} = \sum_{j=1}^N |\langle \Pi u, \varphi_j^u \rangle_{L^2}|^2 \lambda_j^u =-\sum_{j=1}^N \frac{I_j(u)^2}{2\pi}.
\end{equation*}Thus theorem $\ref{action angle coordinate thm}$ introduces (generalized) action--angle coordinates of the BO equation $(\ref{Hamiltonian system expression of BO eq})$ in the sense of $(\ref{Action angle meaning Poisson brackets})$, i.e. $\{I_j, E\}(u)=0$ and $\{\gamma_j, E\}(u)=2 \lambda_j^u$, for every $u \in \mathcal{U}_N$. 
\end{rem} 

\noindent This section is organized as follows. The matrix associated to $G|_{\mathscr{H}_{\mathrm{pp}}(L_u)}$ is expressed in terms of actions and angles in subsection $\mathbf{\ref{subsection Matrix }}$. Then the injectivity of $\Phi_N$ is given by  inversion formulas in subsection $\mathbf{\ref{subsection of inversion formula}}$. In subsection $\mathbf{\ref{subsection poisson brackets}}$, the Poisson brackets of  actions and  angles are used to show the local diffeomorphism property of $\Phi_N$. The surjectivity of $\Phi_N$ is obtained by Hadamard's global inverse theorem in subsection $\mathbf{\ref{subsection diffeo property}}$. Finally, we use subsection $\mathbf{\ref{subsection lagrangian section}}$ and subsection $\mathbf{\ref{subsection symplectomorphism}}$ to prove that $\Phi_N : (\mathcal{U}_N, \omega) \to (\Omega_N \times \mathbb{R}^N, \nu)$ preserves the symplectic structure.

\bigskip

\subsection{The associated matrix }\label{subsection Matrix }
We continue to study the infinitesimal generator $G$ defined in $(\ref{definition of G})$ when restricted to the invariant subspace $\mathscr{H}_{\mathrm{pp}}(L_u)$ with complex dimension $N$. Let $M(u)=(M_{kj}(u))_{1\leq k, j \leq N}$ denote the matrix associated to the operator $G|_{\mathscr{H}_{\mathrm{pp}}(L_u)}$ with respect to the basis $\{\varphi_1^u, \varphi_2^u, \cdots, \varphi_N^u\}$.  Then we state a general linear algebra lemma that describes the location of eigenvalues of the matrix $M(u)$.

\begin{prop}\label{Coefficients of Matrix of G}
For every $u \in \mathcal{U}_N$, the coefficients of matrix $M(u)=(M_{kj}(u))_{1\leq k, j \leq N}$ are given by 
\begin{equation}\label{coefficients of matrix of G def}
M_{k j}(u)= \langle G \varphi_j^u, \varphi_k^u \rangle_{L^2} = \begin{cases}
\frac{i}{\lambda_k^u -  \lambda_j^u} \sqrt{\frac{|\lambda_k^u|}{|\lambda_j^u|}}, \qquad \mathrm{if}\quad j\ne k,\\
  \gamma_j(u) - \frac{i}{2 |\lambda_j^u|}, \qquad \mathrm{if}\quad j = k.
\end{cases}
\end{equation} 
\end{prop}
 
\begin{proof} 
Since $L_u$ is a self-adjoint operator on $L^2_+$ and $\mathscr{H}_{\mathrm{pp}}(L_u) \subset \mathbf{D}(G)$, we have
\begin{equation*}
(\lambda_j^u -  \lambda_k^u )M_{k j}(u) = \langle G L_u\varphi_j^u, \varphi_k^u \rangle_{L^2} - \langle G \varphi_j^u, L_u\varphi_k^u \rangle_{L^2}= \langle [G, L_u ]\varphi_j^u, \varphi_k^u \rangle_{L^2}.
\end{equation*}Since formulas $(\ref{Fourier modes of Lu phi = lambda phi})$ and $(\ref{definition of varphi j u a basis of H pp})$ imply that $-\lambda_j^u \widehat{\varphi_j^u}(0) = \widehat{u \varphi_j^u}(0) =\sqrt{2\pi |\lambda_j^u|}$, we use $(\ref{formula of commutator of G Lu})$ to obtain
\begin{equation*}
(\lambda_j^u -  \lambda_k^u )M_{k j}(u) = \langle i \varphi_j^u - \frac{i}{2 \pi} \widehat{\varphi_j^u}(0^+)\Pi u, \varphi_k^u \rangle_{L^2} =  - \frac{i}{2 \pi} \widehat{\varphi_j^u}(0^+) \overline{\widehat{u \varphi_k^u}}(0) = -i \sqrt{\tfrac{|\lambda_k^u|}{|\lambda_j^u|}}.
\end{equation*}In the case $k=j$, we use Plancherel formula and integration by parts to calculate
\begin{equation*}
\langle G^* f, g \rangle_{L^2} = \langle   f, G g \rangle_{L^2} =-\tfrac{i}{2\pi}\int_0^{+\infty}\hat{f}(\xi) \partial_{\xi}\overline{\hat{g}}(\xi) \mathrm{d} \xi = \tfrac{i}{2\pi}\left[\hat{f}(0^+)\overline{\hat{g}}(0^+) + \int_0^{+\infty}\partial_{\xi} \hat{f}(\xi) \overline{\hat{g}}(\xi) \mathrm{d} \xi \right]
\end{equation*}Thus we have $\langle G^* f, g \rangle_{L^2} = \langle G  f, g \rangle_{L^2} + \frac{i}{2\pi}\hat{f}(0^+)\overline{\hat{g}}(0^+)$, for every $f, g \in \mathscr{H}_{\mathrm{pp}}(L_u)$. Then
\begin{equation*}
\mathrm{Im}M_{j j}(u) = \frac{1}{2i}(\langle G  \varphi_j^u, \varphi_j^u \rangle_{L^2}-\langle G^* \varphi_j^u, \varphi_j^u \rangle_{L^2})= -\frac{|\widehat{\varphi_j^u }(0)|^2}{4\pi} = -\frac{1}{2 |\lambda_j^u|}.
\end{equation*}We conclude by $\gamma_j(u) = \mathrm{Re}\mho_j(u) = \langle G \varphi_j^u, \varphi_j^u\rangle_{L^2}$ defined in corollary $\ref{G varphi varphi analytic}$. 
\end{proof}

\bigskip

\noindent Then we state a linear algebra lemma that describe the location of spectrum of all matrices of the form defined as $(\ref{coefficients of matrix of G def})$.
\begin{lem}\label{lemma to get negative definiteness of Im M}
For every $N\in \mathbb{N}_+$, we choose $N$ negative numbers $\lambda_1 <\lambda_2< \cdots <\lambda_N < 0$ and $N$ real numbers $\gamma_1, \gamma_2, \cdots, \gamma_N \in \mathbb{R}$. The matrix $\mathcal{M}=(\mathcal{M}_{kj})_{1\leq k, j  \leq N} \in \mathbb{C}^{N\times N}$ is defined as  
\begin{equation}\label{general matrix M to prove sp in C negative}
\mathcal{M}_{kj}=
\begin{cases}\frac{i}{\lambda_k- \lambda_j} \sqrt{\frac{|\lambda_k|}{|\lambda_j|}},\qquad \mathrm{if} \quad k \ne j,\\
\gamma_j - \frac{i}{2|\lambda_j|}, \qquad \quad \mathrm{if} \quad k=j.
\end{cases}
\end{equation}Then $\mathrm{Im}\mathcal{M} = \frac{\mathcal{M}-\mathcal{M}^*}{2i}$ is negative semi-definite and $\sigma_{\mathrm{pp}} (\mathcal{M})\subset \mathbb{C}_-$. Furthermore, the map 
\begin{equation*}
(\lambda_1, \lambda_2, \cdots, \lambda_N; \gamma_1, \gamma_2, \cdots, \gamma_N)  \mapsto  \mathcal{M}=  (\mathcal{M}_{kj})_{1\leq k, j  \leq N}  
\end{equation*}defined as $(\ref{general matrix M to prove sp in C negative})$ is real analytic on $\Omega_N \times \mathbb{R}^N $.

\end{lem}

\begin{proof}
The vector $V_{\lambda}\in \mathbb{R}^N$ is defined as $V_{\lambda}^T := ((2| \lambda_1|)^{-\frac{1}{2}}, (2 |\lambda_2|)^{-\frac{1}{2}}, \cdots, (2 |\lambda_N|)^{-\frac{1}{2}})$. So we have
\begin{equation*}
\mathrm{Im}\mathcal{M} = \left(-\tfrac{1}{2 \sqrt{|\lambda_j||\lambda_k|}} \right)_{1\leq k, j \leq N} = - V_{\lambda} \cdot V_{\lambda}^T .
\end{equation*}Recall that $\langle  X, Y \rangle_{\mathbb{C}^N}: = X^T \cdot \overline{Y}$, thus $\langle  (\mathrm{Im}\mathcal{M})X, X \rangle_{\mathbb{C}^N } = -| \langle X, V_{\lambda}\rangle_{\mathbb{C}^N }|^2 \leq 0$. So $\mathrm{Im}\mathcal{M}$ is a negative semi-definite matrix. If $\mu \in \sigma_{\mathrm{pp}}(\mathcal{M})$ and $V \in \mathrm{Ker}(\mu - \mathcal{M})\backslash \{0\}$, it suffices to show that $\mathrm{Im}\mu <0$. 
\begin{equation}\label{V perp V alpha}
-|\langle V, V_{\lambda} \rangle_{\mathbb{C}^N}|^2 = \langle (\mathrm{Im}\mathcal{M})V, V \rangle_{\mathbb{C}^N} = \mathrm{Im} \mu \|V\|_{\mathbb{C}^N}^2, \qquad \mathrm{where} \quad \|V\|_{\mathbb{C}^N}^2 = \langle V, V \rangle_{\mathbb{C}^N} >0.
\end{equation}So we have $\mathrm{Im} \mu \leq 0$. Assume that $\mu \in \mathbb{R}$, then formula $(\ref{V perp V alpha})$ yields that $V \perp V_{\lambda}$. Moreover, we have $(\mathcal{M}-\mathcal{M}^*) V = -2 i  \langle V, V_{\lambda} \rangle_{\mathbb{C}^N} V_{\lambda}=0$. We set $D^{\lambda} \in \mathbb{C}^{N\times N}$ to be the diagonal matrix whose diagonal elements are $\lambda_1, \lambda_2, \cdots ,\lambda_N$, i.e. $D^{\lambda} = \left(\begin{smallmatrix}
\lambda_1 \\
&  \lambda_2 \\
& &  \ddots \\
& & &  \lambda_N
\end{smallmatrix}\right)$. Then we have the following formula
\begin{equation}\label{def of diagonal matrix D alpha}
[\mathcal{M}, D^{\lambda}]  = i(\mathrm{I}_N + 2 D^{\lambda} V_{\lambda} V_{\lambda}^T).
\end{equation}So $[\mathcal{M}, D^{\lambda}]V = iV$ by $(\ref{def of diagonal matrix D alpha})$. Recall that $ \mathcal{M}^* V = \mathcal{M} V = \mu V$. Finally,
\begin{equation*}
i \|V\|_{\mathbb{C}^N}^2 = \langle [\mathcal{M}, D^{\lambda}] V, V \rangle_{\mathbb{C}^N} = \langle (\mathcal{M} - \mu) D^{\lambda}  V, V \rangle_{\mathbb{C}^N} = \langle D^{\lambda}  V, (\mathcal{M}^* - \mu) V \rangle_{\mathbb{C}^N}=0 
\end{equation*}contradicts the fact that $V \ne 0$. Consequently, we have $\mu \in \mathbb{C}_-$.
\end{proof}

\begin{cor}\label{negative definite of G restrict to Hpp}
For every $u \in \mathcal{U}_N$, let  $M(u)=(M_{kj}(u))_{1\leq k,j\leq N} \in \mathbb{C}^{N\times N}$ denote the matrix defined by formula $(\ref{coefficients of matrix of G def})$, then   $ \mathrm{Im}M(u)= \frac{M(u)-M(u)^*}{2i}$   is negative semi-definite and $\sigma_{\mathrm{pp}}(M(u)) \subset \mathbb{C}_-$.
\end{cor}

\begin{rem}
The fact $\sigma_{\mathrm{pp}}(M(u)) \subset \mathbb{C}_-$ can also be given by using the inversion formula $(\ref{inversion formula for Pi u and Q char polynomial })$ and proposition $\ref{Polynomial characterization of N soliton}$. The characteristic polynomial $Q_u(x)= \det(x-M(u))$ has zeros in $\mathbb{C}_-$.
\end{rem}

\bigskip

\subsection{Inverse spectral formulas}\label{subsection of inversion formula}
The injectivity of $\Phi_N$ is proved in this subsection by using inverse spectral formulas. The following lemma describes the relation between the Fourier transform of an eigenfunction $\varphi \in \mathscr{H}_{\mathrm{pp}}(L_u)$ and the inner function associated to $u$ defined by $\Theta_u =  \frac{\overline{Q}_u }{Q_u}$ with $Q_u(x)= \det(x-M(u))$.
  
\begin{lem}
For every monic polynomial $Q \in \mathbb{C}_N[X]$ such that $Q^{-1}(0)\subset \mathbb{C}_-$, the associated inner function is defined by $\Theta =\frac{\overline{Q}}{Q}$. The following identity holds for every $\varphi \in  \frac{\mathbb{C}_{\leq N-1}[X]}{Q}$,  
\begin{equation}\label{varphi xi in terms of 1-Theta}
\hat{\varphi}(\xi)=\langle S(\xi)^*\varphi, 1- \Theta\rangle_{L^2}.
\end{equation}In particular, $\hat{\varphi}(0^+)=\langle  \varphi, 1- \Theta\rangle_{L^2}$.
\end{lem}

\begin{proof}
Since $\varphi=\frac{P}{Q}$, for some $P \in \mathbb{C}_{\leq N-1}[X]$ and $Q^{-1}(0) \subset \mathbb{C}_-$, recall that $Q(x)= \prod_{j=1}^n  (x- z_j )^{m_j}$ with $\mathrm{Im}z_j <0$, $z_1, z_2, \cdots, z_N$ are all distinct and $\sum_{j=1}^n m_j =N$. Formulas $(\ref{fourier transform of  basis sol of differential eq})$ and $(\ref{definition of basis sol of differential eq})$ imply that 
\begin{equation*}
f_{j,l}(x)= \frac{l !}{2\pi [(-i)(x-z_j)]^{l+1}} \Longrightarrow \hat{f}_{j,l}(\xi)= \xi^l e^{-i z_j \xi} \mathbf{1}_{\mathbb{R}_+}(\xi).
\end{equation*}Since $\varphi\in  \mathrm{Span}_{\mathbb{C}}\{f_{j,l}\}_{1\leq j \leq m_j, 1\leq j  \leq n}$, partial-fractional decomposition implies that $\hat{\varphi} \in C^1(\mathbb{R}_+^*)$, and the right limit $\hat{\varphi}(0^+) = \lim_{\xi \to 0^+}\hat{\varphi}(\xi)$ exists. Recall that $\Theta = \frac{\overline{Q}}{Q}$, so we have $\overline{\Theta} \varphi = \frac{Q}{\overline{Q}} \frac{P}{Q} = \frac{P}{\overline{Q}} \in L^2_-$. Since $\Theta(x)= 1 + 2i \sum_{j=1}^N \frac{\mathrm{Im} z_j}{x-z_j}+ \mathcal{O}(\frac{1}{x^2})$, when $x\to +\infty$, we have $1-\Theta \in L^2_+$. Then 
\begin{equation*}
\hat{\varphi}(\xi)  = \int_{\mathbb{R}}\varphi(y)(1-\overline{\Theta (y)})e^{-iy \xi} \mathrm{d}y = \langle \varphi, S(\xi)(1-\Theta) \rangle_{L^2} = \langle S(\xi)^* \varphi,  1-\Theta  \rangle_{L^2} , \qquad \forall \xi \geq 0.
\end{equation*}
\end{proof}

\begin{prop}\label{prop inverse formula for general f in K theta}
For every $u \in \mathcal{U}_N$, we set $Q_u \in \mathbb{C}_N [X]$ to be the characteristic polynomial  of $u$ and we define the associated inner function as $\Theta_u = \frac{\overline{Q}_u}{Q_u}$. Then the following inversion formula holds,
\begin{equation}\label{inversion formula}
f(z)= \frac{1}{2\pi i}\langle (G-z)^{-1}f, 1-\Theta_u \rangle_{L^2}, \qquad f \in \mathscr{H}_{\mathrm{pp}}(L_u), \qquad \forall z \in \mathbb{C}_+. 
\end{equation} 
\end{prop}

\begin{proof} 
If $f \in \mathscr{H}_{\mathrm{pp}}(L_u) = \frac{\mathbb{C}_{\leq N-1}[X]}{Q_u}$, then formula $(\ref{varphi xi in terms of 1-Theta})$ yields that
\begin{equation*}
\hat{f}(\xi)    = \langle S(\xi)^*f,  1-\Theta_u  \rangle_{L^2} = \langle e^{-i\xi G} f,  1-\Theta_u \rangle_{L^2}.
\end{equation*}Since $\mathrm{Im}G:= \frac{G-G^*}{2i}$ is a negative semi-definite operator on $ \mathscr{H}_{\mathrm{pp}}(L_u)$ by proposition $\ref{Coefficients of Matrix of G}$ and lemma $\ref{lemma to get negative definiteness of Im M}$, the operator $\mathrm{Re}(i(z-G))_{| \mathscr{H}_{\mathrm{pp}}(L_u)}=(\mathrm{Im}G-\mathrm{Im}z )_{| \mathscr{H}_{\mathrm{pp}}(L_u)} $ is negative definite, for every $z  \in \mathbb{C}_+$. So
\begin{equation*}
f(z)=   \frac{1}{2\pi}\int_{0}^{+\infty} \langle e^{i\xi (z-G)} f,  1-\Theta_u  \rangle_{L^2} \mathrm{d}\xi =\frac{1}{2\pi i}\langle (G-z)^{-1} f, 1-\Theta_u \rangle_{L^2}.
\end{equation*}
\end{proof}

\noindent Recall that $\langle \Pi u, \varphi_j^u\rangle_{L^2}= \sqrt{2\pi |\lambda_j^u|}$ and $\langle 1-\Theta, \varphi_j^u\rangle_{L^2}= \sqrt{ \frac{2\pi }{|\lambda_j^u|}}$, for every $j=1, 2, \cdots, N$, by $(\ref{Fourier modes of Lu phi = lambda phi})$ and $(\ref{definition of varphi j u a basis of H pp})$. Since $\Pi u \in \mathrm{Hol}(\{z \in \mathbb{C} : \mathrm{Im}z > -\epsilon\})$, for some $\epsilon >0$, we have the following inversion formula
\begin{equation}\label{inversion formula in terms of matrix M X Y}
\Pi u(x)= \tfrac{1}{2\pi i}\langle (G-x)^{-1} \Pi u, 1-\Theta \rangle_{L^2} = -i \langle (M(u)-x)^{-1} X(u), Y(u)\rangle_{\mathbb{C}^N}, \qquad \forall x \in \mathbb{R},
\end{equation}where the two vectors $X(u), Y(u) \in \mathbb{R}^N$ are defined as
\begin{equation}\label{definition of Xu and Yu}
X(u)^T = (\sqrt{|\lambda_1^u|}, \sqrt{|\lambda_2^u|}, \cdots, \sqrt{|\lambda_N^u|}), \qquad Y(u)^T = (\sqrt{|\lambda_1^u|^{-1}}, \sqrt{|\lambda_2^u|^{-1}}, \cdots, \sqrt{|\lambda_N^u|^{-1}}),
\end{equation}and $M(u)$ is the $N\times N$ matrix of the infinitesimal generator $G$ associated to the orthonormal basis $\{\varphi_1^u, \varphi_1^u,\cdots, \varphi_N^u\}$, defined in $(\ref{Coefficients of Matrix of G})$. A consequence of the inverse spectral formula $(\ref{inversion formula in terms of matrix M X Y})$ is the explicit formula of the BO flow with $N$-soliton solutions as described by formula $(\ref{explicit formula of soliton solutions})$.

\begin{cor}\label{injectivity of action-angle coordinates corollary }
The map $\Phi_N : \mathcal{U}_N \to \Omega_N \times \mathbb{R}^N$  is injective.
\end{cor}

\begin{proof}
If $\Phi_N(u)= \Phi_N(v)$ for some $u,v \in \mathcal{U}_N$, then $\lambda_j^u=\lambda_j^v$ and $\gamma_j(u)= \gamma_j(v)$, for every $j$. So 
\begin{equation*}
M(u)=M(v), \qquad X(u)=X(v), \qquad Y(u)=Y(v).
\end{equation*}Then the inversion formula $(\ref{inversion formula in terms of matrix M X Y})$ gives that $\Pi u = \Pi v$. Thus, $u=2\mathrm{Re}\Pi u=2\mathrm{Re}\Pi v=v$.
\end{proof}

\noindent At last we show the equivalence between the inversion formulas $(\ref{inversion formula for Pi u and Q char polynomial })$ and $(\ref{inversion formula in terms of matrix M X Y})$.

\begin{proof}[Revisiting formula $(\ref{inversion formula for Pi u and Q char polynomial })$] 

For every $k, j =1,2, \cdots, N$, let $K^u_{kj}(x)$ denote the ${(N-1)\times (N-1)}$ submatrix obtained by deleting the k th column and j th row of the matrix $M(u)-x $, for every $x \in \mathbb{R}$. So the inversion formula $(\ref{inversion formula in terms of matrix M X Y})$ and the Cramer's rule imply that
\begin{equation}\label{formula to prove inverse formula with R}
i \Pi u(x)= \sum_{1\leq k,j \leq N} \frac{(-1)^{k+j} \det(K^u_{kj}(x))}{\det(M(u)-x)} \sqrt{\frac{\lambda_k^u}{\lambda_j^u}} =   \frac{\sum_{j=1}^N\det(K^u_{j j}(x))+ R}{\det(M(u)-x)},  
\end{equation}where $R :=  \sum_{1\leq k \ne j \leq N} (-1)^{k+j} \det(K^u_{kj}(x))\sqrt{\frac{\lambda_k^u}{\lambda_j^u}}$. The coefficients of the matrix $M(u)-x$ satisfies that
\begin{equation*}
(M(u)-x)_{kj}=M_{kj}(u)= \tfrac{i}{\lambda_k^u- \lambda_j^u}\sqrt{\tfrac{\lambda_k^u}{\lambda_j^u}}, \qquad \mathrm{if} \quad 1\leq  j \ne k \leq N,
\end{equation*}by formula $(\ref{coefficients of matrix of G def})$. Using expansion by minors, we have
\begin{equation*}
i R =  \sum_{1\leq k,j \leq N} (-1)^{k+j}(\lambda_k^u- \lambda_j^u) (M(u)-x)_{kj}  \det(K^u_{kj}(x)) = (\sum_{k=1}^N \lambda_k^u - \sum_{j=1}^N \lambda_j^u)\det(M(u)-x) = 0.
\end{equation*}Finally, let $Q$ denote the characteristic polynomial of the operator $G|_{\mathscr{H}_{\mathrm{pp}(L_u)}}$, so  
\begin{equation*}
Q (x) = \det(x - G|_{\mathscr{H}_{\mathrm{pp}}(L_u)}) =  \det(x - M(u)), \qquad   Q'(x) = (-1)^N \sum_{j=1}^N\det(K^u_{j j}(x)).
\end{equation*} 
\end{proof}

 \bigskip

\subsection{Poisson brackets}\label{subsection poisson brackets}
In this subsection, the Poisson bracket defined in $(\ref{Poisson bracket in intro of GAA coordinates})$ is generalized in order to obtain the first two formulas of $(\ref{Poisson bracket all})$. It can be defined between a smooth function from $\mathcal{U}_N$ to an arbitrary Banach space and another  smooth function from $\mathcal{U}_N$ to $\mathbb{R}$. \\

\noindent The N-soliton subset $(\mathcal{U}_N, \omega)$ is a real analytic symplectic manifold of real dimension $2N$, where 
\begin{equation*}
\omega_u (h_1, h_2) = \frac{i}{2\pi} \int_{\mathbb{R}} \frac{\hat{h}_1(\xi) \overline{\hat{h}_2(\xi)}}{\xi} \mathrm{d}\xi,  \qquad \forall h_1, h_2 \in \mathcal{T}_u (U_N), \qquad \forall u \in \mathcal{U}_N.
\end{equation*}For every smooth function $f : \mathcal{U}_N \to \mathbb{R}$, its Hamiltonian vector field $X_f \in \mathfrak{X}(\mathcal{U}_N)$ is given by $(\ref{formula to obtain relation between XF and nabla F})$. Recall that $X_f(u)=\partial_x \nabla_u f(u)$ and $\mathrm{d}f (u)(h) = \omega (h, X_{f}(u))$, $\forall h \in  \mathcal{T}_u (\mathcal{U}_N)$. For any Banach space $\mathcal{E}$ and any smooth map $F: u \in \mathcal{U}_N \mapsto F(u) \in \mathcal{E}$, we define the Poisson bracket of $f$ and $F$ as follows
\begin{equation}\label{definition of the Poisson bracket of f and F}
\{f, F \} : u \in \mathcal{U}_N \mapsto \{f, F \} ( u) : = \mathrm{d} F(u)(X_{f}(u))   \in  \mathcal{T}_{F(u)}(\mathcal{E})=\mathcal{E}.
\end{equation}If $\mathcal{E}=\mathbb{R}$, then the definition in formula $(\ref{definition of the Poisson bracket of f and F})$ coincide with $(\ref{Poisson bracket in intro of GAA coordinates})$ and we recall it here,
\begin{equation}\label{relation between poisson bracket and symplectic form}
\{f, F \} ( u) =\mathrm{d} F(u)(X_{f}(u)) = \omega_u (X_{f}(u), X_{F}(u)).
\end{equation}For every $\lambda \in \mathbb{C}\backslash \sigma(-L_u)$, the generating function $\mathcal{H}_{\lambda}(u)= \langle (L_u + \lambda)^{-1} \Pi u, \Pi u\rangle_{L^2}$ is well defined. Since $\Pi u = \sum_{j=1}^N \langle \Pi u , \varphi_j^{u}\rangle_{L^2} \varphi_j^{u}$,  we have
\begin{equation}\label{generating function for UN}
\mathcal{H}_{\lambda}(u)=  \sum_{j=1}^N \frac{|\langle \Pi u , \varphi_j^{u}\rangle_{L^2}|^2}{\lambda+ \lambda_j^u} = -\sum_{j=1}^N \frac{2\pi \lambda_j^u}{\lambda+ \lambda_j^u}.
\end{equation}The analytical continuation allow to extend the generating function $\lambda \mapsto \mathcal{H}_{\lambda}(u)$ to the domain $ \mathbb{C}\backslash \sigma_{\mathrm{pp}}(-L_u)$, and it has simple poles at every $\lambda = -\lambda_j^u$. Proposition $\ref{Self adjoint ness of Lu}$ yields that $-\frac{C^2}{4} \|u\|_{L^2}^2 \leq \lambda_1^u <  \cdots< \lambda_N^u < 0$, where $C=\inf_{f \in H^1_+ \backslash \{0\}}\frac{\||\mathrm{D}|^{\frac{1}{4}}f\|_{L^2}}{\| f\|_{L^4}}$ denotes the Sobolev constant. So we introduce 
\begin{equation}\label{definition of Y product}
\mathcal{Y}=\{(\lambda, u)\in \mathbb{R}\times  \mathcal{U}_N : 4\lambda >C^2 \|u\|_{L^2}^2 \} = \mathcal{X}\bigcap \left( \mathbb{R}\times \mathcal{U}_N \right),
\end{equation}where $\mathcal{X}$ is given by definition $\ref{def of generating function}$. Then the subset $\mathcal{Y}$ is open in $\mathbb{R}\times\mathcal{U}_N$ and the map $\mathcal{H}: (\lambda, u)\in \mathcal{Y} \mapsto -\sum_{j=1}^N \frac{2\pi \lambda_j^u}{\lambda+ \lambda_j^u} \in \mathbb{R}$ is real analytic by proposition $\ref{prop eigenvalue is analytic on UN}$. Recall that the Fr\'echet derivative $(\ref{frechet derivative of H lambda})$ is given by
\begin{equation*}
\mathrm{d} \mathcal{H}_{\lambda }(u)(h) = \langle w_{\lambda}  , \Pi h \rangle_{L^2} + \overline{\langle w_{\lambda} , \Pi h  \rangle}_{L^2}  + \langle  T_h w_{\lambda}   , w_{\lambda}  \rangle_{L^2} = \langle h,  w_{\lambda}  + \overline{w}_{\lambda} + |w_{\lambda} |^2  \rangle_{L^2},  \quad \forall h \in \mathcal{T}_u (\mathcal{U}_N).
\end{equation*}where  $w_{\lambda} \in H^1_+$ is given by $w_{\lambda}\equiv w_{\lambda}(u) \equiv w_{\lambda}(x, u)= [(L_u +\lambda  )^{-1} \circ \Pi] u(x)$, for every $x \in \mathbb{R}$. Thus
\begin{equation}\label{exact formula of hamitonian vector fields of H lambda}
X_{\mathcal{H}_{\lambda}}(u)= \partial_x \nabla_u \mathcal{H}_{\lambda}(u)=\partial_x (|w_{\lambda}(u) |^2 +w_{\lambda} (u) + \overline{w}_{\lambda}(u)), \qquad \forall (\lambda, u)\in \mathcal{Y}.
\end{equation}by $(\ref{formula to obtain relation between XF and nabla F})$. The Lax map $L: u \in \mathcal{U}_N \mapsto L_u = \mathrm{D}- T_u \in \mathfrak{B}(H^1_+, L^2_+)$  is $\mathbb{R}$-affine, hence real analytic. The following proposition restates the Lax pair structure of the Hamiltonian equation associated to $\mathcal{H}_{\lambda}$. Even though the stability of $\mathcal{U}_N$ under the Hamiltonian flow of $\mathcal{H}_{\lambda}$ remains as an open problem, the Poisson bracket defined in $(\ref{definition of the Poisson bracket of f and F})$ provides an algebraic method to obtain the first two formulas of $(\ref{Poisson bracket all})$.

\begin{prop}\label{proposition of poisson  brackets for H lambda lambda gamma}
Given $(\lambda, u)\in \mathcal{Y}$ defined by $(\ref{definition of Y product})$,  we have $\{\mathcal{H}_{\lambda},L\} (u) = [B_{\lambda}^u, L_u]$ and 
\begin{equation}\label{poisson brackets for H lambda lambda gamma}
\{\mathcal{H}_{\lambda}, \lambda_j\} (u) =0, \qquad \{\mathcal{H}_{\lambda}, \gamma_j\} (u) =    \mathrm{Re}\langle [G, B^{\lambda}_{u }] \varphi_j^{u }, \varphi_j^{u } \rangle_{L^2}=  -  \frac{\lambda}{(\lambda + \lambda_j^{u} )^2}  ,   
\end{equation}for every $j =1,2,\cdots, N$, where $B_{\lambda}^u=i(T_{w_{\lambda}(u)}T_{\overline{w}_{\lambda}(u)} + T_{w_{\lambda}(u)}+ T_{\overline{w}_{\lambda}(u)})$.  
\end{prop}

\begin{proof}
Since $L: u \in L^2(\mathbb{R}, \mathbb{R}) \mapsto L_u = \mathrm{D}- T_u \in \mathfrak{B}(H^1_+ , L^2_+)$, for every $u \in L^2_+$, we have
\begin{equation*}
\mathrm{d}L(u)(h) = -T_h, \qquad \forall h \in L^2_+.
\end{equation*}If $(\lambda, u)\in \mathcal{Y}$, then the $\mathbb{C}$-linear transformation $L_u+\lambda \in \mathfrak{B}(H^1_+, L^2_+)$ is bijective. So formula $(\ref{exact formula of hamitonian vector fields of H lambda})$ yields that $\{\mathcal{H}_{\lambda}, L\}(u)= \mathrm{d}L(u) (X_{\mathcal{H}_{\lambda}}(u)) = - T_{\mathrm{D} (|w_{\lambda}(u)|^2 + w_{\lambda}(u) + \overline{w}_{\lambda}(u))}$. Then identity $(\ref{equivalence of lax identity for Hamiltonian equation associated to H lambda})$ yields the Lax equation for the Hamiltonian flow of the generating function  $\mathcal{H}_{\lambda}$, i.e.
\begin{equation}\label{algebraic form of Lax equation for   generating function}
\{\mathcal{H}_{\lambda},L\} (u) = [B_{\lambda}^u, L_u] \in \mathfrak{B}(H^1_+, L^2_+).
\end{equation}Consider the map $L \varphi_j : u \in \mathcal{U}_N \mapsto L_u \varphi_j^u = \lambda_j^u \varphi_j^u \in H^1_+$, for every $(\lambda, u)\in \mathcal{Y}$, we have 
\begin{equation*}
\{\mathcal{H}_{\lambda},L\} (u) \varphi_j^u+L_u \left(\{\mathcal{H}_{\lambda},\varphi_j\} (u)\right) =  \lambda_j^u\{\mathcal{H}_{\lambda},\varphi_j\} (u) +  \{\mathcal{H}_{\lambda},\lambda_j\} (u) \varphi_j^u 
\end{equation*}with $\{\mathcal{H}_{\lambda},\varphi_j\}(u) \in H^1_+$ and $\{\mathcal{H}_{\lambda},\lambda_j\} (u) \in \mathbb{R}$. Then $(\ref{algebraic form of Lax equation for   generating function})$ yields that
\begin{equation*}
(\lambda_j^u-L_u) \left(B_u^{\lambda} \varphi_j^u - \{\mathcal{H}_{\lambda},\varphi_j\} (u) \right) =  \{\mathcal{H}_{\lambda},\lambda_j\} (u) \varphi_j^u.
\end{equation*}Since $\varphi_j^u \in \mathrm{Ker}( \lambda_j^u-L_u)$ and $\|\varphi_j^u\|_{L^2}=1$ by the definition in $(\ref{definition of varphi j u a basis of H pp})$, we have
\begin{equation*}
\{\mathcal{H}_{\lambda},\lambda_j\} (u) = \langle (\lambda_j^u-L_u)  \left( B_u^{\lambda} \varphi_j^u - \{\mathcal{H}_{\lambda},\varphi_j\} (u) \right), \varphi_j^u \rangle_{L^2}=0.
\end{equation*}Let $\mathcal{N}_2 : \varphi \in L^2 \mapsto \|\varphi\|_{L^2}^2$, then we have $\mathcal{N}_2 \circ \varphi_j \equiv 1$ on $\mathcal{U}_N$. Then we have 
\begin{equation}\label{coef is imaginary}
0=\mathrm{d}(\mathcal{N}_2 \circ \varphi_j)(u)=2 \mathrm{Re}\langle \varphi_j^u,\{\mathcal{H}_{\lambda},\lambda_j\} (u) \rangle_{L^2}.
\end{equation}So there exists $r \in \mathbb{R}$ such that $B_u^{\lambda} \varphi_j^u - \{\mathcal{H}_{\lambda},\varphi_j\} (u) = ir \varphi_j^u$ because $\mathrm{Ker}( \lambda_j^u-L_u) = \mathbb{C}\varphi_j^u$ by corollary $\ref{Wu's result on simplicity of eigenvalues}$ and formula $(\ref{coef is imaginary})$. Recall that $B_u^{\lambda}$ is skew-adjoint and  $\gamma_j = \mathrm{Re}\langle G\varphi_j^u, \varphi_j^u \rangle_{L^2}$, we have
\begin{equation*}
\{\mathcal{H}_{\lambda}, \gamma_j\} (u) = \mathrm{Re}\left(\langle  G \{\mathcal{H}_{\lambda},\varphi_j\} (u),    \varphi_j^{u } \rangle_{L^2}+ \langle  G \varphi_j^{u } , \{\mathcal{H}_{\lambda},\varphi_j\} (u)\rangle_{L^2}\right) =    \mathrm{Re}\langle [G, B^{\lambda}_{u }] \varphi_j^{u }, \varphi_j^{u } \rangle_{L^2}.
\end{equation*}Furthermore, for every $(\lambda,u) \in \mathcal{Y}$, formula $(\ref{formula of commutator of G T_b})$ implies that $[G, T_{\overline{w}_{\lambda}(u)}] =0$ and
\begin{equation}\label{commutator of G B lambda u f}
[G, B^{\lambda}_{u}]f = i[G, T_{w_{\lambda}(u)}](T_{\overline{w}_{\lambda}(u)}(f) + f) = -\tfrac{1}{2 \pi}     [(\overline{w}_{\lambda}(u) f)^{\wedge} (0^+) + \hat{f}(0^+) ]w_{\lambda}(u),   \quad \forall f \in \mathbf{D}(G).
\end{equation}Since $(\overline{w}_{\lambda}(u) \varphi_j^u)^{\wedge} (0^+) = \langle \varphi_j^u, w_{\lambda}(u) \rangle_{L^2}  = (\lambda + \lambda_j^u)^{-1} \overline{\langle u, \varphi_j^u \rangle}_{L^2}$ and $\overline{\langle u, \varphi_j^u \rangle}_{L^2}=-\lambda_j^u \widehat{\varphi_j^u} (0^+)$, we replace $f$ by $\varphi_j^u$ in formula $(\ref{commutator of G B lambda u f})$ to obtain the following
\begin{equation*}
 \langle [G,  B^{\lambda}_{u}]\varphi_j^{u} , \varphi_j^{u } \rangle_{L^2} = \frac{\overline{\langle u, \varphi_j^u \rangle}_{L^2}}{2\pi}(\frac{1}{\lambda_j^u} - \frac{1}{\lambda+\lambda_j^u}) \langle w_{\lambda}(u) , \varphi_j^u \rangle_{L^2} = - \frac{\lambda}{(\lambda+\lambda_j^u)^2}, \quad \forall (\lambda,u) \in \mathcal{Y}.
\end{equation*} 
\end{proof}

\begin{rem}\label{rem Bo hierarchy poisson bracket}
Recall that $\tilde{\mathcal{H}}_{\epsilon}=\frac{1}{\epsilon}\mathcal{H}_{\frac{1}{\epsilon}} $ and $\tilde{B}_{\epsilon,u} :=\frac{1}{\epsilon}B_u^{\frac{1}{\epsilon}}$ for every $(\epsilon^{-1}, u)\in \mathcal{Y}$. In general, the identity
\begin{equation*}
\{E_n, \gamma_j\}(u)= \mathrm{Re}\langle [G, \tfrac{\mathrm{d}^n}{\mathrm{d}\epsilon^n}\big|_{\epsilon=0}\tilde{B}_{\epsilon, u}] \varphi_j^u,\varphi_j^u\rangle_{L^2}, \qquad 1\leq j\leq N
\end{equation*}holds for every conservation law $E_n = (-1)^n\frac{\mathrm{d}^n}{\mathrm{d}\epsilon^n}\big|_{\epsilon=0}\tilde{\mathcal{H}}_{\epsilon}$ in the BO hierarchy.
\end{rem}

\begin{cor}\label{corollary first two formulas of poisson bracket}
For every $j, k = 1,2, \cdots, N$, we have
\begin{equation}\label{Poisson brackets of gamma lambda}
2\pi\{\lambda_j, \gamma_k\}(u) = \mathbf{1}_{j=k}, \qquad \{\lambda_k, \lambda_j\}(u)=0, \qquad  \forall u \in \mathcal{U}_N.
\end{equation} 
\end{cor}

\begin{proof}
Given $u \in \mathcal{U}_N$, for every $\lambda > \frac{C^2 \|u\|_{L^2}^2}{4}$ then $(\lambda,u) \in \mathcal{Y}$, then $(\ref{generating function for UN})$ and $(\ref{poisson brackets for H lambda lambda gamma})$ imply that
\begin{equation*}
-\frac{\lambda}{(\lambda+\lambda_j^u)^2}= \{\mathcal{H}_{\lambda}, \gamma_j\}(u)= 2\pi \sum_{k=1}^N \{\frac{\lambda}{\lambda+\lambda_k^u}, \gamma_j\}(u) = -2\pi \lambda \sum_{k=1}^N \frac{\{\lambda_k, \gamma_j\}(u)}{(\lambda+\lambda_k^u)^2},
\end{equation*}and $0 = \{\mathcal{H}_{\lambda}, \lambda_j\}(u) = 2\pi \lambda \sum_{k=1}^N \frac{\{\lambda_k, \lambda_j\}(u)}{(\lambda + \lambda_k^u)^2}$, for every $j=1,2,\cdots,N$. The uniqueness of analytic continuation yields that the following formula holds for every $z \in \mathbb{C}\backslash \mathbb{R}$,
\begin{equation*}
-\frac{z}{(z+\lambda_j^u)^2} = -2\pi z  \sum_{k=1}^N \frac{\{\lambda_k, \gamma_j\}(u)}{(z +\lambda_k^u)^2}, \qquad \sum_{k=1}^N \frac{\{\lambda_k, \lambda_j\}(u)}{(z + \lambda_k^u)^2}=0 .
\end{equation*}
\end{proof}

\noindent Recall that the  actions $I_j : u\in \mathcal{U}_N \mapsto 2\pi \lambda_j^u$  and the generalized angles $\gamma_j: u\in \mathcal{U}_N \mapsto \mathrm{Re}\langle G\varphi_j^u, \varphi_j^u\rangle_{L^2}$ are both real analytic functions by proposition $\ref{prop eigenvalue is analytic on UN}$ and corollary $\ref{G varphi varphi analytic}$.

\begin{prop}\label{prop of linear independance}
For every $u \in \mathcal{U}_N$, the family of differentials 
\begin{equation*}
\{\mathrm{d}I_1(u),\mathrm{d}I_2(u), \cdots\mathrm{d}I_N(u);\mathrm{d}\gamma_1(u),\mathrm{d}\gamma_2(u), \cdots\mathrm{d}\gamma_N(u) \}
\end{equation*}is linearly independent in the cotangent space $\mathcal{T}^*_u(\mathcal{U}_N)$.
\end{prop}

\begin{proof}
For every $a_1, a_2, \cdots, a_N, b_1, b_2, \cdots, b_N \in \mathbb{R}$ such that
\begin{equation}\label{formula for proving linear independance of differentials}
\left(\sum_{j=1}^N a_j\mathrm{d}I_j(u)+ b_j\mathrm{d}\gamma_j(u) \right)(h)=0, \qquad \forall  h \in \mathcal{T}_u (\mathcal{U}_N). 
\end{equation}Formula of Poisson brackets $(\ref{Poisson brackets of gamma lambda})$ yields that for every $j,k=1,2,\cdots, N$, we have
\begin{equation*}
\mathrm{d}I_j(u)(X_{I_k}(u))=\{I_k, I_j\}(u)=0, \qquad  \mathrm{d}\gamma_j(u)(X_{I_k}(u))=\{I_k, \gamma_j\}(u)=\mathbf{1}_{j=k}
\end{equation*}We replace $h$ by $X_{I_k}(u)$ in $(\ref{formula for proving linear independance of differentials})$ to obtain that $b_k=0$, $\forall k=1,2,\cdots, N$. Then set $h=X_{\gamma_k}(u)$
\begin{equation*}
- a_k=\sum_{j=1}^N a_j \{\gamma_k, I_j\}(u)=\left(\sum_{j=1}^N a_j\mathrm{d}I_j(u)  \right)(X_{\gamma_k}(u))=0, \qquad \forall k=1,2,\cdots, N.
\end{equation*} 
\end{proof}

\noindent As a consequence, $\Phi_N : \mathcal{U}_N \to \Omega_N \times \mathbb{R}^N$ is a local diffeomorphism. Moreover, since all the actions $(I_j)_{1\leq j\leq N}$ are in evolution by $(\ref{Poisson brackets of gamma lambda})$ and the differentials $(\mathrm{d}I_j(u))_{1\leq j \leq N}$ are linearly independent for every $u \in \mathcal{U}_N$, for every $r=(r^1, r^2, \cdots,r^N) \in \Omega_N$, the level set 
\begin{equation*}
\mathcal{L}_r = \bigcap_{j=1}^N I_j^{-1}(r^j), \qquad \mathrm{where} \quad r=(r^1, r^2, \cdots, r^N)  
\end{equation*}is a smooth Lagrangian submanifold of $U_N$ and $\mathcal{L}_r$ is invariant under the Hamiltonian flow of $I_j$, for every $j=1,2, \cdots, N$, by the Liouville--Arnold theorem (see Theorem 5.5.21 of Katok--Hasselblatt $[\ref{Katok Hasselblatt modern theory to dynamical system}]$, see also Fiorani--Giachetta--Sardanashvily $[\ref{Fiorani Giachetta Sardanashvily Liouville Arnold Nekhoroshev}]$ and Fiorani--Sardanashvily $[\ref{Fiorani Sardanashvily Global action-angle coordinates}]$ for the non-compact invariant manifold case).

\bigskip

\subsection{The diffeomorphism property}\label{subsection diffeo property}
This subsection is dedicated to proving the real bi-analyticity of $\Phi_N : \mathcal{U}_N \to \Omega_N \times \mathbb{R}^N$. It remains to show the surjectivity. Its proof is based on Hadamard's global inverse theorem $\ref{Hadamard global inversion thm}$.

\begin{lem}\label{Phi is proper}
The map $\Phi : \mathcal{U}_N \to \Omega_N \times \mathbb{R}^N$ is proper.
\end{lem}
\begin{proof}
If $K$ is compact in $\Omega_N \times \mathbb{R}^N$, we choose $u_n \in \Phi_N^{-1}(K)$, so
\begin{equation*}
\Phi_N(u_n)=(2\pi \lambda_1^{u_n}, 2\pi \lambda_2^{u_n}, \cdots, 2\pi \lambda_N^{u_n}; \gamma_1(u_n), \gamma_2(u_n), \cdots, \gamma_N(u_n)) \in K,  \qquad \forall n \in \mathbb{N}.
\end{equation*}We assume that there exists $(2\pi\lambda_1, 2\pi\lambda_2, \cdots,2\pi\lambda_N; \gamma_1, \gamma_2, \cdots, \gamma_N) \in K$ such that $\lambda_j^{u_n} \to \lambda_j$ and $\gamma_j(u_n)\to \gamma_j$ up to a subsequence. So $(M(u_n))_{n\in \mathbb{N}}$ converges to some matrix $M \in \mathbb{C}^{N\times N}$ whose coefficients are defined as follows
\begin{equation*}
M_{kj}=
\begin{cases}\frac{i}{\lambda_k- \lambda_j} \sqrt{\frac{|\lambda_k|}{|\lambda_j|}},\qquad \mathrm{if} \quad k \ne j,\\
\gamma_j - \frac{i}{2|\lambda_j|}, \qquad \quad \mathrm{if} \quad k=j.
\end{cases}
\end{equation*}Lemma $\ref{lemma to get negative definiteness of Im M}$ yields that $\sigma_{\mathrm{pp}}(M) \subset \mathbb{C}_-$. We set $Q(x):=\det(x-M)$ and $u=i\frac{Q'}{Q} - i \frac{\overline{Q}'}{\overline{Q}} \in \mathcal{U}_N$. The Vi\`ete map $\mathbf{V}$ is defined in $(\ref{definition of Vieta map})$ and $\mathbf{V}(\mathbb{C}_-^N)$ is open in $\mathbb{C}^N$. Then there exists
\begin{equation*}
\mathbf{a}^{(n)}=(a^{(n)}_0, a^{(n)}_1, \cdots, a^{(n)}_{N-1}), \quad  \mathbf{a}=(a_0, a_1, \cdots, a_{N-1}) \in \mathbf{V}(\mathbb{C}_-^N)
\end{equation*}such that $Q_n(x)=\det(x-M(u_n)) = \sum_{j=0}^{N-1} a^{(n)}_j x^j+ x^N$ and $Q(x)=\sum_{j=0}^{N-1} a_j x^j+ x^N$. We have
\begin{equation*}
\lim_{n\to +\infty} Q_n(x) = Q(x), \quad \forall x \in \mathbb{R} \quad \Longrightarrow \quad\lim_{n \to +\infty}\mathbf{a}^{(n)} = \mathbf{a}
\end{equation*}The continuity of the map $\Gamma_N : \mathbf{a}= (a_0, a_1, \cdots, a_{N-1}) \in \mathbf{V}(\mathbb{C}_-^N) \mapsto \Pi u = i\frac{Q'}{Q} \in   L^2_+$ yields that 
\begin{equation*}
\Pi u_n= i\frac{Q_n'}{Q_n} = \Gamma_N(\mathbf{a}^{(n)}) \to \Gamma_N(\mathbf{a} ) = i \frac{Q'}{Q} = \Pi u \qquad \mathrm{in} \quad L^2_+, \qquad \mathrm{as} \quad n \to +\infty.
\end{equation*}Since $\mathcal{U}_N$ inherits the subspace topology of $L^2(\mathbb{R},\mathbb{R})$, we have $(u_n)_{n\in \mathbb{N}}$ converges to $u$ in $\mathcal{U}_N$. The continuity of the map $\Phi_N$ shows that $\Phi_N(u) =(2\pi\lambda_1, 2\pi\lambda_2, \cdots,2\pi\lambda_N; \gamma_1, \gamma_2, \cdots, \gamma_N) \in K$.

\end{proof}

\begin{prop}\label{prop of real bi analytic diffeo}
The map $\Phi_N : \mathcal{U}_N \to \Omega_N \times\mathbb{R}^N$ is bijective and both $\Phi_N$ and its inverse $\Phi_N^{-1}$ are real analytic.
\end{prop} 

\begin{proof}
The analyticity of $\Phi_N$ is given by proposition $\ref{prop eigenvalue is analytic on UN}$ and corollary $\ref{G varphi varphi analytic}$. The injectivity is given by corollary $\ref{injectivity of action-angle coordinates corollary }$. Proposition $\ref{prop of linear independance}$ yields that $\Phi_N : \mathcal{U}_N \to \Omega_N \times \mathbb{R}^N$ is a local diffeomorphism by inverse function theorem for manifolds. So $\Phi_N$ is an open map. Since every proper continuous map to locally compact space is closed, $\Phi_N$ is also a closed map by lemma $\ref{Phi is proper}$. Since the target space $\Omega_N \times \mathbb{R}^N$ is connected, we have $\Phi_N(\mathcal{U}_N) = \Omega_N \times \mathbb{R}^N$ and $\Phi_N : \mathcal{U}_N \to \Omega_N \times \mathbb{R}^N$ is a real analytic diffeomorphism.

\end{proof}

\begin{rem}
We establish the relation between $\Phi_N : \mathcal{U}_N \to \Omega_N \times  \mathbb{R}^N$ and $\Gamma_N : \mathbf{V}(\mathbb{C}_-^N) \to \Pi (\mathcal{U}_N)$ introduced in proposition $\ref{Lemma of complex analytic manifold Pi UN}$. We set $\mathcal{M} :  \Omega_N \times \mathbb{R}^N \to \mathbb{C}^{N\times N}$ to be the matrix-valued real analytic function $\mathcal{M} (\eta_1, \eta_2, \cdots, \eta_N; \theta_1, \theta_2, \cdots, \theta_N) =(\mathcal{M}_{kj})_{1\leq k,j\leq N} $ with coefficients defined as 
\begin{equation*}
\mathcal{M}_{kj} =
\begin{cases}\frac{2\pi i}{\eta_k- \eta_j} \sqrt{\frac{ \eta_k }{ \eta_j}},   \qquad   \mathrm{if} \quad k \ne j,\\
\theta_j+ \frac{\pi i}{ \eta_j }, \qquad \quad \mathrm{if} \quad k=j.
\end{cases}
\end{equation*}Then, we set $\mathcal{C} : M \in \mathbb{C}^{N \times N} \mapsto (a_0, a_1, \cdots, a_{N-1}) \in \mathbb{C}^{N}$ such that
\begin{equation}\label{Qchar polynomial of M}
Q(x):=\sum_{j=0}^{N-1} a_j x^j +x^N=\det(x-M).  
\end{equation}Since $(-1)^{n-j} a_j  =  \mathrm{Tr}(\boldsymbol{\Lambda}^{n-j} M)$ is the sum of all principle minors of $M$ of size $(N-j)\times (N-j)$, for every $j=1,2, \cdots, N$, the map $\mathcal{C}$ is real analytic on $ \mathbb{C}^{N \times N}$ and $\mathcal{C} \circ \mathcal{M}(\Omega_N \times \mathbb{R}^N) \subset \mathbf{V}(\mathbb{C}^N_-)$ by lemma $\ref{lemma to get negative definiteness of Im M}$, where $\mathbf{V}$ denotes the  Vi\`ete map defined as $(\ref{definition of Vieta map})$. In lemma $\ref{Lemma of complex analytic manifold Pi UN}$, we have shown that the map $\Gamma_N : \mathbf{a}= (a_0, a_1, \cdots, a_{N-1}) \in \mathbf{V}(\mathbb{C}_-^N) \mapsto \Pi u = i\frac{Q'}{Q} \in \Pi(\mathcal{U}_N)$ is biholomorphic, where the polynomial $Q$ is defined as $(\ref{Qchar polynomial of M})$. We conclude by the following identity
\begin{equation}\label{relation between Phi N and Gamma N}
\Phi_N^{-1} = 2\mathrm{Re} \circ \Gamma_N \circ \mathcal{C} \circ \mathcal{M} %% \Longleftrightarrow \Gamma_N^{-1}=  \mathcal{C} \circ \mathcal{M} \circ\Phi_N \circ (2\mathrm{Re}).
\end{equation}  
 \end{rem}

\noindent The smooth manifolds $\Pi(\mathcal{U}_N)$ and $\mathbf{V}(\mathbb{C}^N_-)$ are both diffeomorphic to the convex open subset $\Omega_N \times \mathbb{R}^N$, so they are simply connected (see also proposition $\ref{simply connected property for viete map}$). At last, we recall Hadamard's global inverse theorem.  
\begin{thm}\label{Hadamard global inversion thm}
Suppose $X$ and $Y$ are connected smooth manifolds, then every proper local diffeomorphism $F : X \to Y$ is surjective. If $Y$ is simply connected in addition, then every proper local diffeomorphism $F : X \to Y$ is a diffeomorphism.
\end{thm}

\begin{proof}
For the surjectivity, see Nijenhuis--Richardson $[\ref{Nijenhuis Richardson A theorem on maps with non-negative Jacobians}]$ and the proof of proposition $\ref{prop of real bi analytic diffeo}$. If the target space is simply connected, see Gordon $[\ref{Gordon On the diffeomorphisms of Euclidean space}]$ for the injectivity. 
\end{proof}

\begin{rem}
Since the target space $\Omega_N \times \mathbb{R}^N$ is convex, there is another way to show the injectivity of $\Phi_N$ without using the inversion formulas in subsection $\mathbf{\ref{subsection of inversion formula}}$. It suffices to use the simple connectedness of $\Omega_N \times \mathbb{R}^N$ and Hadamard's global inverse  theorem $\ref{Hadamard global inversion thm}$. 
\end{rem}

\bigskip

\subsection{A Lagrangian submanifold}\label{subsection lagrangian section}
In general, the symplectomorphism property of $\Phi_N$ is equivalent to its Poisson bracket characterization $(\ref{Poisson bracket all})$, which will be proved in proposition $\ref{Property symplectomorphism property Poisson bracket}$. The first two formulas of $(\ref{Poisson bracket all})$  given in corollary $\ref{corollary first two formulas of poisson bracket}$, lead us to focusing on the study of a special Lagrangian submanifold of $\mathcal{U}_N$, denoted by
\begin{equation}\label{Lambda N definition}
\Lambda_N := \{u \in \mathcal{U}_N : \gamma_j(u)=0, \quad \forall j =1,2, \cdots, N\},
\end{equation}where the generalized angles $\gamma_j  :u \in \mathcal{U}_N \mapsto  \mathrm{Re}\langle G \varphi_j^u, \varphi_j^u\rangle_{L^2}$ are defined in $(\ref{j th action and angle})$. A characterization lemma of $\Lambda_N$ is given at first. 

\begin{lem}\label{lem of equivalence def of lagrangian Lambda N}
For every $u \in \mathcal{U}_N$, then each of the following four properties implies the others: \\

\noindent $\mathrm{(a)}$. $u \in \Lambda_N$.\\
\noindent $\mathrm{(b)}$. For every $x \in \mathbb{R}$, we have $\overline{\Pi u}(x)  = \Pi u(-x)$. \\
\noindent $\mathrm{(c)}$. $u$ is an even function $\mathbb{R} \to \mathbb{R}$. \\
\noindent $\mathrm{(d)}$. The Fourier transform $\hat{u}$ is real-valued. 
\end{lem}
\noindent Then every element $u\in \Lambda_N$ has translation--scaling parameter in $(i\mathbb{R})^N\slash \mathrm{S}_N$ i.e. $u(x)= \sum_{j=1}^N \frac{2 \eta_j}{x^2 + \eta_j^2}$, for some $\eta_j >0$.

\begin{proof}
$\mathrm{(a)} \Rightarrow \mathrm{(b)}$: If $u \in \Lambda_N$, then the matrix $M(u)$ defined in $(\ref{coefficients of matrix of G def})$ is an $N \times N$ matrix with purely imaginary coefficients. Recall the definition of $X(u), Y(u) \in \mathbb{R}^N$ in $(\ref{definition of Xu and Yu})$:
\begin{equation*}
X(u)^T = (\sqrt{|\lambda_1^u|}, \sqrt{|\lambda_2^u|}, \cdots, \sqrt{|\lambda_N^u|}), \qquad Y(u)^T = (\sqrt{|\lambda_1^u|^{-1}}, \sqrt{|\lambda_2^u|^{-1}}, \cdots, \sqrt{|\lambda_N^u|^{-1}}).
\end{equation*}The inversion formula $(\ref{inversion formula in terms of matrix M X Y})$ yields that
\begin{equation*}
\overline{\Pi u}(x)  = i \langle (\overline{M}(u)-x)^{-1} X(u), Y(u) \rangle_{\mathbb{C}^N} = - i \langle (M(u)+x)^{-1} X(u), Y(u) \rangle_{\mathbb{C}^N} = \Pi u(-x).
\end{equation*}$\mathrm{(b)} \Rightarrow \mathrm{(c)}$ is given by the formula $u = \Pi u + \overline{\Pi u}$. $\mathrm{(c)} \Rightarrow \mathrm{(d)}$ is given by $\overline{u}(x)=u(x)=u(-x)$. \\

\noindent $\mathrm{(d)} \Rightarrow \mathrm{(a)}$: Choose $\lambda \in \sigma_{\mathrm{pp}}(L_u) = \{\lambda_1^u,\lambda_2^u,\cdots, \lambda_N^u \}$ and $\varphi \in \mathrm{Ker}(\lambda-L_u)$. Since both $u$ and its Fourier transform $\hat{u}$ are real-valued,  we have  $[(\overline{\varphi})^{\vee}]^{\wedge}(\xi)= \overline{\hat{\varphi}(\xi)}$, where $(\overline{\varphi})^{\vee}(x) :=   \overline{\varphi(-x)}$,  $\forall x,\xi \in \mathbb{R}$. Thus,
\begin{equation*}
T_u((\overline{\varphi})^{\vee}) =( \overline{T_u \varphi})^{\vee} \Longrightarrow(\overline{\varphi})^{\vee}\in \mathrm{Ker}(\lambda-L_u)   .
\end{equation*} 

\noindent We choose the orthonormal basis $\{\varphi_1^u, \varphi_2^u,, \cdots, \varphi_N^u\}$ in $\mathscr{H}_{\mathrm{pp}}(L_u)$ as in formula $(\ref{definition of varphi j u a basis of H pp})$. Proposition $\ref{Wu's result on simplicity of eigenvalues}$ yields that $\dim_{\mathbb{C}}\mathrm{Ker}(\lambda-L_u)=1$. For every $j =1,2, \cdots, N$, there exists $\tilde{\theta}_j \in \mathbb{R}$ such that 
\begin{equation*}
(\overline{\varphi_j^u})^{\vee} = e^{i\tilde{\theta}_j}\varphi_j^u \Longleftrightarrow \overline{(\varphi_j^u)^{\wedge}}(\xi) =  e^{i\tilde{\theta}_j}(\varphi_j^u)^{\wedge}(\xi), \quad \forall \xi \in \mathbb{R}.
\end{equation*}So we set $\phi_j^u:=\exp({\tfrac{i\tilde{\theta}_j}{2}})\varphi_j^u$, then its Fourier transform $(\phi_j^u)^{\wedge}$ is a real-valued function. Recall the definition of $G$ in $(\ref{definition of G})$ and $\gamma_j$ in $(\ref{coefficients of matrix of G def})$, then we have
\begin{equation*}
\gamma_j(u)= \mathrm{Re}\langle G\varphi_j^u, \varphi_j^u\rangle_{L^2(\mathbb{R})}= \mathrm{Re}\langle G\phi_j^u, \phi_j^u\rangle_{L^2(\mathbb{R})} = -\frac{1}{2\pi} \mathrm{Im}\langle \partial_{\xi}[(\phi_j^u)^{\wedge}] , (\phi_j^u)^{\wedge}\rangle_{L^2(0, +\infty)}=0.
\end{equation*}by using Plancherel formula.

\end{proof}

\begin{lem}\label{Lem of Lambda is a regular level set}
The level set $\Lambda_N$ is a real analytic Lagrangian submanifold of $(\mathcal{U}_N, \omega)$. 
\end{lem}

\begin{proof}
The map $\mathbf{\gamma} : u \in \mathcal{U}_N \mapsto (\gamma_1(u),  \gamma_2(u), \cdots, \gamma_N(u)) \in \mathbb{R}^N$ is a real analytic submersion by proposition $\ref{prop of linear independance}$. So the level set $\Lambda_N$ is a properly embedded real analytic submanifold of $\mathcal{U}_N$  and $\dim_{\mathbb{R}} \Lambda_N =N$. The classification of the tangent space $\mathcal{T}_u(\mathcal{U}_N)$ is given by formula $(\ref{tangent space of u in UN})$. If $u(x)= \sum_{j=1}^N \frac{2\eta_j}{ x^2 + \eta_j^2 }$, for some $\eta_j >0$, every tangent vector $h \in \Lambda_N$ is an even function by lemma $\ref{lem of equivalence def of lagrangian Lambda N}$. So $\hat{h}$ is real valued and we have
\begin{equation}\label{tangent space of u in Lambda N}
\mathcal{T}_u (\Lambda_N) = \bigoplus_{j=1}^N  \mathbb{R}f_j^u, \qquad \mathrm{where}\quad f_{j}^u(x)=  \tfrac{2[x^2 - \eta_j^2]}{[x^2 + \eta_j^2]^2}.
\end{equation}We have $(f^u_j)^{\wedge}( \xi)=-2\pi |\xi| e^{-\eta_j |\xi|}$. Then by definition of $\omega$, we have
\begin{equation}
\omega_u (h_1, h_2) = \frac{i}{2\pi} \int_{\mathbb{R}} \frac{\hat{h}_1(\xi) \overline{\hat{h}_2(\xi)}}{\xi} \mathrm{d}\xi   = \frac{i}{2\pi} \int_{\mathbb{R}} \frac{\hat{h}_1(\xi)  \hat{h}_2(\xi)}{\xi} \mathrm{d}\xi \in i\mathbb{R}, \qquad \forall h_1, h_2 \in \mathcal{T}_u (\Lambda_N).
\end{equation}Since the symplectic form $\omega$ is real-valued, we have $\omega_u (h_1, h_2)= 0$, for every $h_1, h_2 \in \mathcal{T}_u (\Lambda_N)$. Since $\dim_{\mathbb{R}}(\Lambda_N)=N = \frac{1}{2} \dim_{\mathbb{R}}\mathcal{U}_N$, $\Lambda_N$ is a Lagrangian submanifold of $\mathcal{U}_N$.
\end{proof}

\bigskip

\subsection{The symplectomorphism property}\label{subsection symplectomorphism}
Finally, we prove the assertion $(\mathrm{b})$ in theorem $\ref{action angle coordinate thm}$, i.e.  the map $\Phi_N : (\mathcal{U}_N, \omega) \to (\Omega_N \times \mathbb{R}^N, \nu)$ is symplectic, where $\omega  (h_1, h_2) := \frac{i}{2\pi} \int_{\mathbb{R}} \frac{\hat{h}_1(\xi) \overline{\hat{h}_2(\xi)}}{\xi}\mathrm{d}\xi$, for every $h_1, h_2 \in \mathcal{T}_u(\mathcal{U}_N)$ and 
\begin{equation*}
\Omega_N \times \mathbb{R}^N = \{(r^1, r^2, \cdots, r^N; \alpha^1,  \alpha^2, \cdots, \alpha^N) \in \mathbb{R}^{2N} : r^1< r^2< \cdots < r^N<0\}, \quad \nu = \sum_{j=1}^N \mathrm{d}r^j \wedge \mathrm{d}\alpha^j.
\end{equation*}We set $\Psi_N  = \Phi_N^{-1} : \Omega_N \times \mathbb{R}^N \to \mathcal{U}_N$, let $\Psi_N^* \omega$ denote  the pullback of the symplectic form $\omega$ by $\Psi_N$, i.e. for every $p=(r^1, r^2, \cdots, r^N; \alpha^1,  \alpha^2, \cdots, \alpha^N) \in \Omega_N \times \mathbb{R}^N$, set $u= \Psi_N(p) \in \mathcal{U}_N$,
\begin{equation}\label{def of pullback of omega by Psi N}
\left(\Psi_N^* \omega\right)_{p}(V_1, V_2)= \omega_u(\mathrm{d}\Psi_N(p)(V_1) ,\mathrm{d}\Psi_N(p)(V_2)),
\end{equation}for every $V_1, V_2 \in \mathcal{T}_{p}(\Omega_N \times \mathbb{R}^N).$ The goal is to prove that 
\begin{equation}\label{Goal nu=0}
 \tilde{\nu} := \Psi_N^* \omega - \nu =0.
\end{equation}

\noindent Recall that the coordinate vectors $ \frac{\partial}{\partial r^1}\big|_p,\frac{\partial}{\partial r^2}\big|_p,\cdots, \frac{\partial}{\partial r^N}\big|_p; \frac{\partial}{\partial \alpha^1}\big|_p,\frac{\partial}{\partial \alpha^2}\big|_p,\cdots, \frac{\partial}{\partial \alpha^N}\big|_p $ form a basis for the tangent space $\mathcal{T}_p(\Omega_N \times \mathbb{R}^N)$. We have the following lemma.

\begin{lem}\label{lemma d Phi N send X I k to coordinate vector gamma k}
For every $u \in \mathcal{U}_N$, set $p=\Phi_N(u) \in\Omega_N \times \mathbb{R}^N$. Then we have
\begin{equation}\label{d Phi N send X I k to coordinate vector gamma k}
\mathrm{d}\Phi_N (u) (X_{I_k}(u)) = \frac{\partial}{\partial \alpha^k}\Big|_p, \qquad \forall k=1,2, \cdots, N.
\end{equation}
\end{lem}
 
\begin{proof}
Fix $u \in \mathcal{U}_N$ and $p=\Phi_N(u)$, for every $h \in \mathcal{T}_u (\mathcal{U}_N)$, we have $\mathrm{d}\Phi_N (u) (h)  \in \mathcal{T}_{p} (\Omega_N \times \mathbb{R}^N)$. For every smooth function $f : \mathbf{p}=(r^1, r^2, \cdots, r^N; \alpha^1,  \alpha^2, \cdots, \alpha^N) \in \Omega_N \times \mathbb{R}^N \mapsto f( \mathbf{p}) \in \mathbb{R}$, then
\begin{equation}\label{formula of d phi N h f} 
\left(\mathrm{d}\Phi_N (u) (h) \right)f  =\mathrm{d}(f \circ \Phi_N)(u) (h) = \sum_{j=1}^N \left( \mathrm{d} I_j(u)(h) \frac{\partial f}{\partial r^j}\Big|_p  +  \mathrm{d} \gamma_j(u)(h)\frac{\partial f}{\partial \alpha^j}\Big|_p\right).
\end{equation}For every $k=1,2, \cdots, N$, we replace $h$ by $X_{I_k}(u)$, where $X_{I_k}$ denotes the Hamiltonian vector field  of the k th action $I_k$ defined in $(\ref{j th action and angle})$, thus the Poisson bracket formulas $(\ref{Poisson brackets of gamma lambda})$ yield that
\begin{equation*}
\frac{\partial f}{ \partial \alpha^k} \Big|_p = \sum_{j=1}^N \left( \{I_k,I_j\}(u) \frac{\partial f}{\partial r^j}\Big|_p  +  \{I_k,\gamma_j\}(u) \frac{\partial f}{\partial \alpha^j}\Big|_p\right) = \left(\mathrm{d}\Phi_N (u) (X_{I_k}(u)) \right)f.
\end{equation*}
\end{proof}

\begin{lem}
For every $1 \leq j<k \leq N$, there exists a smooth function $c_{jk} \in C^{\infty}(\Omega_N \times \mathbb{R}^N)$ such that  
\begin{equation}\label{nu tilde depend only on the actions}
\tilde{\nu} = \sum_{1 \leq j <k \leq N}  c_{jk} \mathrm{d}r^j \wedge \mathrm{d}r^k, \qquad   \frac{\partial c_{jk}}{\partial  \alpha^l}\Big|_p   = 0, \quad \forall j,k, l=1,2,\cdots, N,
\end{equation}for every $p =(r^1, r^2, \cdots, r^N; \alpha^1,  \alpha^2, \cdots, \alpha^N)\in \Omega_N \times \mathbb{R}^N$.
\end{lem}

\begin{proof}
The proof is divided into three steps. The first step is to prove that for every $p \in \Omega_N \times \mathbb{R}^N $ and every $V \in \mathcal{T}_p(\Omega_N \times \mathbb{R}^N)$,  
\begin{equation}\label{first step to simplify tilde nu}
\tilde{\nu}_p (\frac{\partial  }{\partial \alpha^l}\Big|_p, V)= 0, \qquad \forall l=1,2, \cdots, N.
\end{equation}In fact, let $u=\Psi_N (p) \in \mathcal{U}_N$ and $p=(r^1, r^2, \cdots, r^N; \alpha^1,  \alpha^2, \cdots, \alpha^N)$, so $r^l=r^l(p) = I_l \circ \Psi_N(p)$. Then
\begin{equation*}
(\Psi_N^* \omega)_p(\frac{\partial  }{\partial \alpha^l}\Big|_p, V) = \omega_u (\mathrm{d}\Psi_N(p) \left(\frac{\partial  }{\partial \alpha^l}\Big|_p \right),\mathrm{d}\Psi_N(p) (V) ) = \omega_u (X_{I_l}(u), \mathrm{d}\Psi_N(p) (V) )
\end{equation*}by $(\ref{d Phi N send X I k to coordinate vector gamma k})$. Thus $(\Psi_N^* \omega)_p(\frac{\partial  }{\partial \alpha^l}\Big|_p, V)  =-\mathrm{d}I_l(u)(\mathrm{d}\Psi_N(p) (V)) = -\mathrm{d}(I_l\circ \Psi_N)(p)(V)$. On the other hand,
\begin{equation*}
\nu_p(\frac{\partial  }{\partial \alpha^l}\Big|_p, V) = \sum_{j=1}^N ( \mathrm{d}r^j \wedge \mathrm{d}\alpha^j)\left(\frac{\partial  }{\partial \alpha^l}\Big|_p, V \right) = - \mathrm{d}r^l(p)(V)
\end{equation*}Thus $(\ref{first step to simplify tilde nu})$ is obtained by $\tilde{\nu} = \Psi_N^* \omega -  \nu$. \\

\noindent Since $\tilde{\nu}$ is a smooth $2$-form on $\Omega_N\times \mathbb{R}^N$, we have
\begin{equation*}
\tilde{\nu} = \sum_{1 \leq j <k \leq N} (a_{jk}\mathrm{d}\alpha^j \wedge \mathrm{d}\alpha^k + b_{jk}\mathrm{d}r^j \wedge \mathrm{d}\alpha^k +  c_{jk}\mathrm{d}r^j \wedge \mathrm{d}r^k) ,
\end{equation*}for some smooth functions $a_{jk}, b_{jk}, c_{jk}\in C^{\infty}(\Omega_N \times \mathbb{R}^N)$,  $1\leq j <k \leq N$. The second step is to prove that $a_{jk}=b_{jk} = 0$ on $\Omega_N \times \mathbb{R}^N$, for every $1\leq j <k \leq N$. In fact, we have $\mathrm{d}r^j \wedge \mathrm{d}r^k (\frac{\partial  }{\partial \alpha^l}\Big|_p, V)= 0$,
\begin{equation*}
 \mathrm{d}r^j \wedge \mathrm{d}\alpha^k (\frac{\partial  }{\partial \alpha^l}\Big|_p, V) = -\mathbf{1}_{k=l} \mathrm{d}r^j(p)(V)  \quad \mathrm{and} \quad \mathrm{d}\alpha^j \wedge \mathrm{d}\alpha^k (\frac{\partial  }{\partial \alpha^l}\Big|_p, V) = \mathbf{1}_{j=l}\mathrm{d}\alpha^k(p)(V) - \mathbf{1}_{k=l}\mathrm{d}\alpha^j(p)(V).
\end{equation*}Then, let $l\in \{2, \cdots, N\}$ be fixed, for every $1\leq j <k \leq N$, we have
\begin{equation}\label{formula  replace V to obtain a jk and bjk =0}
\sum_{1 \leq l <k \leq N}a_{lk}\mathrm{d}\alpha^k(p)(V) - \sum_{1 \leq j <l \leq N}(a_{jl}\mathrm{d}\alpha^j(p)(V)+b_{jl}\mathrm{d}r^j(p)(V))=\tilde{\nu}_p(\frac{\partial  }{\partial \alpha^l}\Big|_p, V)=0.
\end{equation}Then we replace $V$ by $\frac{\partial  }{\partial r^j}\Big|_p$ and $\frac{\partial  }{\partial \alpha^j}\Big|_p$ respectively in formula $(\ref{formula  replace V to obtain a jk and bjk =0})$,  then $a_{jl}=b_{jl}=0$, for every $j = 1,2,\cdots, l-1$.\\

\noindent It remains to show that $c_{jk}$ depends on $r^1, r^2, \cdots, r^N$, for every $1\leq j <k \leq N$. The symplectic form $\omega$ is closed by proposition $\ref{prop omega is closed symplectic manifold UN}$ and $\nu = \mathrm{d} \kappa$ is exact, where $\kappa = \sum_{j=1}^N r^j \mathrm{d}\alpha^j$. So 
\begin{equation*}
\mathrm{d} \tilde{\nu} =\mathrm{d} ( \Psi_N^*\omega)- \mathrm{d}\nu  = \Psi_N^* (\mathrm{d} \omega) =0.
\end{equation*}The exterior derivative of $\tilde{\nu} = \sum_{1 \leq j <k \leq N}   c_{jk}\mathrm{d}r^j \wedge \mathrm{d}r^k$ is computed as following
\begin{equation*}
0= \sum_{1 \leq j <k \leq N} \sum_{l =1}^N \left( \frac{\partial c_{jk}  }{\partial \alpha^l} \mathrm{d}\alpha^l \wedge \mathrm{d}r^j \wedge \mathrm{d}r^k + \frac{\partial c_{jk}  }{\partial r^l} \mathrm{d}r^l \wedge \mathrm{d}r^j \wedge \mathrm{d}r^k \right).
\end{equation*}Since the family $\{\mathrm{d}r^j \wedge \mathrm{d}r^k \wedge \mathrm{d}\alpha^l\}_{1\leq j<k  \leq N, 1\leq l \leq N} \bigcup \{ \mathrm{d}r^j \wedge \mathrm{d}r^k \wedge \mathrm{d}r^l \}_{1\leq j<k<l \leq N}$ is linearly independent in $\boldsymbol{\Omega}^3(\mathcal{U}_N)$, we have $ \frac{\partial c_{jk}  }{\partial \alpha^l} = 0$, for every $1\leq j<k \leq N$ and $l=1,2, \cdots, N$.   
\end{proof}
 
 \bigskip
 
\noindent Since the $2$-form $\tilde{\nu}$ is independent of $\alpha^1, \alpha^2, \cdots, \alpha^N$, it suffices to consider points $p=(r,\alpha) \in \Omega_N \times \mathbb{R}^N$ with $\alpha=0$. We shall prove that $\tilde{\nu}=0$ by introducing the following Lagrangian submanifold of $\Omega_N \times \mathbb{R}^N$,
\begin{equation*}
\Omega_N\times \{0_{\mathbb{R}^N}\}= \{(r^1, r^2, \cdots, r^N; 0,0, \cdots, 0)\in  \mathbb{R}^{2N} : r^1 < r^2< \cdots< r^N <0\}.
\end{equation*}

\begin{proof}[End of the proof of formula $(\ref{Goal nu=0})$]
The submersion level set theorem implies that $\Omega_N\times \{0_{\mathbb{R}^N}\}$ is a properly embedded $N$-dimensional submanifold of $\Omega_N\times \mathbb{R}^N$. We have $\Omega_N\times \{0_{\mathbb{R}^N}\}= \Phi_N(\Lambda_N)$, where $\Lambda_N$ is the Lagrangian submanifold of $(\mathcal{U}_N, \omega)$ defined by $(\ref{Lambda N definition})$. For every $q \in \Omega_N\times \{0_{\mathbb{R}^N}\}$, set $v = \Psi_N(q) \in \Lambda_N$, we claim at first that
\begin{equation}\label{identification of T q and d Phi N T v}
\mathcal{T}_q(\Omega_N\times \{0_{\mathbb{R}^N}\}) = \bigoplus_{j=1}^N \mathbb{R}\frac{\partial  }{\partial r^j}\Big|_q = \mathrm{d} \Phi_N(v) (\mathcal{T}_v(\Lambda_N )).
\end{equation}In fact, every tangent vector $V \in \mathcal{T}_q(\Omega_N\times \{0_{\mathbb{R}^N}\})$ is the velocity at $t=0$ of some smooth curve $\xi : t \in (-1,1) \mapsto \xi(t)=(\xi_1(t), \xi_2(t), \cdots, \xi_N(t); 0,0, \cdots, 0) \in \Omega_N\times \{0_{\mathbb{R}^N}\}$ such that $\xi(0)=q$, i.e.
\begin{equation}\label{Vf formula}
V f= \frac{\mathrm{d}}{\mathrm{d}t}\Big|_{t=0} (f\circ\xi) = \sum_{j=1}^N \xi_j'(0)\frac{\partial f}{\partial r^j}\Big|_q, \qquad \forall f \in C^{\infty}(\Omega_N\times  \mathbb{R}^N).
\end{equation}So the first equality of $(\ref{identification of T q and d Phi N T v})$ is obtained. Then we set $\eta(t)=\Psi_N\circ \xi(t)$, $\forall t \in (-1,1)$. For every $g \in C^{\infty}(\mathcal{U}_N)$, we replace $f$ by $g\circ \Psi_N \in C^{\infty}(\Omega_N\times  \mathbb{R}^N)$ in $(\ref{Vf formula})$ to obtain that
\begin{equation*}
\mathrm{d}\Psi_N(q)(V) g  = V(g\circ \Psi_N ) =  \frac{\mathrm{d}}{\mathrm{d}t}\Big|_{t=0} (g\circ \eta)= \eta'(0) g.
\end{equation*}Since $\eta$ is a smooth curve in the Lagrangian section $\Lambda_N$ such that $\eta(0)=v$, we have $\mathrm{d}\Psi_N(q)(V) =  \eta'(0) \in \mathcal{T}_v (\Lambda_N)$. So formula $(\ref{identification of T q and d Phi N T v})$ holds. Since $\nu= \sum_{j=1}^N \mathrm{d} r^j \wedge \mathrm{d}\alpha^j$, the submanifold $\Omega_N \times \{0_{\mathbb{R}^N}\}$ is Lagrangian. \\

\noindent For every $p =(r^1,r^2, \cdots, r^N;\alpha^1, \alpha^2, \cdots, \alpha^N) \in \Omega_N \times \mathbb{R}^N$ and every $V_1, V_2 \in \mathcal{T}_p(\Omega_N \times \mathbb{R}^N)$, where
\begin{equation*}
V_m = \sum_{j=1}^N \left(a_j^{(m)}\frac{\partial }{\partial r^j}\Big|_p + b_j^{(m)}\frac{\partial }{\partial \alpha^j}\Big|_p \right), \qquad a_j^{(m)}, b_j^{(m)} \in \mathbb{R}, \quad m=1,2,
\end{equation*}we choose $q=(r^1,r^2, \cdots, r^N;0, 0, \cdots, 0) \in \Omega_N \times \{0_{\mathbb{R}^N}\}$ and $W_1, W_2 \in  \mathcal{T}_q(\Omega_N \times \{0_{\mathbb{R}^N}\})$, where
\begin{equation*}
W_m = \sum_{j=1}^N a_j^{(m)}\frac{\partial }{\partial r^j}\Big|_p,  \qquad m=1,2.
\end{equation*}We set $v = \Psi_N(q) \in \Lambda_N$. We have proved that $c_{jk}(p)=c_{jk}(q)$, then $(\ref{nu tilde depend only on the actions})$ yields that
\begin{equation*}
\tilde{\nu}_p(V_1,V_2)= \sum_{1\leq j<k\leq N}a_j^{(1)}a_k^{(2)} c_{jk} (p) = \tilde{\nu}_q(W_1,W_2) =   \omega_v(\mathrm{d}\Psi_N(v)( W_1), \mathrm{d}\Psi_N(v)( W_2)) ,
\end{equation*}because $\nu_q(W_1, W_2)=0$. The identification $(\ref{identification of T q and d Phi N T v})$ yields that $h_m:= \mathrm{d}\Psi_N (v)(W_m) \in \mathcal{T}_v(\Lambda_N)$, for $m=1,2$. Consequently, we have $\tilde{\nu}_p(V_1,V_2) = \omega_v(h_1, h_2)=0$.
\end{proof}
\noindent Formula $(\ref{Goal nu=0})$ is equivalent to $\Phi_N^* \nu = \omega$, so $\Phi_N : (\mathcal{U}_N, \omega) \to (\Omega_N \times \mathbb{R}^N, \nu)$ is a symplectomorphism. Finally, we recall a basic property  in symplectic geometry: the three formulas in $(\ref{Poisson bracket all})$ are equivalent to the symplectomorphism property of $\Phi_N$. 

\begin{prop}\label{Property symplectomorphism property Poisson bracket}
If $\tilde{\Phi}_N : (\mathcal{U}_N, \omega) \to (\Omega_N \times \mathbb{R}^N, \nu)$ is a diffeomorphism,
\begin{equation*}
\tilde{\Phi}_N(u)= (\tilde{I}_1(u), \tilde{I}_2(u), \cdots, \tilde{I}_N(u); \tilde{\gamma}_1(u), \tilde{\gamma}_2(u), \cdots, \tilde{\gamma}_N(u)), \qquad \forall u \in \mathcal{U}_N,
\end{equation*}for some smooth functions $\tilde{I}_j, \tilde{\gamma}_j$ on $\mathcal{U}_N$, then each of the following three properties implies the others:\\

\noindent $(\mathrm{a})$. $\tilde{\Phi}_N : (\mathcal{U}_N, \omega) \to (\Omega_N \times \mathbb{R}^N, \nu)$ is a symplectomorphism, i.e. $\tilde{\Phi}_N^* \nu   =\omega$.\\
\noindent $(\mathrm{b})$. For every $j,k=1,2, \cdots, N$, we have $\{\tilde{I}_j, \tilde{I}_k\}=\{\tilde{\gamma}_j, \tilde{\gamma}_k\} =0$ and $\{\tilde{I}_j, \tilde{\gamma}_k\}=\mathbf{1}_{j=k}$ on $\mathcal{U}_N$. \\
\noindent $(\mathrm{c})$. For every $k=1,2, \cdots, N$, we have  
\begin{equation*}
\mathrm{d}\tilde{\Phi}_N(u) ( X_{\tilde{I}_k}(u)) = \frac{\partial}{\partial \alpha^k}\Big|_{\tilde{\Phi}_N(u) }, \qquad \mathrm{d}\tilde{\Phi}_N(u) ( X_{\tilde{\gamma}_k}(u)) = -\frac{\partial}{\partial r^k}\Big|_{\tilde{\Phi}_N(u) }, \qquad \forall u \in \mathcal{U}_N.
\end{equation*}
\end{prop}

\begin{proof} 
$(\mathrm{a}) \Rightarrow (\mathrm{b})$. For any smooth function $f : \Omega_N \times \mathbb{R}^N \to \mathbb{R}$, its Hamiltonian vector field is given by
\begin{equation}\label{Hamiltonian vector field of functions on OmegaN}
X_f(p)= \sum_{j=1}^N \frac{\partial f}{\partial r^j}(p) \frac{\partial}{\partial \alpha^j}\Big|_p -  \frac{\partial f }{\partial \alpha^j}(p)\frac{\partial}{\partial r^j}\Big|_p, \qquad \forall p \in \Omega_N \times \mathbb{R}^N.
\end{equation}If $\tilde{\Phi}_N^* \nu   =\omega$, then $X_{f \circ \tilde{\Phi}_N}(u)=\mathrm{d}\tilde{\Psi}_N  (p) \circ X_f(p)$, if $p = \tilde{\Phi}_N(u)$, where $\tilde{\Psi}_N = \tilde{\Phi}_N^{-1}$. The Poisson bracket of two smooth functions $f, g$ on $\Omega_N \times \mathbb{R}^N$ is given by
\begin{equation}\label{Poisson bracket in Omega }
\{f, g\}^{\nu}(p)=(X_f g) \Big|_p =\nu_p(X_f(p), X_g(p))= \sum_{j=1}^N \frac{\partial f}{\partial r^j}(p) \frac{\partial g}{\partial \alpha^j}(p) -  \frac{\partial f }{\partial \alpha^j}(p)\frac{\partial g}{\partial r^j}(p).
\end{equation}Then $\{f\circ \tilde{\Phi}_N, g\circ \tilde{\Phi}_N\}=\{f, g\}^{\nu}\circ \tilde{\Phi}_N$ on $\mathcal{U}_N$. It suffices to choose $f, g \in \{\tilde{I}_j   \circ \tilde{\Psi}_N ,  \tilde{\gamma}_j \circ \tilde{\Psi}_N\}_{1\leq j\leq N}$. \\

\noindent $(\mathrm{b}) \Rightarrow (\mathrm{c})$. We do the same calculus as in lemma $\ref{lemma d Phi N send X I k to coordinate vector gamma k}$ to obtain that
\begin{equation}\label{tangent correspondant}
\mathrm{d}\tilde{\Phi}_N(u) ( X_{\tilde{I}_k}(u)) = \frac{\partial}{\partial \alpha^k}\Big|_{\tilde{\Phi}_N(u) }, \qquad \mathrm{d}\tilde{\Phi}_N(u) ( X_{\tilde{\gamma}_k}(u)) = -\frac{\partial}{\partial r^k}\Big|_{\tilde{\Phi}_N(u) }, \qquad \forall u \in \mathcal{U}_N.
\end{equation}

\noindent $(\mathrm{c}) \Rightarrow (\mathrm{a})$. Formula $(\ref{tangent correspondant})$ implies that $\{X_{\tilde{I}_1} , X_{\tilde{I}_2} , \cdots, X_{\tilde{I}_N} ; X_{\tilde{\gamma}_1} , X_{\tilde{\gamma}_2} , \cdots, X_{\tilde{\gamma}_N} \}$ forms a basis in $\mathfrak{X}(\mathcal{U}_N)$. Since the $2$-covectors $(\tilde{\Phi}_N^* \nu )_u $ and $ \omega_{u}$ coincide at every couple of elements of this basis, they are the same, so $ \tilde{\Phi}_N^* \nu  = \omega $.
\end{proof}

\bigskip
\bigskip

\appendix

\section{Appendices}\label{section appendix}
We establish several topological properties of the $N$-soliton manifold $\mathcal{U}_N$ without using the action--angle map $\Phi_N : \mathcal{U}_N \to \Omega_N \times \mathbb{R}^N$. The Vi\`ete map $\mathbf{V} : (\beta_1,  \beta_2, \cdots, \beta_N) \in \mathbb{C}^N \mapsto (a_0, a_1, \cdots, a_{N-1}) \in \mathbb{C}^N$ is defined by
\begin{equation}\label{definition of Vieta map appendix}
 \prod_{j=1}^N (X- \beta_j) = \sum_{k=0}^{N-1} a_k X^k + X^N.
\end{equation} 

\begin{prop}\label{prop Pi UN is a simply connected Kahler manifold}
Endowed with the Hermitian form $\mathfrak{H}$ introduced in $(\ref{definition of hermitian form H})$, $(\Pi (\mathcal{U}_N), \mathfrak{H})$ is a simply connected K\"ahler manifold which is biholomorphically equivalent to $\mathbf{V}(\mathbb{C}_-^N)$.
\end{prop}

\begin{prop}\label{Universal covering map UN soliton to UN gap potential}
The $N$-soliton manifold $\mathcal{U}_N$ is a universal covering manifold of the following $N$-gap potential manifold for the BO equation on the torus $\mathbb{T}:=\mathbb{R} \slash 2\pi \mathbb{Z}$ as described by G\'erard--Kappeler $[\ref{Gerard kappeler Benjamin ono birkhoff coordinates}]$,
\begin{equation}\label{defintion of f gap p UN}
U_N^{\mathbb{T}} = \{v =h +\overline{h}\in L^2(\mathbb{T}, \mathbb{R}): \quad h : y \in \mathbb{T} \mapsto -e^{iy} \frac{\mathfrak{Q}'(e^{iy})}{\mathfrak{Q} (e^{iy})} \in \mathbb{C}, \quad \mathfrak{Q} \in \mathbb{C}_N^+[X]\},
\end{equation}where $\mathbb{C}_N^+[X]$ consists of all monic polynomial $\mathfrak{Q}\in \mathbb{C}[X]$ of degree $N$, whose roots are contained in the annulus $\mathscr{A}:=  \{z \in \mathbb{C} : |z|>1\}$.  The fundamental group of $U_N^{\mathbb{T}}$ is $(\mathbb{Z},+)$.
\end{prop}

\begin{rem}
The real analytic symplectic manifold $U_N^{\mathbb{T}}$ is mapped real bi-analytically  onto $\mathbb{C}^{N-1}\times \mathbb{C}^*$ by the restriction of the Birkhoff map constructed in G\'erard--Kappeler $[\ref{Gerard kappeler Benjamin ono birkhoff coordinates}]$.  The union of all finite gap potentials $\bigcup_{N \geq 0}U^{\mathbb{T}}_N$ is dense in $L^2_{r,0}(\mathbb{T}) = \{v \in L^2(\mathbb{T}, \mathbb{R}): \int_{\mathbb{T}}v =0\}$. However $\bigcup_{N \geq 1}\mathcal{U}_N$ is not dense in $L^2(\mathbb{R}, (1+x^2)\mathrm{d}x)$.  We refer to Coifman--Wickerhauser $[\ref{Coifman Wickerhauser The scattering transform for the Benjamin Ono equation}]$ to see solutions with sufficiently small initial data and the case of non-existence of rapidly decreasing solitons. 
\end{rem}

\noindent The simple connectedness of $\mathcal{U}_N$ is proved in subsection $\mathbf{\ref{subsection simple connectedness}}$. Then  we establish a real analytic covering map  $\mathcal{U}_N \to U_N^{\mathbb{T}}$ in subsection $\mathbf{\ref{subsection covering manifold}}$.

\subsection{The simple connectedness of $\mathcal{U}_N$}\label{subsection simple connectedness}
Thanks to the biholomorphical equivalence between the K\"ahler manifolds $\Pi (\mathcal{U}_N)$ and $\mathbf{V}(\mathbb{C}^N_-)$ established in lemma $\ref{Lemma of complex analytic manifold Pi UN}$, it suffices to prove the simple connectedness of the subset $\mathbf{V}(\mathbb{C}^N_-)$, where $\mathbf{V}$ denotes the Vi\`ete map defined by $(\ref{definition of Vieta map appendix})$. Since every fiber of the Vi\`ete map is invariant under the permutation of components, we introduce the following group action.  Equipped with the discrete topology, the symmetric group $\mathrm{S}_N$ acts continuously on $\mathbb{C}^N$ by permuting the components of every vector:
\begin{equation}\label{action of SN on CN}
\boldsymbol{\sigma} : (\beta_0, \beta_1, \cdots, \beta_{N-1}) \in \mathbb{C}^N \mapsto (\beta_{\sigma(0)}, \beta_{\sigma(1)}, \cdots, \beta_{\sigma(N-1)})\in \mathbb{C}^N,  \qquad \forall \sigma \in \mathrm{S}_N.
\end{equation}A subset $A \subset \mathbb{C}^N$ is said to be $stable$ $under$ $\mathrm{S}_N$ if $\bigcup_{\sigma \in \mathrm{S}_N}\boldsymbol{\sigma}(A)=A$. We recall the  basic property of the Vi\`ete map $\mathbf{V}$ and the action of  symmetric group $\mathrm{S}_N$.
\begin{lem}\label{lemma of Viete map closed quotient map}
The Vi\`ete map $\mathbf{V} : \mathbb{C}^N \to \mathbb{C}^N$ is a both open and closed quotient map. For every $A \subset \mathbb{C}^N$, $A$ is stable under $\mathrm{S}_N$ if and only if $A$ is saturated with respect to $\mathbf{V}$, the quotient space $A\slash \mathrm{S}_N$ is homeomorphic to $\mathbf{V}(A)$. 
\end{lem}

\noindent We set $\Delta:=\{(\beta, \beta, \cdots, \beta) \in \mathbb{C}^N : \forall \beta\in \mathbb{C}\}$. The goal of this subsection is to prove the following result.
\begin{prop}\label{simply connected property for viete map}
For every open simply connected subset $A \subset \mathbb{C}^N$, if $A$ is stable under the symmetric group $\mathrm{S}_N$ and $A \bigcap \Delta \ne \emptyset$, then $\mathbf{V}(A)$ is an open simply connected subset of $\mathbb{C}^N$.
\end{prop}

\begin{proof} 
Let $A\subset \mathbb{C}^N$ be a nonempty open simply connected subset that is stable by $\mathrm{S}_N$. The subset $B:=\mathbf{V}(A)$ is open, connected and locally simply connected, then it admits a universal covering space $E$ and a covering map $\pi : E \to B$. The triple $(E, \pi, B)$ is identified as a fiber bundle over $B$ whose model fiber $\mathbf{F}$ is discrete. The target is to show that $\mathbf{F}$ has cardinality $1$. \\

\noindent Let $(\mathscr{P}, q, B)$ denote the fiber product (Husem\"oller $[\ref{Husemoller book on Fibre bundles}]$) of bundles $(A, \mathbf{V}, B)$ and $(E, \pi, B)$, defined by
\begin{equation}\label{definition of fiber product}
\mathscr{P} = A \times_B  E := \{(\beta, e)\in A \times E : \pi(e)=\mathbf{V}(\beta) \}, \qquad q : (\beta,e) \in \mathscr{P} \mapsto \mathbf{V}(\beta)=\pi(e) \in B.
\end{equation}The total space $\mathscr{P}$ is equipped with the subspace topology of the product space $ A \times E$ and projections onto the first factor and onto the second factor are denoted respectively by
\begin{equation}\label{definition of projections on 1st 2nd factors}
p : (\beta,e) \in \mathscr{P} \mapsto \beta \in A, \qquad \mathbf{W} : (\beta , e) \in \mathscr{P} \mapsto e \in E.
\end{equation}Both $p$ and $\mathbf{W}$ are continuous functions on $\mathscr{P}$ and the following diagram commutes.
\begin{center}
 \begin{tikzpicture} 
  \matrix (m) [matrix of math nodes,row sep=3em, column sep=5em,minimum width=2em]
  {
    \mathscr{P} & E \\
     A & B \\};
  \path[-stealth]
    (m-1-1) edge node [left] {$p$} (m-2-1)
            edge   node [right] {$q$}(m-2-2) 
            edge node [above] {$\mathbf{W}$} (m-1-2)
    (m-2-1.east|-m-2-2) edge node [below] {$\mathbf{V}$} (m-2-2)
    (m-1-2) edge node [right] {$\pi$} (m-2-2);
\end{tikzpicture}
\end{center}We claim two properties concerning the projections $p$ and $\mathbf{W}$.\\

\noindent $\romannumeral1$. $\mathbf{W} : \mathscr{P} \to E$ is an open quotient map and $p:\mathscr{P} \to A$ is a covering map whose model fiber is  $\mathbf{F}$.\\
\noindent $\romannumeral2$. Equipped with the discrete topology, the symmetric group $\mathrm{S}_N$ acts continuously on $\mathscr{P}$ by permuting components of the first factor
\begin{equation*}
\overline{\boldsymbol{\sigma}} : (\beta,e ) \in \mathscr{P} \mapsto (\boldsymbol{\sigma}(\beta),e ) \in \mathscr{P}, \qquad \forall \sigma \in \mathrm{S}_N,
\end{equation*}where $\boldsymbol{\sigma} \in \mathrm{GL}_N(\mathbb{C})$ is defined by $(\ref{action of SN on CN})$. Hence the quotient map $\mathbf{W} : \mathscr{P} \to E$ is closed.\\

\noindent Thanks to the simple connectedness of the base space $A$, the covering space $\mathscr{P}$ is the disjoint union of its connected components $(\mathscr{A}_k)_{k\in \mathbf{F}}$ and the restriction of the covering map $p|_{\mathscr{A}_k} : \mathscr{A}_k \to A$ is a homeomorphism. Since $\mathscr{P}$ is locally path-connected, every component $\mathscr{A}_k$ is both open and closed, then $\mathbf{W}|_{\mathscr{A}_k} : \mathscr{A}_k \to E$ an open closed quotient map. So is the lift $g_k:=   \mathbf{W}|_{\mathscr{A}_k} \circ (p|_{\mathscr{A}_k})^{-1}  :A \to E$. Note that $\pi \circ g_k = \mathbf{V}$ and $\mathrm{S}_N$ stabilizes every element of $\Delta$. We choose $\beta \in A \bigcap \Delta$ and $b:= \mathbf{V}(\beta)$. Since the fiber $\mathbf{V}^{-1}(b)=\{\beta\}$ is a singleton, so is the fiber $\pi^{-1}(b)$. Hence $|\mathbf{F}|=1$ and the universal covering map $p : E \to B$ is a homeomorphism. So $B$ is simply connected.  
\end{proof}
 
\begin{rem}\label{rem Delta complement is simply connected}
Let $F$ be a closed submanifold of a smooth connected manifold $M$ without boundary of finite dimension. If $\dim_{\mathbb{R}}M - \dim_{\mathbb{R}}F \geq 3$, then the inclusion map $i : M\backslash F \to M$   induces an isomorphism between the fundamental groups $i_* : \pi_1(M\backslash F, x) \to \pi_1(M,x)$, for every $x \in M\backslash F$ (see Th\'eor\`eme 2.3 in P.146 of Godbillon $[\ref{Godbillon topology algebrique}]$). Note that the closed submanifold $\Delta \subset \mathbb{C}^N$ has real dimension $2$. When $N\geq 3$, the condition $A \bigcap \Delta \ne \emptyset$ cannot be deduced by the other three conditions in the hypothesis of proposition $\ref{simply connected property for viete map}$: $A$ is open, simply connected and stable by $\mathrm{S}_N$. 
\end{rem}

\noindent As a consequence, $\mathbf{V}(\mathbb{C}_-^N)$ is  open and simply connected because $ \mathbb{C}_-^N$ is an open convex subset of $\mathbb{C}^N$ which is stable under the symmetric group $\mathrm{S}_N$ and $\Delta \bigcap \mathbb{C}_-^N =\{(z,z, \cdots, z) \in \mathbb{C}^N : \mathrm{Im}z <0\}$. Together with lemma $\ref{Lemma of complex analytic manifold Pi UN}$, we finish the proof of proposition $\ref{prop Pi UN is a simply connected Kahler manifold}$.

\bigskip

\subsection{Covering manifold}\label{subsection covering manifold}
The Szeg\H{o} projector on $L^2(\mathbb{T}, \mathbb{C})$ is given by  $\Pi^{\mathbb{T}} v(x) =   \sum_{n \geq 0}v_n e^{inx}$, for every $v  \in L^2(\mathbb{T}, \mathbb{C})$ such that $v(x)= \sum_{n\in \mathbb{Z}}v_n e^{inx}$ with $v_n = \frac{1}{2\pi}\int_{0}^{2\pi} v(x) e^{-inx}\mathrm{d}x$. Equipped with the subspace topology of $\Pi^{\mathbb{T}}(L^2(\mathbb{T}, \mathbb{C}))$ and the Hermitian form
\begin{equation*}
\mathfrak{H}^{\mathbb{T}}(v_1, v_2) =  \langle \mathrm{D}^{-1} \Pi^{\mathbb{T}} v_1, \Pi^{\mathbb{T}} v_2\rangle_{L^2(\mathbb{T})} = \frac{1}{2\pi}\int_{0}^{2\pi} \mathrm{D}^{-1} \Pi^{\mathbb{T}} v_1(x)  \overline{\Pi^{\mathbb{T}} v_2(x)} \mathrm{d}x,
\end{equation*}the subset $\Pi^{\mathbb{T}}(U_N^{\mathbb{T}})$ is a K\"ahler manifold, which is mapped biholomorphically onto $\mathbf{V}(\mathscr{A}^N)$ with $\mathscr{A} = \mathbb{C}\backslash \overline{D}(0,1) = \{z \in \mathbb{C} : |z|>1\}$ in G\'erard--Kappeler $[\ref{Gerard kappeler Benjamin ono birkhoff coordinates}]$.
\begin{prop}\label{holomorphic covering map V CN- to V AN}
There exists a  covering map $\boldsymbol{\pi} :\mathbf{V}(\mathbb{C}_-^N) \to  \mathbf{V}(\mathscr{A}^N)$.
\end{prop}

\begin{rem}
Consider the cubic Szeg\H{o} equation on the torus (see G\'erard--Grellier $[\ref{gerardgrellier1 ann ens}, \ref{gerardgrellier2 invention invariant tori}, \ref{Gerard-grellier explicit formula szego equation}, \ref{Gerard grellier book cubic szego equation and hankel operators}]$) 
\begin{equation}\label{cubic szego equation on the torus}
i\partial_t w^{\mathbb{T}} = \Pi^{\mathbb{T}} (| w^{\mathbb{T}}|^2 w^{\mathbb{T}}), \qquad (t,x) \in \mathbb{R} \times \mathbb{T},
\end{equation}and the cubic Szeg\H{o} equation on the line (see Pocovnicu $[\ref{pocovnicu Traveling waves for the cubic Szego eq},\ref{pocovnicu Explicit formula for the cubic Szego eq}]$), we set $\Pi^{\mathbb{R}}:=\Pi$ in $(\ref{Szego projector})$,
\begin{equation}\label{cubic szego equation on the line}
i\partial_t w^{\mathbb{R}} = \Pi^{\mathbb{R}} (| w^{\mathbb{R}}|^2 w^{\mathbb{R}}),  \qquad (t,x) \in \mathbb{R} \times \mathbb{R}.
\end{equation}The manifold of $N$-solitons for the cubic Szeg\H{o} equation on the line is not simply connected. Let $\mathcal{M}(N)^{\mathbb{R}}$ denote all rational functions of the form $w^{\mathbb{R}}:x \in \mathbb{R} \mapsto \frac{P(x)}{Q(x)} \in \mathbb{C}$ where $P\in \mathbb{C}_{\leq N-1}[X]$ and $Q \in \mathbb{C}_{N}[X]$ is a monic polynomial such that $Q^{-1}(0)\subset \mathbb{C}_-$ and $P, Q$ have no common factors. Then $\mathcal{M}(N)^{\mathbb{R}}$ is a K\"ahler manifold of complex dimension $2N$. So is the subset $\mathcal{M}(N)^{\mathbb{T}}$ consisting of all rational functions of the form $w^{\mathbb{T}} :x \in \mathbb{T} \mapsto \frac{P(e^{ix})}{Q(e^{ix})} \in \mathbb{C}$ where $P\in \mathbb{C}_{\leq N-1}[X]$ and $Q \in \mathbb{C}_{N}[X]$ is a monic polynomial such that $Q^{-1}(0)\subset \mathscr{A}$ and $P, Q$ have no common factors. Both of them have rank characterization of Hankel operators by Kronecker-type theorem (see Lemma $8.12$ in Chapter 1 of Peller $[\ref{Peller Hankel operators of class Sp}]$, p. $54$). So the manifold $\mathcal{M}(N)^{\mathbb{R}}$ (resp. $\mathcal{M}(N)^{\mathbb{T}}$) is invariant under the flow of equation $(\ref{cubic szego equation on the line})$ (resp. of equation  $(\ref{cubic szego equation on the torus})$) and the (generalized) action--angle coordinates of equation  $(\ref{cubic szego equation on the line})$ (resp. of equation $(\ref{cubic szego equation on the torus})$) are defined in some open dense subset of $\mathcal{M}(N)^{\mathbb{R}}$ (resp. of $\mathcal{M}(N)^{\mathbb{T}}$). Moreover, if $N\geq 2$ then $\mathcal{M}(N)^{\mathbb{R}}$ is  simply connected by proposition $\ref{simply connected property for viete map}$ and remark $\ref{rem Delta complement is simply connected}$. There exists a holomorphic covering map $\mathcal{M}(N)^{\mathbb{R}} \to \mathcal{M}(N)^{\mathbb{T}}$ by following the construction in proposition $\ref{holomorphic covering map V CN- to V AN}$. The manifold of $N$-solitons for the cubic Szeg\H{o} equation on the line is an open dense subset of $\mathcal{M}(N)^{\mathbb{R}}$
\end{rem}

\end{document}